\documentclass[10pt]{article}

\usepackage{fullpage}
\usepackage{setspace}
\usepackage{parskip}
\usepackage{titlesec}
\usepackage{placeins}
\usepackage{xcolor}
\usepackage{lineno}

\usepackage{graphicx}
\usepackage{subcaption}

\usepackage{amssymb}
\usepackage{amsmath,amsthm}
\usepackage{enumitem}

\PassOptionsToPackage{hyphens}{url}
\usepackage[colorlinks = true,
            linkcolor = blue,
            urlcolor  = blue,
            citecolor = blue,
            anchorcolor = blue]{hyperref}
\usepackage{etoolbox}

\usepackage{natbib}

\renewenvironment{abstract}
  {{\bfseries\noindent{\abstractname}\par\nobreak}\footnotesize}
  {\bigskip}

\titlespacing{\section}{0pt}{*3}{*1}
\titlespacing{\subsection}{0pt}{*2}{*0.5}
\titlespacing{\subsubsection}{0pt}{*1.5}{0pt}

\usepackage{authblk}

\usepackage{graphicx}
\usepackage[space]{grffile}
\usepackage{latexsym}
\usepackage{textcomp}
\usepackage{longtable}
\usepackage{tabulary}
\usepackage{booktabs,array,multirow}
\usepackage{amsfonts,amsmath,amssymb}
\providecommand\citet{\cite}
\providecommand\citep{\cite}

\newif\iflatexml\latexmlfalse
\AtBeginDocument{\DeclareGraphicsExtensions{.pdf,.PDF,.eps,.EPS,.png,.PNG,.tif,.TIF,.jpg,.JPG,.jpeg,.JPEG}}

\iflatexml
\documentclass[a4paper]{article}
\fi

\usepackage[a4paper,left=2.5cm,bottom=1.7cm,right=2.5cm,top=1.7cm]{geometry}

 \usepackage[english]{babel} 
 \usepackage[T1]{fontenc} 

\usepackage[utf8]{inputenc}

\usepackage{color}
\usepackage{graphicx}

 \usepackage{subcaption} 
\usepackage{amssymb}

\usepackage{amsmath,amsthm}
\usepackage{amsmath,amsthm}

\usepackage{graphicx}

 \usepackage{enumitem} 
\newtheorem{remark}{Remark}
\newtheorem{theorem}{Theorem}
\newtheorem{lemma}{Lemma}

\newtheorem{proposition}{Proposition}
\newtheorem{definition}{Definition}
\def\finpv{\hfill $\square$  \\ \newline }

\newcommand{\keywords}[1][]{\par\addvspace\baselineskip\noindent\enspace\ignorespaces#1}

\begin{document}

\title{Contrast function estimation for the drift parameter of ergodic jump diffusion process.}

\author[1]{Chiara Amorino}%
\author[1]{Arnaud Gloter}%
\affil[1]{Laboratoire de Mathématiques et Modélisation d'Evry, Université Paris-Saclay}%

\vspace{-1em}

  \date{\today}

\begingroup
\let\center\flushleft
\let\endcenter\endflushleft
\maketitle
\endgroup

\selectlanguage{english}
\begin{abstract}
In this paper we consider an ergodic diffusion process with jumps whose drift coefficient depends on an unknown parameter. We suppose that the process is discretely observed. We introduce an estimator based on a contrast function, which is efficient without requiring any conditions on the rate at which the step discretization goes to zero, and where we allow the observed process to have non summable jumps. This extends earlier results where the condition on the step discretization was needed
and where the process was supposed to have summable jumps.
In general situations, our contrast function is not explicit and one has to resort to some approximation. In the case of a finite jump activity, we propose explicit approximations of the contrast function, such that the efficient estimation of the drift parameter is
 feasible. This extends the results obtained by Kessler in the case of continuous processes. 
\end{abstract}%

\section{Introduction}
Diffusion processes with jumps have been widely used to describe the evolution of phenomenon arising in various fields.
 In finance, jump-processes were introduced to model the dynamic of asset prices \hyperref[csl:1]{(Merton, 1976)},\hyperref[csl:2]{(Kou,2002)}, exchange rates \hyperref[csl:3]{(Bates, 1996)}, or volatility processes \hyperref[csl:4]{(Barndorff-Nielsen \& Shephard, 2001)},\hyperref[csl:5]{(Eraker, Johannes, \& N, 2003)}. Utilization of jump-processes in neuroscience can be found for instance in \hyperref[csl:6]{(Ditlevsen \& Greenwood, 2013)}. 
 
 Practical applications of these models has lead to the recent development of many statistical methods. In this work, our aim is to estimate the drift parameter $\theta$ from a discrete sampling of the process $X^\theta$ solution to
\begin{equation*}
X_t^\theta=X_0^\theta+ \int_0^t b(\theta,X_s^\theta) ds + \int_0^t \sigma(X_s^\theta) dW_s + 
\int_0^t \int_{\mathbb{R}\backslash \left \{0 \right \}} \gamma(X_{s-}^\theta) z \tilde{\mu}(ds,dz),
\end{equation*}
where $W$ is a one dimensional Brownian motion and $\tilde{\mu}$ a compensated Poisson random measure, with a possible infinite jump activity. We assume that the process is sampled at the times
$(t^n_i)_{i=0,\dots,n}$ where the sampling step $\Delta_n:=\sup_{i=0,\dots,n-1} t^n_{i+1}-t^n_i$ goes to zero. Due to the presence of a Gaussian component, we know that it is impossible to estimate the drift parameter on a finite horizon of time. Thus, we assume that $t_n \to \infty$ and the ergodicity of the process $X^\theta$.

Generally, the main difficulty while considering statistical inference of discretely observed stochastic processes comes from the lack of explicit expression for the likelihood. Indeed, the transition density of a jump-diffusion process is usually unknown explicitly. Several methods have been developed to circumvent this difficulty. For instance, closed form expansions of the transition density of jump-diffusions is studied in \hyperref[csl:7]{(A\"it-Sahalia \& Yu, 2006)}, \hyperref[csl:8]{(Li \& Chen, 2016)}. In the context of high frequency observation, the asymptotic behaviour of estimating functions are studied in \hyperref[csl:9]{(Jakobsen \& S{\o}rensen, 2017)}, and conditions are given to ensure rate optimality and efficiency.  Another approach, fruitful in the case of high frequency observation, is to consider pseudo-likelihood method, for instance based on the high frequency approximation of the dynamic of the process by the one of the Euler scheme. This leads to explicit contrast functions with Gaussian structures (see e.g. \hyperref[csl:10]{(Shimizu \& Yoshida, 2006)},\hyperref[csl:11]{(Shimizu, 2006)},\hyperref[csl:12]{(Masuda, 2013)}).

The validity of the approximation by the Euler pseudo-likelihood is justified by the high frequency assumption of the observations, and actually proving that  the estimators are asymptotic normal usually necessitates some conditions on the rate at which $\Delta_n$ should tend to zero. For applications, it is important that the condition on $\Delta_n \to 0$ is less
 stringent as possible.
 
In the case of continuous processes, Florens-Zmirou \hyperref[Florens Zmirou]{(Florens-Zmirou, 1989)} proposes estimation of drift and diffusion parameters under the fast sampling assumption $n\Delta_n^2 \to 0$. Yoshida \hyperref[csl:14]{(Yoshida, 1992)} suggests a correction of the contrast function that yields to the condition $n \Delta_n^3 \to 0$.
In Kessler \hyperref[csl:15]{(Kessler, 1997)}, the author introduces 
an explicit modification of the Euler scheme contrast such that the associated estimators are asymptotically normal, under the condition $n\Delta_n^k \to 0$ where $k \ge 2$ is arbitrarily large. Hence, the result by Kessler allows for any arbitrarily slow polynomial decay to zero of the sampling step.

In the case of jump-diffusions, Shimizu \hyperref[csl:11]{(Shimizu, 2006)} proposes parametric estimation of drift, diffusion and jump coefficients.
The asymptotic normality of the estimators are obtained under some explicit conditions relating the sampling step and jump intensity of the process. These conditions on  $\Delta_n$ are more restrictive as the intensity of jumps near zero is high. In the situation where this jump intensity is finite, the conditions of \hyperref[csl:11]{(Shimizu,2006)} reduces to $n \Delta_n^2 \to 0$. In \hyperref[csl:16]{(Gloter, Loukianova, \& Mai, 2018)}, the condition on the sampling step is relaxed to $n \Delta_n^3 \to 0$, when one estimates the drift parameter only.

In this paper, we focus on the estimation of the drift parameter, and our aim is to weaken the conditions on the decay of the sampling step in way comparable to Kessler's work \hyperref[csl:15]{(Kessler, 1997)}, but in the framework of jump-diffusion processes.

One of the idea in Kessler's paper is to replace, in the Euler scheme contrast function, the contribution of the drift by the exact value of the first conditional moment 
$ \displaystyle
m^{(1)}_{\theta,t_i,t_{i+1}}(x)=E[X^\theta_{t_{i+1}}\mid X^\theta_{t_{i}}=x]$ or some explicit approximation with arbitrarily high order when $\Delta_n \to 0$. In presence of jumps, the contrasts functions in \hyperref[csl:10]{(Shimizu \& Yoshida, 2006)} (see also \hyperref[csl:11]{(Shimizu, 2006)}, \hyperref[csl:16]{(Gloter, Loukianova, \& Mai, 2018)}) resort to a filtering procedure in order to suppress the contribution of jumps and recover the continuous part of the process.
Based on those ideas, we introduce a contrast function (see Definition \ref{D:definition_contrast}), whose expression relies on the quantity $m_{\theta,t_i,t_{i+1}}(x)=\frac{E[X^\theta_{t_{i+1}} \varphi( (X^\theta_{t_{i+1}}-X^\theta_{t_{i}})/(t_{i+1}-t_i)^\beta) \mid X^\theta_{t_{i}}=x]}{E[\varphi( (X^\theta_{t_{i+1}}-X^\theta_{t_{i}})/(t_{i+1}-t_i)^\beta) \mid X^\theta_{t_{i}}=x]}$, where $\varphi$ is some compactly supported function and $\beta<1/2$. The function $\varphi$ is such that $\varphi( (X^\theta_{t_{i+1}}-X^\theta_{t_{i}})/(t_{i+1}-t_i)^\beta)$ vanishes when the 
increments of the data are too large compared to the typical increments of a continuous diffusion process, and thus can be used to filter the contribution of the jumps.

The main result of our paper is that the associated estimator converges at rate $\sqrt{t_n}$, with some explicit asymptotic variance and is efficient.
Comparing to earlier results (\hyperref[csl:10]{(Shimizu \& Yoshida, 2006)}, \hyperref[csl:11]{(Shimizu, 2006)}, \hyperref[csl:16]{(Gloter, Loukianova, \& Mai, 2018)}), the sampling step $(t_i^n)_{i=0,\dots,n}$ can be irregular, no condition is needed on the rate at which $\Delta_n \to 0$ and we have suppressed the assumption that
the jumps of the process are summable. Let us stress that when the jumps activity is so high that the jumps are not summable, we have to choose $\beta<1/3$ (see Assumption $A_\beta$).

Moreover, in the case where the intensity is finite and with the specific choice of $\varphi$ being an oscillating function, we prove that we can approximate our contrast function by a completely explicit one, exactly as in the paper by Kessler \hyperref[csl:15]{(Kessler, 1997)}.
This yields to an efficient estimator under the condition
$n \Delta_n^{k} \to 0$, where $k$ is related to the oscillating properties of the function $\varphi$. 
As $k$ can be chosen arbitrarily high, up to a proper choice of $\varphi$, our method allows to estimate efficiently the drift parameter, under the assumption that the sampling step tends to zero at some polynomial rate. \\
We also show numerically that, when the jump activity is finite, the estimator we deduce from the explicit approximation of the contrast function performs well, making the bias visibly reduced. \\
On the other side, considering the case of infinite jumps activity (taking in particular a tempered $\alpha$-stable jump process with $\alpha <1$), we implement our main results building an approximation of $m$ (see Theorem \ref{th: dL mtheta alpha <1} below) from which we deduce an approximation of the contrast that we minimize in order to get the estimator of the drift coefficient. The estimator we found is a corrected version of the estimator that would result from the choice of an Euler scheme approximation. We see numerically that our estimator is well-performed and that the correction term we give drastically reduces the bias, especially as $\alpha$ gets bigger.

The outline of the paper is the following. In Section \ref{S:Model} we present the assumptions on the process $X$. The Section \ref{S:Construction_and_main} contains the main results of the paper: in Section \ref{S:Construction} we define the contrast function while the consistency and asymptotic normality of the estimator are stated in Section \ref{S:Main}.
In Section \ref{S: Practical implementation} we explain how to use in practice the contrast function and so we deal with its approximations in Section \ref{S: Approximate contrast} while 
its explicit modification is presented in the case of finite jump activity in Section \ref{S:Explicit}. The Section \ref{S: Numerical} is devoted to numerical results and perspectives for practical applications. In Section \ref{S:Limit} we state limit theorems useful to study the asymptotic behavior of the contrast function. The proofs of the main statistical results are given in Section \ref{S:Proof_main}, while the proofs of the limit theorems and some technical results are presented in the Appendix.

\section{Model, assumptions}\label{S:Model}
Let $\Theta$ be a compact subset of $\mathbb{R}$ and $X^\theta$ a solution to
\begin{equation}
X_t^\theta= X_0^\theta + \int_0^t b(\theta, X_s^\theta)ds + \int_0^t a(X_s^\theta)dW_s + \int_0^t \int_{\mathbb{R} \backslash \left \{0 \right \} }
\gamma(X_{s^-}^\theta)z \tilde{\mu}(ds,dz), \quad t \in \mathbb{R}_+, 
\label{eq: model} 
\end{equation}
where $W=(W_t)_{t \ge 0}$ is a one dimensional Brownian motion, $\mu$ is a Poisson random measure associated to the L\'evy process $L=(L_t)_{t \ge 0}$, with $L_t:= \int_0^t \int_\mathbb{R} z \tilde{\mu} (ds, dz)$ and $\tilde{\mu}= \mu - \bar{\mu}$ is the compensated one, on $[0, \infty) \times \mathbb{R}$. We denote $(\Omega, \mathcal{F}, \mathbb{P})$ the probability space on which $W$ and $\mu$ are defined. \\
We suppose that the compensator has the following form: $\bar{\mu}(dt,dz): = F(dz) dt $, where conditions on the Levy measure $F$ will be given later. \\
The initial condition $X_0^\theta$, $W$ and $L$ are independent.

\subsection{Assumptions}
We suppose that the functions $b: \Theta \times \mathbb{R} \rightarrow \mathbb{R}$, $a : \mathbb{R} \rightarrow \mathbb{R}$ and $\gamma : \mathbb{R} \rightarrow \mathbb{R}$ satisfy the following assumptions: \\
\\
ASSUMPTION 1: \textit{The functions $a(x)$, $\gamma(x)$ and, for all $\theta \in \Theta$, $b(x, \theta)$ are globally Lipschitz. Moreover, the Lipschitz constant of $b$ is uniformly bounded on $\Theta$.} \\
\\
Under Assumption 1 the equation (\ref{eq: model}) admits a unique non-explosive c\`adl\`ag adapted solution possessing the strong Markov property, cf \hyperref[csl:17]{(Applebaum, 2009)} (Theorems 6.2.9. and 6.4.6.). \\
\\
ASSUMPTION 2: \textit{For all $\theta \in \Theta$ there exists a constant $t > 0$ such that $X_t^\theta$ admits a density $p_t^\theta(x,y)$ with respect to the Lebesgue measure on $\mathbb{R}$; bounded in $y \in \mathbb{R}$ and in $x \in K$ for every compact $K \subset \mathbb{R}$. Moreover, for every $x \in \mathbb{R}$ and every open ball $U \in \mathbb{R}$, there exists a point $z = z(x, U) \in supp(F) $ such that $\gamma(x)z \in U$.} \\
\\
The last assumption was used in \hyperref[18 GLM]{(Masuda, 2007)} to prove the irreducibility of the process $X^\theta$. Other sets of conditions, sufficient for irreducibility, are in \hyperref[18 GLM]{(Masuda, 2007)}. \\
\\
ASSUMPTION 3 (Ergodicity): \textit{
\begin{enumerate}
\item For all $q >0$, $\int_{|z|> 1} |z|^q F(z) dz < \infty$.
\item For all $\theta \in \Theta$ there exists $C > 0$ such that $xb(x,\theta)\le -C|x|^2$, if $|x| \rightarrow \infty$.
\item $|\gamma(x)| / |x| \rightarrow 0$ as $|x|\rightarrow \infty$.
\item $|a(x)| / |x| \rightarrow 0$ as $|x|\rightarrow \infty$.
\item $\forall \theta \in \Theta$, $\forall q > 0$ we have $\mathbb{E}|X_0^\theta|^q < \infty$.
\end{enumerate}} 
Assumption 2 ensures, together with the Assumption 3, the existence of unique invariant distribution $\pi^\theta$, as well as the ergodicity of the process $X^\theta$, as stated in the Lemma \ref{lemma 2.1 GLM} below. \\
\\
ASSUMPTION 4 (Jumps): \textit{
\begin{enumerate}
\item The jump coefficient $\gamma$ is bounded from below, that is $\inf_{x \in \mathbb{R}}|\gamma(x)|:= \gamma_{min} >0$. 
\label{it:1}
\item The L\'evy measure $F$ is absolutely continuous with respect to the Lebesgue measure and we denote $F(z) = \frac{F(dz)}{dz}$.
\label{it:3}
\item We suppose that $ \exists \, c> 0$ s.t., for all $z \in \mathbb{R}$, $F(z) \le \frac{c}{|z|^{1 + \alpha}}$, with $\alpha \in (0,2)$.
\label{it:2}
\label{it: 4}
\end{enumerate}
} 
Assumptions 4.1 is useful to compare size of jumps of $X$ and $L$.  \\
\\
ASSUMPTION 5 (Non-degeneracy): \textit{There exists some $\alpha > 0$, such that $a^2(x) \ge  \alpha$ for all $x \in \mathbb{R}$} \\
\\
The Assumption $5$ ensures the existence of the contrast function defined in Section $3.1$. \\
\\
ASSUMPTION 6 (Identifiability): \textit{For all $\theta \ne \theta'$,$(\theta, \theta') \in \Theta^2$,
$$\int_\mathbb{R} \frac{(b(\theta, x) - b(\theta', x))^2}{a^2(x)} d\pi^\theta (x) > 0 $$} \\ \\
We can see that this last assumption is equivalent to
\begin{equation}
\forall \theta \ne \theta', \qquad (\theta, \theta') \in \Theta^2, \qquad b(\theta,.) \ne b(\theta',.).
\label{eq: 2.1 GLM} 
\end{equation}
We also need the following technical assumption: \\
\\
ASSUMPTION 7:

\begin{enumerate}
\item The derivatives $\frac{\partial^{k_1 + k_2} b}{ \partial x^{k_1} \partial \theta^{k_2}}$, with $k_1 +k_2 \le 4$ and $k_2 \le 3$, exist and they are bounded if $k_1 \ge 1$. If $k_1 = 0$, for each $k_2 \le 3$ they have polynomial growth.
\item The derivatives $a^{(k)}(x)$ exist and they are bounded for each $1 \le k \le 4$.
\item The derivatives $\gamma^{(k)}(x)$ exist and they are bounded for each $ 1 \le k \le 4$. \\

\end{enumerate}

Define the asymptotic Fisher information by
\begin{equation}
I(\theta)= \int_\mathbb{R} \frac{(\dot{b}(\theta,x))^2}{a^2(x)} \pi^\theta (dx).
\label{eq: 2.2 GLM, Fisher} 
\end{equation}
ASSUMPTION 8: \textit{For all $\theta \in \Theta$, $I(\theta) >0$} . \\ \\

\begin{remark}
 \label{R:alpha_smaller_1} 
If $\alpha <1$, using Assumption 4.3 the stochastic differential equation (\ref{eq: model}) can be rewritten as follows:
\begin{equation}
X_t^\theta= X_0^\theta + \int_0^t \bar{b}(\theta, X_s^\theta)ds + \int_0^t a(X_s^\theta)dW_s + \int_0^t \int_{\mathbb{R} \backslash \left \{0 \right \}} \gamma(X_{s^-}^\theta)z \mu(ds,dz), \quad t \in \mathbb{R}_+, 
\label{eq: model no-compensated measure} 
\end{equation}
where $\bar{b}(\theta, X_s^\theta) = b(\theta, X_s^\theta) - \int_{\mathbb{R} \backslash \left \{0 \right \}} \gamma(X_{s^-}^\theta)z F(z) dz$. \\
This expression implies that $X$ follows diffusion equation $X_t^\theta= X_0^\theta + \int_0^t \bar{b}(\theta, X_s^\theta)ds + \int_0^t a(X_s^\theta)dW_s$ in the interval in which no jump occurred.

\end{remark}
\medskip
From now on we denote the true parameter value by $\theta_0$, an interior point of the parameter space $\Theta$ that we want to estimate. We shorten $X$ for $X^{\theta_0}$. \\
We will use some moment inequalities for jump diffusions, gathered in the following lemma:

\begin{lemma}
Let $X$ satisfies Assumptions 1-4. Let $L_t:= \int_0^t \int_\mathbb{R} z \tilde{\mu}(ds, dz)$ and let $\mathcal{F}_s := \sigma \left \{ (W_u)_{0 < u \le s}, (L_u)_{0 < u \le s}, X_0 \right \}$. \\
Then, for all $t > s$, \\
1) for all $p \ge 2$, $\mathbb{E}[|X_t - X_s|^p]^\frac{1}{p} \le c |t-s|^\frac{1}{p}$, \\
2) for all $p \ge 2$, $p \in \mathbb{N}$, $\mathbb{E}[|X_t - X_s|^p|\mathcal{F}_s] \le c|t-s|(1 + |X_s|^p)$. \\
3) for all $p \ge 2$, $p \in \mathbb{N}$, $\sup_{h \in [0,1]} \mathbb{E}[|X_{s+h}|^p|\mathcal{F}_s] \le c(1 + |X_s|^p)$.
\label{lemma: Moment inequalities} 

\end{lemma}
The first two points follow from Theorem 66 of \hyperref[Protter GLM]{(Protter, 2005)} and Proposition 3.1 in \hyperref[csl:10]{(Shimizu \& Yoshida, 2006)}. The last point is a consequence of the second one: $\forall h \in [0,1]$,
$$\mathbb{E}[|X_{s+h}|^p|\mathcal{F}_s] = \mathbb{E}[|X_{s+h} - X_s + X_s|^p|\mathcal{F}_s] \le c(\mathbb{E}[|X_{s+h} - X_s|^p|\mathcal{F}_s] + \mathbb{E}[|X_s|^p|\mathcal{F}_s]),$$
where $c$ may change value line to line. Using the second point of Lemma \ref{lemma: Moment inequalities} and the measurability of $X_s$ with respect to $\mathcal{F}_s$, it is upper bounded by
$ c|h|(1 + |X_s|^p) + c |X_s|^p$. Therefore
$$\sup_{h \in [0,1]} \mathbb{E}[|X_{s+h}|^p|\mathcal{F}_s] \le \sup_{h \in [0,1]} c|h|(1 + |X_s|^p) + c |X_s|^p \le c(1 + |X_s|^p). $$

\subsection{Ergodic properties of solutions}
An important role is playing by ergodic properties of solution of equation (\ref{eq: model}) \\
The following Lemma states that Assumptions $1 - 4$ are sufficient for the existence of an invariant measure $\pi^\theta$ such that an ergodic theorem holds and moments of all order exist.

\begin{lemma}
Under assumptions 1 to 4,
for all $\theta \in \Theta$, $X^\theta$ admits a unique invariant distribution $\pi^\theta$ and the ergodic theorem holds:

\begin{enumerate}
\item For every measurable function $g: \mathbb{R} \rightarrow \mathbb{R}$ satisfying $\pi^\theta(g) < \infty$, we have a.s.
$$\lim_{t \rightarrow \infty} \frac{1}{t} \int_0^t g(X_s^\theta) ds = \pi^\theta (g). $$
\item For all $q > 0$, $\pi^\theta(|x|^q) < \infty $.
\item For all $q >0$, $\sup_{t \ge 0} \mathbb{E}[|X_t^\theta|^q] < \infty$.

\end{enumerate}
\label{lemma 2.1 GLM} 

\end{lemma}
A proof is in \hyperref[csl:16]{(Gloter, Loukianova, \& Mai, 2018)} (Section 8 of Supplement) in the case $\alpha \in (0,1)$, the proof relies on \hyperref[18 GLM]{(Masuda, 2007)}. In order to use it also in the case $\alpha \ge 1$ we have to show that, taken $q > 2$ $q$ even and $f^\star (x) = |x|^q $, $f^\star$ satisfies the drift condition $A f^\star = A_df^\star + A_c f^\star \le -c_1 f^\star + c_2 $, where $c_1 >0$ and $c_2 >0$. \\
Using Taylor's formula up to second order we have 
$$|A_d f^\star(x)| \le c \int_{\mathbb{R}} \int_0^1 |z|^2 \left \| \gamma \right \|_\infty|f''^\star(x + s z \gamma(y))| F(z) ds dz = $$
\begin{equation}
= c \int_{\mathbb{R}} \int_0^1 |z|^2 \left \| \gamma \right \|_\infty q(q-1)|x + s z \gamma(y)|^{q - 2} F(z) ds dz  =  o(|x|^q).
\label{eq: Adfstar} 
\end{equation}
Concerning the generator's continuous part, we use the second point of Assumption 3 to get
\begin{equation}
A_c f^\star (x) = \frac{1}{2} \sigma^2(x) q (q- 1)x^{q - 2} + b(\theta,x)q \, x\, x^{q - 2} \le o(|x|^q) - c q|x|^2x^{q - 2} \le o(|x|^q) - c f^\star(x). 
\label{eq: Acfstar} 
\end{equation}
By \eqref{eq: Adfstar} and \eqref{eq: Acfstar}, the drift condition holds.

\section{Construction of the estimator and main results}\label{S:Construction_and_main}
We exhibit a contrast function for the estimation of a parameter in the drift coefficient. We prove that the derived estimator is consistent and asymptotically normal.

\subsection{Construction of the estimator}\label{S:Construction}
Let $X^\theta$ be the solution to (\ref{eq: model}). Suppose that we observe a finite sample 
$$X_{t_0},  ... , X_{t_n}; \qquad 0=t_0 \le t_1 \le ... \le t_n,$$
where $X$ is the solution to \eqref{eq: model} with $\theta = \theta_0$.  
Every observation time point depends also on $n$, but to simplify the notation we suppress this index. We will be working in a high-frequency setting, i.e.
$$\Delta_n := \sup_{i =0, ... , n-1} \Delta_{n,i}\longrightarrow 0, \quad n \rightarrow \infty,$$
with $\Delta_{n,i}: = (t_{i+1} - t_i) $. \\
We assume $\lim_{n \rightarrow \infty} t_n = \infty$ and $n\Delta_n = O(t_n)$ as $n \rightarrow \infty$. \\
We introduce a jump filtered version of the gaussian quasi-likelihood. This leads to the following contrast function:
\begin{definition}
 \label{D:definition_contrast} 
For $\beta \in (0, \frac{1}{2})$ and $k >0$, we define the contrast function $U_n(\theta)$ as follows:
\begin{equation}
U_n(\theta):= \sum_{i =0}^{n -1} \frac{(X_{t_{i+1}} - m_{\theta, t_i,t_{i+ 1}}(X_{t_i}))^2}{a^2(X_{t_i})(t_{i+ 1} - t_i)} \, \varphi_{\Delta_{n,i}^\beta}(X_{t_{i+1}} - X_{t_i}) 1_{\left \{|X_{t_i}| \le \Delta_{n,i}^{-k} \right \} }
\label{eq: contrast function} 
\end{equation}
where
\begin{equation}
m_{\theta, t_i,t_{i+ 1}}(x): =\frac{\mathbb{E}[X_{t_{i+1}}^\theta \varphi_{\Delta_{n,i}^\beta}(X_{t_{i+1}}^\theta - X_{t_i}^\theta)|X_{t_i}^\theta = x]}{\mathbb{E}
[\varphi_{\Delta_{n,i}^\beta}(X_{t_{i+1}}^\theta - X_{t_i}^\theta)|X_{t_i}^\theta = x]}
\label{eq: definition m} 
\end{equation}
and 
$$\varphi_{\Delta_{n,i}^\beta}(X_{t_{i+1}} - X_{t_i}) = \varphi( \frac{X_{t_{i+1}} - X_{t_i}}{\Delta_{n,i}^\beta}),$$ with $\varphi$ a smooth version of the indicator function, such that
$\varphi(\zeta) = 0$ for each $ \zeta$, with $|\zeta| \ge 2$ and $\varphi(\zeta) = 1$ for each $ \zeta $, with $ |\zeta| \le 1$. \\
The last indicator aims to avoid the possibility that $|X_{t_i}|$ is big. The constant $k$ is positive and it will be choosen later, related to the development of $m_{\theta, t_i,t_{i+ 1}}(x)$ (cf. Remark $2$ below). \\
Moreover we define 
$$m_{\theta, h}(x) : = \frac{\mathbb{E}[X_h^\theta \varphi_{h^\beta}(X_h^\theta - X_0^\theta)|X_0^\theta = x]}{\mathbb{E}
[\varphi_{h^\beta}(X_h^\theta - X_0^\theta)|X_0^\theta = x]}.$$
By the homogeneity of the equation we get that $m_{\theta, t_i,t_{i+ 1}}(x)$ depends only on the difference $t_{i+1} - t_i$ and so $m_{\theta, t_i,t_{i+ 1}}(x) = m_{\theta, t_{i+ 1} - t_i}(x)$ that we may denote simply as $m_\theta(x)$, in order to make the notation easier.
\end{definition}
We define an estimator $\hat{\theta}_n$ of $\theta_0$ as
\begin{equation}
\hat{\theta}_n \in \mbox{arg}\min_{\theta \in \Theta} U_n(\theta).
\label{eq: def thetan} 
\end{equation}
The idea, with a finite intensity, is to use the size of $X_{t_{i+1}} - X_{t_i}$ in order to judge the existence of a jump in an interval $[t_i, t_{i + 1})$. The increment of $X$ with continuous transition  could hardly exceed the threshold $\Delta_{n,i}^\beta$ with $\beta \in (0, \frac{1}{2})$. Therefore we can judge a jump occurred if $|X_{t_{i+1}} - X_{t_i}| > \Delta_{n,i}^\beta $. We keep the idea even when the intensity is no longer finite. \\
With a such defined $m_\theta(X_{t_i})$, using the true parameter value $\theta_0$, we have that
$$\mathbb{E}[(X_{t_{i+1}} - m_{\theta_0, t_i, t_{i+1}}(X_{t_i}))\varphi_{\Delta_{n,i}^\beta}(X_{t_{i+1}} - X_{t_i})| X_{t_i} = x ] = \mathbb{E}[X_{t_{i+1}}\varphi_{\Delta_{n,i}^\beta}(X_{t_{i+1}} - x)| X_{t_i} = x ] + $$ 
$$ - \frac{\mathbb{E}[X_{t_{i+1}} \varphi_{\Delta_{n,i}^\beta}(X_{t_{i+1}}  - X_{t_i})|X_{t_i}= x]}{\mathbb{E}[\varphi_{\Delta_{n,i}^\beta}(X_{t_{i+1}} - X_{t_i})|X_{t_i} = x]} \mathbb{E}[\varphi_{\Delta_{n,i}^\beta}(X_{t_{i+1}} - X_{t_i})|X_{t_i} = x] = 0 ,$$
where we have just used the definition and the measurability of $m_{\theta_0, t_i, t_{i+1}}(X_{t_i})$. \\
But, as the transition density is unknown, in general there is no closed expression for $m_{\theta,h} (x)$, hence the contrast is not explicit.  However, in the proof of our results we will need an explicit development of (\ref{eq: contrast function}).  \\ \\
In the sequel, for $\delta \ge 0$, we will denote $R(\theta, \Delta_{n,i}^\delta, x)$ for any function $R(\theta, \Delta_{n,i}^\delta, x)= R_{i,n}(\theta, x)$, where
$R_{i,n}: \Theta \times \mathbb{R} \longrightarrow \mathbb{R}$, $(\theta, x) \mapsto R_{i,n}(\theta, x) $ is such that
\begin{equation}
\exists c > 0 \qquad |R_{i,n}(\theta,x)| \le c(1 + |x|^c)\Delta_{n,i}^\delta
\label{eq: definition R} 
\end{equation}
uniformly in $\theta$ and with $c$ independent of $i,n$.  \\
The functions $R$ represent the term of rest and have the following useful property, consequence of the just given definition:
\begin{equation}
R(\theta, \Delta_{n,i}^\delta, x)= \Delta_{n,i}^\delta R(\theta, \Delta_{n,i}^0, x).
\label{propriety power R} 
\end{equation}
We point out that it does not involve the linearity of $R$, since the functions $R$ on the left and on the right side are not necessarily the same but only two functions on which the control (\ref{eq: definition R}) holds with $\Delta_{n,i}^\delta $ and $\Delta_{n,i}^0$, respectively. \\ \\
We state asymptotic expansions for $m_{\theta, \Delta_{n,i}}$. The cases $\alpha < 1$ and $\alpha  \ge 1$ yield to different magnitude for the rest term. \\ \\
\textbf{Case $\alpha \in (0,1)$:}

\begin{theorem}
Suppose that Assumptions 1 to 4 hold and that $\beta \in (0, \frac{1}{2})$ and $\alpha \in (0,1)$ are given in definition $1$ and the third point of Assumption $4$, respectively. Then
\begin{equation}
\mathbb{E}[\varphi_{\Delta_{n,i}^\beta}(X^\theta_{t_{i+1}} - X_{t_i}^\theta)|X_{t_i}^\theta = x] = 1 + R(\theta,\Delta_{n,i}^{(1 - \alpha \, \beta) \land (2 - 3 \beta)}, x).
\label{eq: dl phi} 
\end{equation}
\label{th: dl den alpha <1} 
\end{theorem}

\begin{theorem}
Suppose that Assumptions 1 to 4 hold and that $\beta \in (0, \frac{1}{2})$ and $\alpha \in (0,1)$ are given in definition $1$ and the third point of Assumption $4$, respectively. Then
\begin{equation}
\mathbb{E}[(X_{t_{i+1}}^\theta - x) \varphi_{\Delta_{n,i}^\beta}(X_{t_{i+1}}^\theta - X_{t_i}^\theta)|X_{t_i}^\theta = x]=\Delta_{n,i}\, b(x, \theta) +
\label{eq: dl num} 
\end{equation}
$$-  \Delta_{n,i} \, \int_{\mathbb{R} \backslash \left \{0 \right \}} z\, \gamma(x)\,[1 - \varphi_{\Delta_{n,i}^\beta}(\gamma(x)z)]\,F(z)dz \, + \, R(\theta,\Delta_{n,i}^{2 - 2 \beta}, x).$$
There exists $k_0 > 0$ such that, for $|x| \le \Delta_{n,i}^{- k_0}$,
\begin{equation}
m_{\theta, \Delta_{n,i}}(x) = x + \Delta_{n,i} \, b(x, \theta) + 
\label{eq: dl m} 
\end{equation}
$$-  \Delta_{n,i} \, \int_{\mathbb{R} \backslash \left \{0 \right \}} z\, \gamma(x)\,[1 - \varphi_{\Delta_{n,i}^\beta}(\gamma(x)z)]\,F(z)dz \, + \, R(\theta,\Delta_{n,i}^{2 - 2 \beta}, x). $$
\label{th: dL mtheta alpha <1} 
\end{theorem}
. \\
\textbf{Case $\alpha \in [1,2)$:}
\begin{theorem}
Suppose that Assumptions 1 to 4 hold and that $\beta \in (0, \frac{1}{2})$ and $\alpha \in [1,2)$ are given in definition $1$ and the third point of Assumption $4$, respectively. Then
\begin{equation}
\mathbb{E}[\varphi_{\Delta_{n,i}^\beta}(X^\theta_{t_{i+1}} - X_{t_i}^\theta)|X_{t_i}^\theta = x] = 1 + R(\theta,\Delta_{n,i}^{(1 - \alpha\beta) \land (2-4\beta)}, x).
\label{eq: dl phi alpha >1} 
\end{equation}
\label{th: dl den alpha >1} 
\end{theorem}
\begin{theorem}
Suppose that Assumptions 1 to 4 hold and that $\beta \in (0, \frac{1}{3})$ and $\alpha \in [1,2)$ are given in definition $1$ and the third point of Assumption $4$, respectively. Then
\begin{equation}
\mathbb{E}[(X_{t_{i+1}}^\theta - x) \varphi_{\Delta_{n,i}^\beta}(X_{t_{i+1}}^\theta - X_{t_i}^\theta)|X_{t_i}^\theta = x]=\Delta_{n,i}\, b(x, \theta) +
\label{eq: dl num alpha >1} 
\end{equation}
$$-  \Delta_{n,i} \, \int_{\mathbb{R} \backslash \left \{0 \right \}} z\, \gamma(x)\,[1 - \varphi_{\Delta_{n,i}^\beta}(\gamma(x)z)]\,F(z)dz \, + \, R(\theta,\Delta_{n,i}^{2 - 3 \beta}, x).$$
There exists $k_0 > 0$ such that, for $|x| \le \Delta_{n,i}^{- k_0}$,
\begin{equation}
m_{\theta, \Delta_{n,i}}(x) = x + \Delta_{n,i} \, b(x, \theta) + 
\label{eq: dl m alpha>1} 
\end{equation}
$$-  \Delta_{n,i} \, \int_{\mathbb{R} \backslash \left \{0 \right \}} z\, \gamma(x)\,[1 - \varphi_{\Delta_{n,i}^\beta}(\gamma(x)z)]\,F(z)dz \, + \, R(\theta,\Delta_{n,i}^{2 - 3 \beta}, x). $$
\label{th: dL mtheta alpha  > 1} 
\end{theorem}

\begin{remark}
The constant $k$ in the definition \eqref{eq: contrast function} of contrast function can be taken in the interval $(0, k_0]$. In this way $\Delta_{n,i}^{-k} \le \Delta_{n,i}^{-k_0} $ and so \eqref{eq: dl m} or \eqref{eq: dl m alpha>1} holds for $|x| = |X_{t_i}| $ smaller than $\Delta_{n,i}^{-k}$. \\
If it is not the case the contribution of the observation $X_{t_i}$ in the contrast function is just $0$. However we will see that suppressing the contribution of too big $|X_{t_i}| $ does not effect the efficiency property of our estimator.
\end{remark}

\begin{remark}
In the development \eqref{eq: dl num} or \eqref{eq: dl num alpha >1} the term $\Delta_{n,i} \, \int_{\mathbb{R} \backslash \left \{0 \right \}} z\, \gamma(x)\,[1 - \varphi_{\Delta_{n,i}^\beta}(\gamma(x)z)]\,F(z)dz$ is independent of $\theta$, hence it will disappear in the difference $m_\theta(x) - m_{\theta_0}(x)$, but it is not negligible compared to $ \Delta_{n,i} \, b(x, \theta) $ since its order is $ \Delta_{n,i}$ if $\alpha \in (0,1)$ and at most $\Delta_{n,i}^{\frac{1}{2}}$ if $\alpha \in [1,2)$. Indeed, by the definition of the function $\varphi$, we know that we can consider as support of $\varphi_{\Delta_{n,i}^\beta}(0) - \varphi_{\Delta_{n,i}^\beta}(\gamma(x)z)$ the interval $ c \times [- \frac{\Delta^\beta_{n,i}}{\left \| \gamma \right \|_\infty}, \frac{\Delta^\beta_{n,i}}{\left \| \gamma \right \|_\infty} ]^c$. If $\alpha < 1$, using moreover the third point of Assumption 4 we get the following estimation:
\begin{equation}
|\Delta_{n,i} \, \int_{\mathbb{R} \backslash \left \{0 \right \}} z\, \gamma(x)\,[1 - \varphi_{\Delta_{n,i}^\beta}(\gamma(x)z)]\,F(z)dz |\le R(\theta_0, \Delta_{n,i}^1, X_{t_i}).
\label{eq: preuve ordre 1/2, alpha <1 } 
\end{equation}
Otherwise, if $\alpha \ge 1$, we have
$$|\Delta_{n,i} \, \int_{\mathbb{R} \backslash \left \{0 \right \}} z\, \gamma(x)\,[1 - \varphi_{\Delta_{n,i}^\beta}(\gamma(x)z)]\,F(z)dz | \le c |\Delta_{n,i}| \int_{c \times [- \frac{\Delta^\beta_{n,i}}{\left \| \gamma \right \|_\infty}, \frac{\Delta^\beta_{n,i}}{\left \| \gamma \right \|_\infty} ]^c} |z|^{- \alpha}= R(\theta, \Delta_{n,i}^{1 + \beta (1 - \alpha)}, x), $$
with $\beta \in (0, \frac{1}{2})$ and $\alpha \in [1,2)$, hence the exponent on $\Delta_{n,i}$ is always more than $\frac{1}{2}$. \\
We can therefore write in the first case
\begin{equation}
m_{\theta, \Delta_{n,i}}(x)= x + R(\theta, \Delta_{n,i}, x)= R(\theta, \Delta_{n,i}^0, x)
\label{eq: mtheta as R(theta,1) alpha >1} 
\end{equation}
and in the second
\begin{equation}
m_{\theta, \Delta_{n,i}}(x)= x + R(\theta, \Delta_{n,i}^{1 + \beta(1 - \alpha)}, x)= R(\theta, \Delta_{n,i}^0, x).
\label{eq: mtheta as R(theta,1) alpha <1} 
\end{equation}
\end{remark}

\begin{remark}
In Theorems \ref{th: dl den alpha <1} - \ref{th: dl den alpha >1} we do not need conditions on $\beta$ because, for each $\beta \in (0, \frac{1}{2})$ and for each $\alpha \in (0,2)$ the exponent on $\Delta_{n,i}$ is positive and therefore the last term of \eqref{eq: dl phi alpha >1} is negligible compared to $1$. In Theorem \ref{th: dL mtheta alpha  > 1}, instead, $R$ is a negligible function if and only if $2- 3 \beta \ge 1$, it means that it must be $\beta \le \frac{1}{3}$.
We have taken $\beta \in (0, \frac{1}{3})$ and so such a condition is always respected.
\end{remark}

\subsection{Main results}\label{S:Main}
Let us introduce the Assumption $A_\beta$ that turns out starting from Theorems \ref{th: dl den alpha <1}, \ref{th: dL mtheta alpha <1}, \ref{th: dl den alpha >1} and \ref{th: dL mtheta alpha  > 1}:\\ \\
ASSUMPTION $A_\beta$: We choose $\beta \in (0, \frac{1}{2})$ if $\alpha \in (0,1)$. If on the contrary $\alpha \in [1,2)$, then we take $\beta$ in $(0, \frac{1}{3})$. \\

The following theorems give a general consistency result and the asymptotic normality of the estimator $\hat{\theta}_n$, that hold without further assumptions on $n$ and $\Delta_n$.

\begin{theorem}
{(Consistency)} \\
Suppose that Assumptions 1 to 7 and $A_\beta$ hold and let $k$ of the definition of the contrast function \eqref{eq: contrast function} be in $(0, k_0)$. Then the estimator $\hat{\theta}_n$ is consistent in probability:
$$\hat{\theta}_n \xrightarrow{\mathbb{P}} \theta_0, \qquad n \rightarrow \infty.$$
\end{theorem}
Recalling that the Fisher information I is given by (\ref{eq: 2.2 GLM, Fisher}), we give the following theorem.

\begin{theorem}
{(Asymptotic normality)} \label{th:Asymptotic_normality}  \\
Suppose that Assumptions 1 to 8 and $A_\beta$ hold, and $0 < k < k_0$. \\
Then the estimator $\hat{\theta}_n$ is asymptotically normal:
$$\sqrt{t_n}(\hat{\theta}_n - \theta_0) \xrightarrow{\mathcal{L}} N(0, I^{-1}(\theta_0)), \qquad n\rightarrow\infty.$$
\end{theorem}

\begin{remark}
Furthermore, the estimator $\hat{\theta}_n$ is asymptotically efficient in the sense of the H\'ajek-Le Cam convolution theorem. \\
The H\'ajek$-$LeCam convolution theorem states that any regular estimator in a parametric model which satisfies LAN property is asymptotically equivalent to a sum of two independent random variables, one of which is normal with asymptotic variance equal to the inverse of Fisher information, and the other having arbitrary distribution. The efficient estimators are those with the second component identically equal to zero. \\
The model \eqref{eq: model} is LAN with Fisher information $I(\theta)= \int_\mathbb{R} \frac{(\dot{b}(\theta,x))^2}{a^2(x)} \pi^\theta (dx)$ (see \hyperref[csl:16]{(Gloter, Loukianova, \& Mai, 2018)}) and thus $\hat{\theta}_n$ is efficient. 
\end{remark}

\begin{remark}
We point out that, contrary to the papers \hyperref[csl:16]{(Gloter, Loukianova, \& Mai, 2018)} and \hyperref[csl:10]{(Shimizu \& Yoshida, 2006)}, in this case there is not any condition on the sampling, that can be irregular and with $\Delta_n$ that goes slowly to zero. 
 On the other hand, our contrast function relies on the quantity $m_{\theta,h}(x)$ which is not explicit in general.
\end{remark}

\section{Practical implementation of the contrast method}\label{S: Practical implementation}
In order to use in practice the contrast function \eqref{eq: contrast function}, one need to know the values of the quantities
$m_{\theta,t_i,t_{i+1}}(X_{t_i})$. 
In most cases, it seems impossible to find an explicit expression for the function $m_{\theta,h}$ appearing in Definition \ref{D:definition_contrast}. However, explicit or numerical approximations of this function seem available in many situations.

\subsection{Approximate contrast function}\label{S: Approximate contrast}
Let us assume that one has at disposal an approximation of the function $m_{\theta,h}(x)$, denoted by $\widetilde{m}_{\theta,h}(x)$ which satisfies, for $|x|\le h^{-k_0}$,
\begin{equation*}
| \widetilde{m}_{\theta,h}(x) -m_{\theta,h}(x) | \leq R(\theta,h^{\rho},x)
\end{equation*}
where the constant $\rho>1$  assesses the quality of the approximation.  
 We assume that the first three derivatives of $\tilde{m}_{h,\theta}$ with respect to the parameter provide approximation of the derivatives of ${m}_{h,\theta}$, in the following way
\begin{align}
 \label{E:approx_m_mtilde_dot} 
&| \frac{\partial^i \widetilde{m}_{\theta, h} (x) }{\partial \theta^i}  -
\frac{\partial^i  m_{\theta, h}(x)}{\partial \theta^i}| \le R(\theta,h^{1+\epsilon},x), \quad \text{for $i=1,2$},
\\ \label{E:approx_m_mtilde_3dot} 
&| \frac{\partial^3\widetilde{m}_{\theta, h}(x)}{\partial \theta^3}  -
\frac{\partial^3m_{\theta, h}(x)}{\partial \theta^3} | \le R(\theta,h,x) ,
\end{align}
for all $|x|\le h^{-k_0}$ and where $\epsilon>0$.
Let us stress that from Proposition \ref{P:expansion_m_dot} below, we know the derivatives with respect to $\theta$ of the quantity $m_{h,\theta}$.

Now, we consider $\widetilde{\theta}_n$ the estimator obtained from minimization of the contrast function \eqref{eq: contrast function} where one has replaced $m_{\theta,t_i,t_{i+1}}(X_{t_i})$ by its approximation $\widetilde{m}_{\theta,\Delta_{n,i}}(X_{t_i})$. Then, the result of  Theorem \ref{th:Asymptotic_normality} can be extended as follows. \\

\begin{proposition}
\label{P:normality_approximate_contrast} 
Suppose that Assumptions 1 to 8 and $A_\beta$ hold, with  $0 < k < k_0$, and that $\sqrt{n}\Delta_n^{\rho -1/2} \rightarrow  0$ as $n \rightarrow \infty$. \\
Then, the estimator $\widetilde{\theta}_n$ is asymptotically normal:
$$\sqrt{t_n}(\widetilde{\theta}_n - \theta_0) \xrightarrow{\mathcal{L}} N(0, I^{-1}(\theta_0)), \qquad n\rightarrow\infty.$$
\end{proposition}

We give below several examples of approximations of $m_{\theta,h}$. Let us stress that, in general, 
Theorem \ref{th: dL mtheta alpha <1} (resp. Theorem \ref{th: dL mtheta alpha  > 1}) provides  an explicit approximation of $m_{\theta,\Delta_{n,i}}(x)$ with an error of order $\Delta_{n,i}^{2-2\beta}$  (resp. of order $\Delta_{n,i}^{2-3\beta}$). They can be used to construct an explicit contrast function.
In the next section we show that when the intensity is finite, it is possible to construct an explicit approximation of $m_{\theta,h}$ with arbitrarily high order.

\subsection{Explicit contrast in the finite intensity case.}\label{S:Explicit}
In the case with finite intensity it is possible to make the contrast explicit, using the development of $m_{\theta, \Delta_{n,i}}$ proved in the next proposition. We need the following assumption: \\ \\
ASSUMPTION $A_f$:
\begin{enumerate}
\item We have $F(z) = \lambda F_0(z)$, $\int_\mathbb{R} F_0(z)dz= 1$ and $F$ is a $\mathcal{C}^\infty$ function.
\item We assume that $x \mapsto a(x)$, $x \mapsto b(x, \theta)$ and $x \mapsto \gamma(x)$ are $\mathcal{C}^\infty$ functions, they have at most uniform in $\theta$ polynomial growth as well as their derivatives.
\end{enumerate}
Let us define $A_K^{(k)}(x) = \bar{A}_c^k(g)(x)$, with $g(y) = (y - x)$ and $\bar{A}_c(f) = \bar{b} f' + \frac{1}{2} a^2 f''$; $\bar{b}(\theta, y) = b( \theta, y) - \int_{\mathbb{R}} \gamma (y) z F(z) dz$ as in the Remark $1$. \\

\begin{proposition}
Assume that $A_f$ holds and let $\varphi$ be a $\mathcal{C}^\infty$ function that has compact support and such that $\varphi \equiv 1$ on $[-1, 1]$ and $\forall k \in \left \{ 0, ... , M \right \}$, $\int_\mathbb{R} x^k \varphi(x) dx = 0$ for $M \ge 0$. Then, for  $|x| \le \Delta_{n,i}^{   -k_0} $ with some $k_0 > 0$,
\begin{equation}
m_{\theta, \Delta_{n,i}}(x) = x + \sum_{ k = 1}^{\lfloor \beta  (M+2)  \rfloor} A_K^{(k)}(x) \frac{\Delta_{n,i}^k}{k!} + R(\theta, \Delta_{n,i}^{\beta  (M+2) }, x).
\label{eq: dl m intensita finita} 
\end{equation}
\label{prop: dl m intensita finita} 
\end{proposition}
In order to say that \eqref{eq: dl m intensita finita} holds, we have to prove the existence of a function $\varphi$ with a compact support such that $\varphi \equiv 1$ on $[-1, 1]$ and, $\forall k \in \left \{ 0, ..., M \right \}$, $\int_\mathbb{R} x^k \varphi (x) dx$. 
We build it through $\psi$, a function with compact support, $\mathcal{C}^\infty$, such that
$\psi|_{[-1, 1]}(x)= \frac{x^M}{M !}$. We then define $\varphi(x) : = \frac{\partial^M}{\partial x^M} \psi (x)$. \\
In this way we have $\varphi \equiv 1$ on $[-1, 1]$, $\varphi$ is $\mathcal{C}^\infty$, with compact support and such that for each $l \in \left \{ 0, ... M \right \}$, using the integration by parts,
$\int_\mathbb{R} x^l \varphi(x) dx = 0$, as we wanted. \\

\begin{remark}
The development \eqref{eq: dl m intensita finita} is the same found in Kessler \hyperref[csl:15]{(Kessler, 1997)} in the case without jumps and it is obtained by the iteration of the continuous generator $\bar{A}_c$.  Hence, it is completely explicit. Let us stress that in Kessler \hyperref[csl:15]{(Kessler, 1997)} the right hand side of \eqref{eq: dl m intensita finita} stands for an approximation of $E[\bar{X}^\theta_{\Delta_{n,i}} \mid \bar{X}^\theta_0=x]$ where $\bar{X}^\theta$ is the continuous diffusion solution of
$d \bar{X}_t^\theta=\bar{b}(\theta,\bar{X}_s^\theta) ds + \sigma(\bar{X}_s^\theta) dW_s$. From Proposition \ref{prop: dl m intensita finita}, the right hand side of \eqref{eq: dl m intensita finita} is also an approximation of
$m_{\theta,\Delta_{n,i}}(x)= \displaystyle \frac{ E[X^\theta_{\Delta_{n,i}} \varphi_{\Delta^\beta_{n,i}}(X^\theta_{\Delta_{n,i}}-x) \mid {X}^\theta_0=x]}{E[\varphi_{\Delta^\beta_{n,i}}(X^\theta_{\Delta_{n,i}}-x)  \mid {X}^\theta_0=x]} $ 
in the case of finite activity jumps, and for a truncation kernel $\varphi$ satisfying $\forall k \in \left \{ 0, ... , M \right \}$, $\int_\mathbb{R} x^k \varphi(x) dx = 0$. We emphasize that
in the expansion of $m_{\theta,\Delta_{i,n}}$ given in Proposition \ref{prop: dl m intensita finita}, the contribution of the discontinuous part of the generator disappears only thanks to the choice of an oscillating function $\varphi$.
\end{remark} 
\begin{remark}
In the definition of the contrast function \eqref{eq: contrast function} we can replace  $m_{\theta, \Delta_{i,n}}(x)$ with the explicit approximation $\widetilde{m}_{\theta,\Delta_{n,i}}^k(x): = x + \sum_{h = 1}^k \frac{\Delta_{n,i}^h}{h!} A_K^{(h)}(x) $, with an error $R(\theta, \Delta_{n,i}^k, x)$, for $k \le \lfloor 2(M+1) \beta \rfloor$. 
Using $A_K^{(1)}(x)=\Delta_{n,i} [ b(\theta,x) - \int_{\mathbb{R}} \gamma (y) z F(z) dz ] $ and the expansions
\eqref{eq: dl dot mtheta}--\eqref{eq: estimation 3dot m } we deduce that the conditions \eqref{E:approx_m_mtilde_dot} --
\eqref{E:approx_m_mtilde_3dot} are valid.  
Then, by application of Proposition \ref{P:normality_approximate_contrast}, we can see that the associated estimator is efficient under the assumption $\sqrt{n}\Delta^{k - \frac{1}{2}}_n \rightarrow 0$ for $n \rightarrow \infty$. As $M$, and thus $k$, can be chosen arbitrarily large, we see that the sampling step $\Delta_n$ is allowed to converge to zero in a arbitrarily slow polynomial rate as a function of $n$. It turns out that a slow sampling step necessitates to choose a truncation function with more vanishing moments.
\end{remark}

\section{Numerical experiments}\label{S: Numerical}

\subsection{Finite jump activity}
Let us consider the model
\begin{equation}
 \label{E:OU_in_simulation} 
X_t=X_0+ \int_0^t (\theta_1 X_s + \theta_2) ds + \sigma W_t + \gamma
\int_0^t \int_{\mathbb{R}\backslash \left \{0 \right \}}  z \tilde{\mu}(ds,dz),
\end{equation}
where the compensator of the jump measure is $\overline{\mu}(ds,dz)=\lambda F_0(z)ds dz$ for $F_0$ the probability density of 
the law $\mathcal{N}(\mu_J,\sigma_J^2)$ with $\mu_J \in \mathbb{R}$, $\sigma_J >0$, $\sigma>0$,  $\theta_1<0$, $\theta_2 \in \mathbb{R}$, $\gamma \ge 0$, $\lambda \ge 0$. Since the jump activity is finite, we know from Section \ref{S:Explicit} that the function $m_{(\theta_1,\theta_2),\Delta_{n,i}}(x)$ can be approximated at any order using
\eqref{eq: dl m intensita finita}. As the latter is also the asymptotic expansion of the first conditional moment for the continuous S.D.E. 
$ \bar{X}_t=\bar{X}_0+ \int_0^t (\theta_1 \bar{X}_s + \theta_2 - \gamma \lambda \mu_J) ds + \sigma W_t$ , which is explicit due to the linearity of the model, we decide to directly use the expression of the conditional moment and set
\begin{equation}
 \label{E:m_tilde_OU}  
\widetilde{m}_{(\theta_1,\theta_2),\Delta_{n,i}}(x)= (x + \frac{\theta_2}{\theta_1} - \frac{\gamma \lambda \mu_J}{\theta_1}) e^{\theta_1 \Delta_{n,i}}+\frac{\gamma \lambda \mu_J-\theta_2}{\theta_1}.
\end{equation}
Following Nikolskii \hyperref[csl:20]{(Nikolskii, 1977)}, we construct oscillating truncation functions in the following way. First, we choose $\varphi^{(0)}: \mathbb{R} \to [0,1]$ a $\mathcal{C}^\infty$ symmetric function with support on $[-2,2]$ such that $\varphi^{(0)}(x)=1$ for $|x| \leq 1$. We let, for $d>1$, $\varphi^{(1)}_d(x)= (d\varphi^{(0)}(x)-\varphi^{(0)}(x/d))/(d-1)$, which is a function equal to $1$ on $[-1,1]$,  vanishing on $[-d,d]^c$ and such that $\int_{\mathbb{R}} \varphi^{(1)}_d(x) dx=0$. For $l \in \mathbb{N}$, $l \ge 1$, and $d>1$, we set 
$\varphi_d^{(l)}(x)=c_d^{-1}\sum_{k=1}^l C_l^k (-1)^{k+1} \frac{1}{k} \varphi^{(1)}_d(x/k)$, where
$c_d=\sum_{k=1}^l C_l^k (-1)^{k+1} \frac{1}{k}$.  One can check that $\varphi_d^{(l)}$ is compactly supported, equal to $1$ on $[-1,1]$,  and that for all 
$k \in  \{0, \dots, l \}, \int_{\mathbb{R}} x^k  \varphi_d^{(l)}(x) dx =0$, for $l \ge 1$. 
With these notations, we estimate the parameter $\theta=(\theta_1,\theta_2)$ by minimization of the contrast function
\begin{equation}
 \label{E:contrast_OU} 
U_n(\theta)=\sum_{i=0}^{n-1} (X_{t_{i+1}}-
\widetilde{m}_{(\theta_1,\theta_2),\Delta_{n,i}}(X_{t_i}))^2
\varphi^{(l)}_{c \Delta_{n,i}^\beta}(X_{t_{i+1}}-X_{t_i}),
\end{equation}
where $l \in \mathbb{N}$ and $c>0$ will be specified latter.

For numerical simulations, we choose $T=2000$, $n=10^4$, $\Delta_{i,n}=\Delta_n=1/5$, $\theta_1=-0.5$, $\theta_2=2$ and $X_0=x_0=4$. We estimate the bias and standard deviation of our estimators using a Monte Carlo method based on 5000 replications. 
As a start, we consider a situation without jumps $\lambda=0$ , in which we remove the truncation function $\varphi$
 in the contrast, as it is useless in absence of jumps. 
 In Table \ref{T:no_jumps}, we compare the estimator $\widetilde{\theta}_n$ which uses the Kessler exact bias correction given by \eqref{E:m_tilde_OU}, with an estimator based on the Euler scheme approximation where one uses the 
 approximation $\widetilde{m}^{\text{Euler}}_{(\theta_1,\theta_2),{\Delta_{n,i}}}(x)=x+\Delta_{n,i}(\theta_1 x + \theta_2)$.
From Table \ref{T:no_jumps} we see that the estimator $\tilde{\theta}^{\text{Euler}}_n$ based on Euler contrast exhibits some bias  which is completely removed using Kessler's correction.   
Next, we set a jump intensity $\lambda=0.1$, with jumps size whose common law is 
  $\mathcal{N}(0,2)$ and set $\gamma=1$. We use the contrast function relying on \eqref{E:m_tilde_OU}. Results are given for three choices of truncation function, $\varphi^{(0)}$, $\varphi^{(2)}_d$ and $\varphi^{(3)}_d$ where $d=3$. Plots of these functions are given in Figure
\ref{Fig:phi}. 
We choose $\beta=0.49$ and $c=1$. As the true value of the volatility is $\sigma=0.3$, this choice enables most of the increments without jumps of $X$ on $[t_{i},t_{i+1}]$ to be such that $\varphi^{(l)}_{c \Delta_{n,i}^\beta}(X_{t_{i+1}}-X_{t_i})=1$.
  Let us stress that, if $\sigma$ is unknown, it is possible to estimate, 
    even roughly, the local value of the volatility in order to choose $c$ accordingly (see \hyperref[csl:16]{(Gloter, Loukianova, \& Mai, 2018)} for analogous discussion). Results in Table \ref{T:few_jumps} show that the estimator works well, with a reduced bias for all choices of truncation function. Especially the bias is much smaller than the one of the Euler scheme contrast in absence of jumps.
 It shows the benefit of using \eqref{E:m_tilde_OU} in the contrast function, even if the truncation function is not oscillating as is it when we consider $\varphi^{(0)}$. We remark that by the choice of a symmetric truncation function one has $\int_{\mathbb{R}} u \varphi^{(0)}(u) d u=0$ and inspecting the proof of Proposition \ref{P:expansion_m_dot} it can be seen that this conditions is sufficient, in the expansion of $m_{\theta,\Delta_{n,i}}$, to suppress the largest contribution of the discrete part of the generator.
 
  If the number of jumps is greater, e.g. for $\lambda=1$, we see in Table \ref{T:many_jumps} that using the oscillating kernels $\varphi^{(2)}_d, ~\varphi^{(3)}_d$ yields to a smaller bias than using $\varphi^{(0)}$, whereas it tends to increase the standard deviation of the estimator. The estimator we get using $\varphi^{(3)}_d$ performs well in this situation, it has a negligible bias and a standard deviation comparable to the one in the case where the process has no jump.
\begin{table}[ht]
\begin{center}
\begin{tabular}{|c|c|c|}
\hline 
& Mean (std) for $\theta_1=-0.5$& Mean (std) for $\theta_2=2$ \\ 
\hline 
\rule{0pt}{1.3em}
$\tilde{\theta}^{\text{Euler}}_n$ &-0.4783 (0.0213) &  1.9133 (0.0856)  \\ 
\hline 
\rule{0pt}{1.3em} $\displaystyle \widetilde{\theta}_n$& -0.5021 (0.0236) &  2.0084 (0.0947)  \\ 
\hline 
\end{tabular} 
\caption{Process without jump \label{T:no_jumps}} 
\end{center}
\end{table}
\begin{table}[ht]
\begin{center}
\begin{tabular}{|c|c|c|}
\hline 
& Mean (std) for $\theta_1=-0.5$& Mean (std) for $\theta_2=2$ \\ 
\hline 
\rule{0pt}{1.3em} $\widetilde{\theta_n}$ using $\varphi^{(0)}$ &  -0.4967 (0.0106) &     1.9869 (0.0430)  \\ 
\hline 
\rule{0pt}{1.3em} $\widetilde{\theta_n}$ using $\varphi^{(2)}_d$& -0.4990 (0.0153) &   1.9959 (0.0622)  \\ 
\hline 
\rule{0pt}{1.3em} $\widetilde{\theta_n}$ using $\varphi^{(3)}_d$& -0.5006 (0.0196) &    2.0023 (0.0798)  \\ 
\hline 
\end{tabular} 
\caption{Gaussian jumps with $\lambda=0.1$ \label{T:few_jumps}}
\end{center} 
\end{table}
\begin{table}[ht]
\begin{center}
\begin{tabular}{|c|c|c|}
\hline 
& Mean (std) for $\theta_1=-0.5$& Mean (std) for $\theta_2=2$ \\ 
\hline 
    \rule{0pt}{1.3em} $\widetilde{\theta_n}$ using $\varphi^{(0)}$ &  -0.4623 (0.0059) &  1.8495 (0.0256)  \\ 
\hline 
\rule{0pt}{1.3em}$\widetilde{\theta_n}$ using $\varphi^{(2)}_d$& -0.4886  (0.0161)&    1.9549 (0.0710)  \\ 
\hline 
\rule{0pt}{1.3em}$\widetilde{\theta_n}$ using $\varphi^{(3)}_d$& -0.5033 (0.0243) & 2.0136 (0.1059)  \\ 
\hline 
\end{tabular} 
\caption{Gaussian jumps with $\lambda=1$ \label{T:many_jumps}} 
\end{center}
\end{table}

\begin{figure}[ht]
	\centering
	\begin{subfigure}[b]{0.3\linewidth}
		\includegraphics[width=\linewidth]{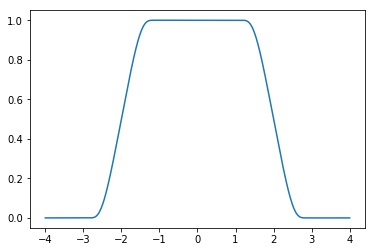}
		\caption{$\varphi^{(0)}$}
	\end{subfigure}
	\begin{subfigure}[b]{0.3\linewidth}
		\includegraphics[width=\linewidth]{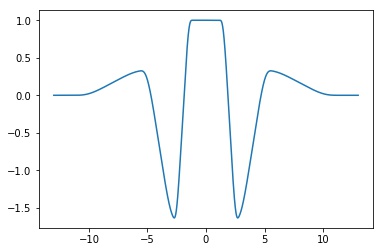}
		\caption{$\varphi^{(2)}_d$ with $d=3$}
	\end{subfigure}
	\begin{subfigure}[b]{0.3\linewidth}
		\includegraphics[width=\linewidth]{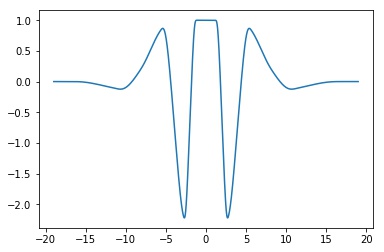}
		\caption{$\varphi^{(3)}_d$ with $d=3$}
	\end{subfigure}
\caption{Plot of the truncation functions}
\label{Fig:phi}
\end{figure}

\subsection{Infinite jumps activity} \label{Ss:Num_infinite}
Let us consider $X$ solution to the stochastic differential equation \eqref{E:OU_in_simulation}, where the compensator of the jump measure is $\overline{\mu}(ds,dz)= \frac{e^{-z}}{z^{1+\alpha}} 1_{(0,\infty)}(z) ds dz$ with $\alpha \in (0,1)$. This situation corresponds to the choice of the Levy process 
$(\int_0^t \int_{\mathbb{R}\backslash \left \{0 \right \}}  z \tilde{\mu}(ds,dz))_t$ being a tempered $\alpha$-stable jump process. 
In the case of infinite jump activity, we have no result providing approximations at any arbitrary order of $m_{\theta, \Delta_{n,i}}$. However, we can use Theorem \ref{th: dL mtheta alpha <1} to find some useful explicit approximation.

According to \eqref{eq: dl m} and taking into account that the threshold level is $c\Delta_{n,i}^\beta$ for some $c>0$, we have
\begin{align*}
m_{\theta,\Delta_{n,i}}(x)
 &= x + \Delta_{n,i}(\theta_1 x+ \theta_2) - \Delta_{n,i} \gamma \int_0^\infty \frac{e^{- z}}{z^{\alpha}}  d z +\Delta_{n,i} \gamma \int_0^\infty \varphi_{c \Delta_{n,i}^\beta}(\gamma z)   \frac{e^{- z}}{z^{\alpha}} d z + R(\theta,\Delta_{n,i}^{2 - 2 \beta}, x)
 \\
 &= x +  \Delta_{n,i}  \overline{b}(x,\theta_1,\theta_2) +  \Delta_{n,i}^{1+\beta(1-\alpha)}c^{1-\alpha} \gamma^\alpha \int_0^\infty \varphi(v)   \frac{e^{- \frac{c v \Delta_{n,i}^\beta}{\gamma}} }{v^\alpha} d v + R(\theta,\Delta_{n,i}^{2 - 2 \beta}, x),
\end{align*}
where in the last line, following the notation of Remark \ref{R:alpha_smaller_1}, we have set $\overline{b}(x,\theta_1,\theta_2)=
(\theta_1 x+ \theta_2) - \gamma \int_0^\infty \frac{e^{- z}}{z^{\alpha}}  d z$, and we
 make the change of variable $v=\frac{\gamma z}{c \Delta_{n,i}^\beta}$.
This leads us to consider the approximation
\begin{equation}
\label{E:m_tilde_stable} 
\widetilde{m}_{\theta,\Delta_{n,i}}(x)= x + \Delta_{n,i}\overline{b}(x,\theta_1,\theta_2)
+ \Delta_{n,i}^{1+\beta(1-\alpha)}c^{1-\alpha} \gamma^\alpha \int_0^\infty \varphi(v)   \frac{1}{v^\alpha} d v,
\end{equation}
which is such that $| \widetilde{m}_{\theta,\Delta_{n,i}}(x)-{m}_{\theta,\Delta_{n,i}}(x) |  
\le R(\theta,\Delta_{n,i}^{(2-2\beta) \wedge (1+\beta(2-\alpha))},x)$.

For numerical simulations, we choose $T=100$, $n=10^4$, $\Delta_{i,n}=\Delta_n=1/100$, $\theta_1=-0.5$, $\theta_2=2$, $X_0=x_0=4$, $\gamma=1$, $\sigma=0.3$ and $\alpha \in \{0.1,~0.3,~0.5\}$. To illustrate the estimation method, we focus on the estimation of the parameter $\theta_2$ only, as the minimisation of the contrast defined by \eqref{E:contrast_OU}--\eqref{E:m_tilde_stable} yields to the explicit estimator,
\begin{align}
 \nonumber
\widetilde{\theta}_{2,n}&=
\frac{ \sum_{i=0}^{n-1} (X_{t_{i+1}} - X_{t_i}- \Delta_n \theta_1 X_{t_i}) \varphi_{c \Delta_{n}^\beta} (X_{t_{i+1}}-X_{t_{i}})  }
{\Delta_n \sum_{i=0}^{n-1} \varphi_{c \Delta_{n}^\beta} (X_{t_{i+1}}-X_{t_{i}}) } 
- \gamma \int_0^\infty \frac{e^{-z}}{z^\alpha} dz \\
\nonumber
& \quad \quad \quad \quad \quad \quad \quad \quad \quad \quad \quad \quad
-\Delta_{n}^{\beta(1-\alpha)} c^{1-\alpha} \gamma^\alpha \int_0^\infty \varphi(v) \frac{1}{v^\alpha} dv
\\ \label{E:theta_tilde_theat_Euler} 
&=: \widetilde{\theta}_{2,n}^{\text{Euler}} - \Delta_{n}^{\beta(1-\alpha)} c^{1-\alpha} \gamma^\alpha \int_0^\infty \varphi(v) \frac{1}{v^\alpha} dv. 
\end{align}
We can see that the estimator $\widetilde{\theta}_{2,n}$ is a corrected version of the estimator $\widetilde{\theta}_{2,n}^{\text{Euler}}$, that would result from the choice of the approximation $m_{\theta,\Delta_n}(x)
\approx x + \Delta_{n}\overline{b}(x,\theta_1,\theta_2)$ in the definition of the contrast function. 
Comparing with estimators of earlier works (e.g. \hyperref[csl:16]{(Gloter, Loukianova, \& Mai, 2018)}, \hyperref[csl:11]{(Shimizu, 2006)}), the presence of this correction term appears new.
In lines 2--3 of Table \ref{T:stable}, we compare the mean and standard deviation of  $\widetilde{\theta}_{2,n}$ and $\widetilde{\theta}_{2,n}^{\text{Euler}}$ for $\alpha \in \{0.1,~0.3,~0.5\}$ and with the choice
$c=1$, $\beta=0.49$ and $\varphi=\varphi^{(0)}$ (see Figure \ref{Fig:phi}).
We see that the estimator $\widetilde{\theta}_{2,n}$ performs well  and the correction term in \eqref{E:theta_tilde_theat_Euler} drastically reduces the bias present in $\widetilde{\theta}_{2,n}^{\text{Euler}}$, especially when the jump activity near $0$ is high, corresponding to larger values of $\alpha$.
If we take a threshold level $c=1.5$ higher, we see in line 5 of Table \ref{T:stable} that the bias of the estimator $\widetilde{\theta}_{2,n}^{\text{Euler}}$ increases, since the estimator $\widetilde{\theta}_{2,n}^{\text{Euler}}$ keeps more jumps that induce a stronger bias. On the other hand, the bias of the estimator $\widetilde{\theta}_{2,n}$ remains small (see  line 4 of Table \ref{T:stable}), as the correction term in \eqref{E:theta_tilde_theat_Euler} increases with $c$.
\begin{table}[ht]
\begin{center}
\begin{tabular}{|c|c|c|c|c|}
\hline 
& & $\alpha=0.1$ & $\alpha=0.3$ & $\alpha=0.5$  \\  
\hline 
\rule{0pt}{1.3em}
c=1& $\widetilde{\theta}_{2,n}$ & 1.99 (0.0315) &  1.98 (0.0340) & 1.97 (0.0367) \\ 
 \cline{2-5}
\rule{0pt}{1.3em} 
& $\widetilde{\theta}_{2,n}^{\text{Euler}}$& 2.20 (0.0315) & 2.37 (0.0340)  & 2.76 (0.0367) \\ 
\hline 
\rule{0pt}{1.3em}
c=1.5& $\widetilde{\theta}_{2,n}$ & 1.97 (0.0340) &  1.96 (0.0363) & 1.94 (0.0397) \\ 
 \cline{2-5}
\rule{0pt}{1.3em} 
& $\widetilde{\theta}_{2,n}^{\text{Euler}}$& 2.28 (0.0340) & 2.48 (0.0363)  & 2.90 (0.0397) \\ 
\hline
\end{tabular} 
\caption{Mean (std) for the estimation of $\theta_2=2$ \label{T:stable}} 
\end{center}
\end{table}

\subsection{Conclusion and perspectives for practical applications}
In this paper, we have shown that it is theoretically possible to estimate the drift parameter efficiently under the sole condition of a sampling step converging to zero. However, the contrast function relies on the quantity $m_{h,\theta}(x)$ which is usually not explicit. For practical implementation, the question of approximation of $m_{h,\theta}(x)$ is crucial, and one also has face the question of choosing the threshold level, characterized here by $c$, $\beta$ and $\varphi$. On contrary to more conventional threshold methods, it appears here that the estimation quality seems less sensitive to choice of the threshold level, as the quantity $m_{h,\theta}(x)$ depends by construction on this threshold level and may compensate for too large threshold.
On the other hand, the quantity $m_{h,\theta}(x)$ can be numerically very far from the approximation derived form the Euler scheme approximation. This can be seen in the example of Section \ref{Ss:Num_infinite}, where the correction term of the estimator is, on this finite sample example, essentially of the same magnitude as the estimated quantity.  A perspective, in the situation of infinite jump activity, would be to numerically approximate the function $x\mapsto m_{h,\theta}(x)$, using for instance a Monte Carlo approach, and provide more precise corrections than the explicit correction used in Section \ref{Ss:Num_infinite}.

In the specific situation of finite activity, we proposed an explicit approximation of $m_{h,\theta}(x)$ with arbitrary order. This approximation is the same one as Kessler's approximation in absence of jumps, and it relies on the choice of oscillating truncation functions. A crucial point in the proof of the expansion of $m_{h,\theta}(x)$ given in Proposition \ref{prop: dl m intensita finita} is that the support of the truncation function 
$\varphi_{c \Delta_{n,i}^\beta}$
is small compared to the typical scale where the density of the jumps law varies. 
However, our construction of oscillating function is such that the support of $\varphi=\varphi_d^{(l)}$ tends to be larger as the number of oscillations $l$ is larger, which yields to restrictions for the choice of $l$ on finite sample. Moreover, the truncation function takes large negative values as well, which makes the minimization of the contrast function unstable if the parameter set is too large. Perspective for further works would be to extend  Proposition \ref{prop: dl m intensita finita} for a non oscillating function $\varphi$. We expect that the resulting asymptotic expansion would involve additional terms related to the quantities $\int u^k \varphi (u) d u$.

\section{Limit theorems}\label{S:Limit}
The asymptotic properties of estimators are deduced from the asymptotic behavior of the contrast function. We therefore prepare some limit theorems for triangular arrays of the data, that we will prove in the Appendix. \\

\begin{proposition}
Suppose that Assumptions 1 to 4 hold, $\Delta_n \rightarrow 0$ and $t_n \rightarrow \infty$. \\
Moreover suppose that $f$ is a differentiable function $\mathbb{R} \times \Theta \rightarrow \mathbb{R}$ such that $|f(x,\theta)|\le c(1 + |x|)^c$, $|\partial_xf(x, \theta)| \le c(1 + |x|)^c$ and $|\partial_\theta f(x, \theta)| \le c(1+ |x|)^c$. \\
Then, $ x \mapsto f(x, \theta)$ is a $\pi$-integrable function for any $\theta \in \Theta$ and the following convergence result holds as $n \rightarrow \infty$:
\\
(i) $\sup_{\theta \in \Theta} |\frac{1}{t_n} \sum_{i=0}^{n-1}\Delta_{n,i}f(X_{t_i}, \theta)1_{\left \{|X_{t_i}| \le \Delta_{n,i}^{-k} \right \} } - \int_\mathbb{R} f(x, \theta) \pi(dx)| \xrightarrow{\mathbb{P}} 0,$ \\
(ii) $\sup_{\theta \in \Theta} |\frac{1}{t_n} \sum_{i=0}^{n-1}\Delta_{n,i}f(X_{t_i}, \theta)\varphi_{\Delta_{n,i}^\beta}(X_{t_{i+1}} - X_{t_i})1_{\left \{|X_{t_i}| \le \Delta_{n,i}^{-k} \right \} } - \int_\mathbb{R} f(x, \theta) \pi(dx)| \xrightarrow{\mathbb{P}} 0.$
\label{prop: LT1} 
\end{proposition}

The next proposition will be used in order to prove the consistency. \\
First, we prepare some notations. We define
\begin{equation}
\zeta_i := \int_{t_i}^{t_{i+1}} a(X_s)dW_s + \int_{t_i}^{t_{i+1}} \int_{\mathbb{R} \backslash \left \{0 \right \}} \gamma(X_{s^-})z \tilde{\mu}(ds,dz) +  \Delta_{n,i}\, \int_{\mathbb{R} \backslash \left \{0 \right \}} z \, \gamma(X_{t_i})\,[1 - \varphi_{\Delta_{n,i}^\beta}(\gamma(X_{t_i})z)]\,F(z)dz.
\label{eq: def zeta} 
\end{equation}
We now observe that using the dynamic of the process $X$ and the development (\ref{eq: dl m}) of $m$ we get
\begin{equation}
X_{t_{i+1}} - m_\theta(X_{t_i}) + R(\theta,\Delta_{n,i}^{2 - 2 \beta}, X_{t_i}) =  ( \int_{t_i}^{t_{i+1}} b(X_s, \theta_0) ds -  \Delta_{n,i} b (X_{t_i}, \theta) ) + \zeta_i,
\label{eq: link betwen m and zeta} 
\end{equation}
if $\alpha < 1$ and the same but with the different rest term $R(\theta,\Delta_{n,i}^{2 - 3 \beta}, X_{t_i})$ if $\alpha \ge 1$. From the choice that we have made on $\alpha$ and $\beta$ in Theorems \ref{th: dL mtheta alpha <1} and \ref{th: dL mtheta alpha  > 1}, the exponent on $\Delta_{n,i}$ in the rest function is always more than $1$. Hence, from now on, we will call it simply $R(\theta, \Delta_{n,i}^{1 + \delta}, X_{t_i})$, with $\delta > 0$.   
That is the reason why we choose such a definition for $\zeta_i$. \\
\begin{proposition}
Suppose that Assumptions 1 to 4 and $A_\beta$ hold, $\Delta_n \rightarrow 0$ and $t_n \rightarrow \infty$ and, $\forall i \in \left \{ 0, ... , n-1 \right \} $,  $f_{i,n}$: $\mathbb{R} \times \Theta \rightarrow \mathbb{R}$. Moreover we suppose that $\exists c$: $|f_{i,n}(x, \theta)| \le c (1 + |x|^c)$ $\forall i, n$. \\
Then, $\forall \theta \in \Theta$, 
$$\frac{1}{t_n} \sum_{i= 0}^{n-1} f_{i,n}(X_{t_i}, \theta) \, \zeta_i \, \varphi_{\Delta_{n,i}^\beta}(X_{t_{i+1}} - X_{t_i})1_{\left \{|X_{t_i}| \le \Delta_{n,i}^{-k} \right \} } \xrightarrow{\mathbb{P}} 0.$$
\label{prop: LT2} 
\end{proposition}
The proof relies on the following lemma: \\
\begin{lemma}
Suppose that Assumptions 1 to 4 and $A_\beta$ hold. Then
\begin{equation}
1. \quad \mathbb{E}[\zeta_i \varphi_{\Delta_{n,i}^\beta}(X_{t_{i+1}} - X_{t_i})1_{\left \{|X_{t_i}| \le \Delta_{n,i}^{-k} \right \} }| \mathcal{F}_{t_i}] = R(\theta_0, \Delta_{n,i}^{(1 + \delta) \land \frac{3}{2}}, X_{t_i} ),
\label{eq: convergence expected value prop 2}
\end{equation}
\begin{equation}
2. \quad \mathbb{E}[\zeta_i^2 \, \varphi^2_{\Delta_{n,i}^\beta}(X_{t_{i+1}} - X_{t_i})1_{\left \{|X_{t_i}| \le \Delta_{n,i}^{-k} \right \} }| \mathcal{F}_{t_i}] = R(\theta_0, \Delta_{n,i}, X_{t_i} ),
\label{eq: conv squared prop 2} 
\end{equation}
and
\begin{equation}
3. \quad \mathbb{E}[(X_{t_{i+1}} - m_{\theta_0}(X_{t_i}))^2 \, \varphi^2_{\Delta_{n,i}^\beta}(X_{t_{i+1}} - X_{t_i})1_{\left \{|X_{t_i}| \le \Delta_{n,i}^{-k} \right \} }| \mathcal{F}_{t_i}] = R(\theta_0, \Delta_{n,i}, X_{t_i} ),
\label{eq: x - m prop 2} 
\end{equation}
where $(\mathcal{F}_s)_s$ is the filtration defined in Lemma \ref{lemma: Moment inequalities} and $\delta$ is positive as defined above. 
\label{lemma: conditional expected value zeta} 
\end{lemma}
We now give an asymptotic normality result: \\
\begin{proposition}
Suppose that Assumptions 1 to 4 and $A_\beta$ hold, $\Delta_n \rightarrow 0$, $t_n \rightarrow \infty$. \\
Moreover suppose that $f$ is a continuous function $\Theta \times \mathbb{R} \rightarrow \mathbb{R}$ that satisfies conditions in Proposition \ref{prop: LT1}. Then for all $\theta$
$$\frac{1}{\sqrt[]{t_n }}\sum_{i=0}^{n-1} (X_{t_{i+1}} - m_{\theta_0}(X_{t_i})) f(X_{t_i}, \theta) \,  \varphi_{\Delta_{n,i}^\beta}(X_{t_{i+1}} - X_{t_i})1_{\left \{|X_{t_i}| \le \Delta_{n,i}^{-k} \right \} } \xrightarrow{\mathcal{L}} N(0, \int_\mathbb{R}f^2(x,\theta)\,a^2(x) \, \pi(dx)).$$
\label{prop: LT3} 
\end{proposition}

\section{Proof of main results}\label{S:Proof_main}
We state a proposition that will be used repeatedly in the proof of Theorems \ref{th: dl den alpha <1},\ref{th: dL mtheta alpha <1},\ref{th: dl den alpha >1} and \ref{th: dL mtheta alpha  > 1}. This proposition is an estimation of some expectations related to the event that increments of the process $X$ lies where $\varphi_{\Delta_{n,i}}$, that is the smooth version of the indicator function, becomes singular for $\Delta_n \rightarrow 0$. The proof is postponed to Section \ref{S:Proof_prop moche}. \\
\begin{proposition}
Suppose that Assumptions 1 to 4 and $A_\beta$ hold.
Moreover suppose that $h:\mathbb{R} \times \Theta \longrightarrow \mathbb{R}$ is a function for which $\exists c > 0 : \sup_{\theta \in \Theta}|h(x,\theta)| \le c(1 + |x|)^c$. Then $\forall k \ge 1$ $\forall \epsilon > 0$, we have
$$ \sup_{u \in [t_i, t_{i+1}]}\mathbb{E}[|h(X_{u}^\theta, \theta)||\varphi_{\Delta_{n,i}^\beta}^{(k)}(X_{u}^\theta - X^\theta_{t_i})||X^\theta_{t_i} = x] = R(\theta,\Delta_{n,i}^{1 - \alpha\beta - \epsilon}, x).$$
with $\alpha$ and $\beta$ given in the third point of Assumption 4 and Definition 1. We have used $\varphi_{\Delta_{n,i}^\beta}^{(k)}(y)$ in order to denote $\varphi^{(k)}(\frac{y}{\Delta_{n,i}^\beta})\Delta_{n,i}^{-\beta}$.
\label{prop: truc moche h} 
\end{proposition}
Proposition \ref{prop: truc moche h} is a consequence of the following more general proposition: \\
\begin{proposition}
Suppose that Assumption 1 to 4 and $A_\beta$ hold. For $c >0$, we define $$\mathcal{Z}_{h,c, p}:= \left \{ Z = (Z_\theta)_{\theta \in \Theta} \mbox{family of random variables } \mathcal{F}_h \mbox{ measurable such that } \sup_{\theta \in \Theta} \mathbb{E}[|Z_\theta|^p|X^\theta_0 = x] \le c(1 + |x|^c)  \right \}.$$ Then $\forall k \ge 1$ we have, $\forall \epsilon \ge \frac{1}{p}$,
$$\sup_{Z \in \mathcal{Z}_{h,c, p}}\mathbb{E}[|Z_\theta||\varphi_{h^\beta}^{(k)}(X_h^\theta - X_0^\theta)||X_0^\theta = x] \le R(\theta, h^{(1 - \alpha \beta)(1 - \epsilon)}, x) ,$$
where $R(\theta, h^\delta, x)$ denotes any function such that $\exists c > 0$: $|R(\theta, h^\delta, x)| \le c(1 + |x|^c) h^\delta$ uniformly in $\theta $, with $c$ independent of $h$.
\label{prop: truc moche z} 
\end{proposition}

\subsection{Development of $m_{\theta, \Delta_{n,i}}(x)$}
In order to study the asymptotic behavior of the contrast function we need some explicit approximation of $m_{\theta, \Delta_{n,i}}$. We study the asymptotic expansion of $m_{\theta, \Delta_{n,i}}(x)$ as $\Delta_{n,i} \rightarrow 0$. The main tools is the iteration of the Dynkin's formula that provides us the following expansion for every function $f$: $\mathbb{R} \rightarrow \mathbb{R}$ such that $f$ is in $C^{2(k+1)}$:
\begin{equation}
\mathbb{E}[f(X^\theta_{t_{i+1}})| X^\theta_{t_i} = x] = \sum_{j=0}^k \frac{\Delta_{n,i}^j}{j!}A^jf(x) + \int_{t_i}^{t_{i+1}} \int_{t_i}^{u_1} ... \int_{t_i}^{u_k} \mathbb{E}[A^{k+1}f(X^\theta_{u_{k+1}})| X^\theta_{t_i} = x]\, du_{k+1} ... du_2\, du_1 
\label{eq: Dynkin formula beginning} 
\end{equation}
where $A$ denotes the generator of the diffusion. $A$ is the sum of the continuous and discrete part: $A: = A_c + A_d$, with
$$A_cf(x)= \frac{1}{2}a^2(x)f''(x) + b(x, \theta)f'(x)$$
and
$$A_df(x)= \int_{\mathbb{R}} (f(x + \gamma(x)z) - f(x) - z\gamma(x)f'(x))\, F(z) dz. $$
We set $A^0 = Id$.

\subsubsection{Proof of Theorem \ref{th: dl den alpha <1}:}
\begin{proof}
We have to show (\ref{eq: dl phi}). Using the formula \eqref{eq: Dynkin formula beginning} in the case $k =1$, we get
$$\mathbb{E}[\varphi_{\Delta_{n,i}^\beta}(X_{t_{i+1}}^\theta - X_{t_i}^\theta)|X_{t_i}^\theta = x] =$$
\begin{equation}
= A^0\varphi_{\Delta_{n,i}^\beta}(0) + (t_{i+1} - t_i)A\varphi_{\Delta_{n,i}^\beta}(0) + \int_{t_i}^{t_{i +1}} \int_{t_i}^{u_1} \mathbb{E}[A^2\varphi_{\Delta_{n,i}^\beta}(X^\theta_{u_2})| X^\theta_{t_i} = x]du_2 du_1.
\label{eq: Dynkin on varphi} 
\end{equation}
We have defined $\varphi$ as a smooth version of the indicator function, it means that in a neighborhood of $0$ its value is $1$ and so that $\varphi^{(k)}(0) =0$ for each $k \ge 1$. \\
We denote $f_{i,n}(y):= \varphi_{\Delta_{n,i}^\beta}(y - x) = \varphi(\frac{y-x}{\Delta_{n,i}^\beta})$, with $\beta \in (0, \frac{1}{2})$. 
By the building, $f_{i,n}(x)=1$ and $f_{i,n}^{(k)}(x) =0$ for each $k \ge 1$, so we get $A_cf_{i,n}(x)= 0$ and $A_df_{i,n}(x)=\int_{\mathbb{R} \backslash \left \{0 \right \}} [f_{i,n}(x + \gamma(x)z) - 1]\, F(z) dz $. \\
In the sequel the constant $c > 0$ may change from line to line.\\
From the definition of $f_{i,n}$ and the fact that $\varphi = 1$ on $[-1, 1]$ we have that $f_{i,n}(y)= 1$ for $|y-x| \le \Delta_{n,i}^\beta$. Thus
$$|A_df_{i,n}(x)| \le \, 2 {\left \| \varphi_{\Delta_{n,i}^\beta}\right \|}_{\infty} \int_{\left \{ z : |z \gamma(x)| \ge \Delta_{n,i}^\beta \right \}} F(z) dz \le $$
$$\le  2 {\left \| \varphi_{\Delta_{n,i}^\beta}\right \|}_{\infty} \int_{\left \{ z : |z| \ge \frac{\Delta_{n,i}^\beta}{|\gamma(x)|} \right \}} |z|^{-1 - \alpha} dz \le c {\left \| \varphi_{\Delta_{n,i}^\beta}\right \|}_{\infty} |\gamma(x)|^{\alpha} \Delta_{n,i}^{-\beta\alpha}= R(\theta, \Delta_{n,i}^{- \alpha \beta}, x),  $$
where the second inequality follows from point 3 of Assumption 4.
Substituting in (\ref{eq: Dynkin on varphi}) we get
\begin{equation}
\mathbb{E}[\varphi_{\Delta_{n,i}^\beta}(X_{t_{i+1}}^\theta - X_{t_i}^\theta)|X_{t_i}^\theta = x] = 1 + \Delta_{n,i}R(\theta, \Delta_{n,i}^{-\alpha\beta}, x) + \int_{t_i}^{t_{i +1}} \int_{t_i}^{u_1} \mathbb{E}[A^2\varphi_{\Delta_{n,i}^\beta}(X^\theta_{u_2})| X^\theta_{t_i} = x]du_2 du_1 .
\label{eq: dynkin varphi avec A} 
\end{equation}
In order to prove (\ref{eq: dl phi}), we want to show that the last term is negligible. \\
We consider the generator's decomposition in discrete and continuous part $A = A_c + A_d$ that yields: $A^2f_{i,n}(y)=(A_c^2f_{i,n})(y) + A_c(A_df_{i,n})(y) + A_d(A_cf_{i,n})(y) + (A_d^2f_{i,n})(y)$. \\
We observe that we can write $(A_c^2f_{i,n})(y)$  as
 $$\sum_{j=1}^4 \Delta_{n,i}^{-\beta j} h_j(y, \theta) \varphi ^{(j)}_{\Delta_{n,i}^\beta} (y - x),$$
where $\varphi ^{(j)}_{\Delta_{n,i}^\beta} (y - x) = \varphi^{(j)}(\frac{(y-x)}{\Delta_{n,i}^\beta})$.
For each $j \in \left \{ 1, 2, 3, 4 \right \}$, $h_j$ is a function of $a$, $b$ and their derivatives up to second order: $h_1= \frac{1}{2} a^2b^{''} + bb'$, $h_2 = \frac{1}{2} a^2(a')^2 + \frac{1}{2} a^3a'' + a^2b' + aa'b + b^2 $, $h_3= a^3a' + a^2b$ and $h_4 = \frac{1}{4} a^4$.  \\
Using the Proposition \ref{prop: truc moche h} we get that $\sup_{u_2 \in [t_i, t_{i+1}]} |\mathbb{E}[(A_c^2f_{i,n})(X^\theta_{u_2})|X^\theta_{t_i} = x]|$ is upper bounded by 
$$\sup_{u_2 \in [t_i, t_{i+1}]} | \sum_{j=1}^4 \Delta_{n,i}^{-\beta j} \, \mathbb{E}[h_j(X^\theta_{u_2}, \theta) \varphi ^{(j)}_{\Delta_{n,i}^\beta} (X^\theta_{u_2} - X^\theta_{t_i})|X^\theta_{t_i} = x] | =$$
$$= | \sum_{j=1}^4 \Delta_{n,i}^{-\beta j} R(\theta,\Delta_{n,i}^{1 - \alpha\beta - \epsilon}, x) | = R(\theta,\Delta_{n,i}^{1 - \alpha\beta - \epsilon - 4\beta}, x).$$
Let us now consider $A_c(A_df_{i,n})(y)$. Substituting the definition of $A_df_{i,n}$ we get
\begin{equation}
A_c(A_df_{i,n})(y) = A_c(\int_\mathbb{R} g_n(\cdot, z) F(z)dz) (y),
\label{eq: Ac(Adf)} 
\end{equation}
where
\begin{equation}
g_n(y, z): = \varphi_{\Delta_{n,i}^\beta}(y - x + z \gamma(y)) - \varphi_{\Delta_{n,i}^\beta} (y -x) - \Delta_{n,i}^{-\beta}\varphi'_{\Delta_{n,i}^\beta}(y -x)\gamma(y) z 
\label{eq: definition g_n} 
\end{equation}
and where the notation used means that we are applying the differential operator $A_c$ with respect to the variable represented with a dot.
In order to estimate it we observe that
\begin{equation}
|g_n(y, z)| \le \Delta_{n,i}^{-\beta} \left \| \varphi' \right \|_\infty |z||\gamma(y)|,
\label{eq: estimation gn} 
\end{equation}
\begin{equation}
|\frac{\partial}{\partial y} g_n(y, z)| \le \Delta_{n,i}^{-2\beta} P(y)|z| \qquad \mbox{and}
\label{eq: estimation derivé gn} 
\end{equation}
\begin{equation}
|\frac{\partial ^2}{\partial y^2} g_n(y, z)| \le \Delta_{n,i}^{-3\beta} P(y)(|z|+ |z|^2);
\label{eq: estimation derivé deuxieme gn} 
\end{equation}
where $P(y)$ is a polynomial function in $y$, that may change from line to line. \\
Since the functions $a^2$ and $b$ have polynomial growth, we obtain
\begin{equation}
|A_cg_n(\cdot,z)(y)| \le \Delta_{n,i}^{-3\beta} P(y)(|z|+ |z|^2).
\label{eq: estimation Acgn} 
\end{equation}
Using the dominated convergence theorem we get
$$A_c(\int_\mathbb{R} g_n(\cdot,z) F(z)dz)_{(y)} = \int_\mathbb{R} (A_cg_n)(\cdot,z)_{(y)} F(z)dz,$$
Therefore, using (\ref{eq: estimation Acgn}),
$$|A_c(\int_\mathbb{R} g_n(\cdot,z) F(z)dz)_{(y)}| \le \Delta_{n,i}^{-3\beta}P(y)\int_\mathbb{R}(|z|+ |z|^2)F(z)dz,$$
that is upper bounded by $c \Delta_{n,i}^{-3\beta}P(y) $ since $\alpha$ is less than $1$. It turns
$$\sup_{u_2 \in [t_i, t_{i+1}]}|\mathbb{E}[(A_c(A_df_{i,n}))(X^\theta_{u_2})|X^\theta_{t_i} = x]| \le \sup_{u_2 \in [t_i, t_{i+1}]} |\mathbb{E}[c \Delta_{n,i}^{-3\beta}P(X^\theta_{u_2})|X^\theta_{t_i} = x]| =R(\theta,\Delta_{n,i}^{-3\beta}, x) $$
where, in the last equality, we have used the third point of Lemma \ref{lemma: Moment inequalities}. \\
We reason in the same way on $A_d(A_cf_{i,n})(y)$, which is equal to
\begin{equation}
\int_\mathbb{R} [A_cf_{i,n}(y + z \gamma(y)) - A_cf_{i,n}(y) - z \gamma(y)(A_cf_{i,n})'(y)]F(z)dz.
\label{eq: AdAcf} 
\end{equation}
It is, in module, upper bounded by
\begin{equation}
c \int_0^1 \int_\mathbb{R}[|(A_cf_{i,n})'(y + z \gamma(y)s)|+ |(A_cf_{i,n})'(y)|]|z||\gamma(y)|F(z)ds\,dz .
\label{eq: adacf up bound} 
\end{equation}
We observe that, $\forall y'$, $(A_cf_{i,n})'(y')= (b'f_{i,n}'+ bf_{i,n}'' + aa'f_{i,n}'' + \frac{1}{2}a^2f_{i,n}''')(y')$. \\
By the fact that $|\frac{\partial ^j}{\partial y^j} \varphi_{\Delta_{n,i}^\beta}(y)| \le c\Delta_{n,i}^{-\beta j}$ for $j = 1, 2, 3$ and recalling $f_{i,n}(y)= \varphi_{\Delta_{n,i}^\beta}(y - x) $, we get that
\begin{equation}
|(A_cf_{i,n})'(y')| \le c\,P(y')\Delta_{n,i}^{-3\beta},
\label{eq: estimation (Acf)'} 
\end{equation}
where we have used that $b$ and $a^2$ have polynomial growth. We obtain that \eqref{eq: adacf up bound} is upper bounded by
$$\Delta_{n,i}^{-3\beta} \int_0^1 \int_\mathbb{R} (P(y + z\gamma(y)s)+ P(y)) |z||\gamma(y)|F(z)ds\,dz \le \Delta_{n,i}^{-3\beta}\int_\mathbb{R} P(y)P(z)|z|F(z) dz \le c\Delta_{n,i}^{-3\beta} P(y), $$
where we have used the first point of Assumptions 3 and the third of Assumption 4, with $\alpha \in (0,1)$, in order to get $\int_\mathbb{R} P(z)|z|F(z) dz \le \infty$. \\
Considering the controls \eqref{eq: adacf up bound} and \eqref{eq: estimation (Acf)'} on \eqref{eq: AdAcf} it yields, using again the third point of Lemma \ref{lemma: Moment inequalities}, 
$$\sup_{u_2 \in [t_i, t_{i+1}]} |\mathbb{E}[(A_d(A_cf_{i,n}))(X^\theta_{u_2})|X^\theta_{t_i} = x] |= R(\theta,\Delta_{n,i}^{-3\beta}, x).$$
To conclude, we consider $A_d(A_df_{i,n})(y)$:
\begin{equation}
\int_\mathbb{R} [A_df_{i,n}(y + z \gamma(y)) - A_df_{i,n}(y) - z \gamma(y)(A_df_{i,n})'(y) ] F(z)dz.
\label{eq: AdAdf} 
\end{equation}
Again, (\ref{eq: AdAdf}) is, in module, upper bounded by
\begin{equation}
c \int_0^1 \int_\mathbb{R}[|(A_df_{i,n})'(y - x + z \gamma(y)s)|+ |(A_df_{i,n})'(y)|]|z||\gamma(y)|F(z)ds\,dz
\label{eq: estimation AdAdf} 
\end{equation}
But
\begin{equation}
A_df_{i,n}(y')= \int_\mathbb{R} g_n(y',z) F(z)dz,
\label{eq: Ad come int gin} 
\end{equation}
with $g_n(y', z)$ given in \eqref{eq: definition g_n}
Using control equation \eqref{eq: estimation derivé gn} and dominated convergence theorem, we get that its derivative is upper bounded by $c\Delta^{-2\beta}_{n,i}P(y')$. \\
Using also (\ref{eq: AdAdf}) and (\ref{eq: estimation AdAdf}),
$$|A_d^2f_{i,n}(y)| \le \Delta_{n,i}^{-2\beta} P(y) \int_\mathbb{R}|z| F(z) dz $$ and it turns, using third point of Lemma \ref{lemma: Moment inequalities},
$$\sup_{u_2 \in [t_i, t_{i+1}]} | \mathbb{E}[(A_d^2f_{i,n})(X^\theta_{u_2})|X^\theta_{t_i} = x] | = R(\theta,\Delta_n^{-2\beta}, x).$$
By the decomposition of the generator in $A_c$ and $A_d$ we get 
$$ \sup_{u_2 \in [t_i, t_{i+1}]} | \mathbb{E}[A^2f_{i,n}(X^\theta_{u_2})| X^\theta_{t_i} = x] | = R(\theta,\Delta_{n,i}^{1 - \alpha\beta - 4\beta - \epsilon}, x) + R(\theta, \Delta_{n,i}^{-3\beta}, x) + R(\theta, \Delta_{n,i}^{-2\beta}, x),$$
with $\alpha \in (0,1)$ and $\beta \in (0, \frac{1}{2})$, so it is 
$ R(\theta, \Delta_{n,i}^{-3\beta}, x) $, since the other $R$ functions are always negligible compared to it. \\
Using (\ref{eq: dynkin varphi avec A}) we get
$$\mathbb{E}[\varphi_{\Delta_{n,i}^\beta}(X_{t_{i+1}}^\theta - X_{t_i}^\theta)|X_{t_i}^\theta = x] = 1 + \Delta_{n,i}R(\theta, \Delta_{n,i}^{-\alpha\beta}, x) + \frac{\Delta_{n,i}^2}{2}R(\theta, \Delta_{n,i}^{-3\beta}, x). $$
We deduce, using the definition of $\Delta_{n,i}$ and \eqref{propriety power R}, that it is
$$1 + R(\theta, \Delta_{n,i}^{1-\alpha\beta}, x) + R(\theta, \Delta_{n,i}^{2-3\beta}, x)= 1 + R(\theta, \Delta_{n,i}^{(1-\alpha\beta) \land (2 - 3 \beta)}, x), $$
as we wanted.
\end{proof}

\subsection{Proof of Theorem \ref{th: dl den alpha >1}}

\begin{proof}
Let $\alpha$ now be in $[1,2)$. In the sequel we skip the study of the case $\alpha =1$ for simplicity, in order to avoid the appearance of logarithmic functions. However, such a specific case is embedded in the case $\alpha > 1$ by taking $\alpha = 1 + \epsilon$ with a choice of $\epsilon >0$ arbitrarily small.\\
Using again Dynkin formula, we have that \eqref{eq: dynkin varphi avec A} is still true. Considering the generator's decomposition, we act like in the case where $\alpha$ is less than $1$ to get that
\begin{equation}
\sup_{u_2 \in [t_i, t_{i+1}]} |\mathbb{E}[(A_c^2f_{i,n})(X^\theta_{u_2})|X^\theta_{t_i} = x]| = R(\theta, \Delta_{n,i}^{1 - \alpha \beta - \epsilon - 4 \beta}, x).
\label{eq: AcAcf alpha >1} 
\end{equation}
Concerning $A_c(A_df_{i,n})(y)$, we use \eqref{eq: Ac(Adf)} with $g_n$ defined in \eqref{eq: definition g_n}. Using Taylor development to the second order we get
\begin{equation}
|g_n(y,z)| \le \left \| \varphi_{\Delta_{n,i}^\beta}'' \right \|_\infty |\Delta_{n,i}|^{- 2 \beta} \frac{|z|^2 \gamma(y)^2}{2}. 
\label{eq: estim gn alpha >1} 
\end{equation}
In the same way we get the following two estimations:
$$
|\frac{\partial }{\partial y} g_n(y,z)| \le |\Delta_{n,i}|^{- 2 \beta}\left \| \varphi_{\Delta_{n,i}^\beta}'' \right \|_\infty |\gamma(y)|| \gamma'(y)||z|^2 +  |\frac{\Delta_{n,i}|^{- 3 \beta}}{2} \left \| \varphi_{\Delta_{n,i}^\beta}''' \right \|_\infty |z|^2 \gamma^2(y)|1+\gamma'(y)z |,$$
\begin{equation}
|\frac{\partial^2 }{\partial y^2} g_n(y,z)| \le |\Delta_{n,i}|^{- 2 \beta}|z|^2P(y) + |\Delta_{n,i}|^{- 3 \beta}| P(y)(|z|^2 + |z|^3) + |\Delta_{n,i}|^{- 4 \beta}P(y)(|z|^2 + |z|^3).
\label{eq: deriv secondq gn alpha >1} 
\end{equation}
Since $a^2$ and $b$ have polynomial growth, \eqref{eq: deriv secondq gn alpha >1} provides us an estimation on $|A_cg_n(\cdot,z)(y)|$. Using dominated convergence theorem, \eqref{eq: Ac(Adf)}, the estimation of $|A_cg_n(\cdot,z)(y)|$ obtained from \eqref{eq: deriv secondq gn alpha >1} and the fact that $\int_\mathbb{R} (|z|^2 + |z|^3) F(z) dz < \infty$, we get
\begin{equation}
\sup_{u_2 \in [t_i, t_{i+1}]} |\mathbb{E}[(A_c A_d f_{i,n})(X^\theta_{u_2})|X^\theta_{t_i} = x]| = R(\theta, \Delta_{n,i}^{- 2 \beta}, x)+ R(\theta, \Delta_{n,i}^{- 3 \beta}, x) +R(\theta, \Delta_{n,i}^{- 4 \beta}, x) = R(\theta, \Delta_{n,i}^{- 4 \beta}, x).
\label{eq: AcAdf alpha >1} 
\end{equation}
We now consider $A_d(A_c f_{i,n})(y)$. Using \eqref{eq: AdAcf} and the development to the second order of the function $A_cf_{i,n}(y + z\gamma(y))$ we obtain
\begin{equation}
|A_d(A_c f_{i,n})(y)| \le c \int_\mathbb{R}\int_0^1 |(A_cf_{i,n})''(y + s \, z \gamma(y))||z|^2 |\gamma^2(y)| F(z) ds dz.
\label{eq: estim AdAcf alpha >1} 
\end{equation}
We observe that $(A_cf_{i,n})''(y') = [b''f_{i,n}' + 2b'f_{i,n}'' + bf_{i,n}''' + (a')^2f_{i,n}'' + a(a''f_{i,n}'' + a' f_{i,n}''') + 2aa'f_{i,n}''' + \frac{1}{2} a^2 f_{i,n}^{(4)}](y')$. By the fact that $|\frac{\partial ^j}{\partial y^j} \varphi_{\Delta_{n,i}^\beta}(y)| \le c\Delta_{n,i}^{-\beta j}$ for $j = 1, 2, 3$ and recalling $f_{i,n}(y)= \varphi_{\Delta_{n,i}^\beta}(y - x) $, we get that
\begin{equation}
|(A_cf_{i,n})''(y')| \le c\,P(y')\Delta_{n,i}^{-4\beta}.
\label{eq: estimation (Acf)'' alpha >1} 
\end{equation}
Using \eqref{eq: estim AdAcf alpha >1} and \eqref{eq: estimation (Acf)'' alpha >1} it yields
\begin{equation}
\sup_{u_2 \in [t_i, t_{i+1}]} |\mathbb{E}[(A_d A_c f_{i,n})(X^\theta_{u_2})|X^\theta_{t_i} = x]| = R(\theta, \Delta_{n,i}^{- 4 \beta}, x).
\label{eq: AdAcf alpha >1} 
\end{equation}
To conclude, we consider $A_dA_df_{i,n}$. Using \eqref{eq: AdAdf} and the development up to the second order we get
$$|A_d(A_d f_{i,n})(y)| \le c \int_\mathbb{R} \int_0^1|(A_df)''(y + s \,z \gamma(y))||z|^2 |\gamma^2(y)| F(z) ds dz.
\label{eq: estim AdAdf alpha >1} $$
We recall that \eqref{eq: Ad come int gin} still holds, with $g_n$ defined in \eqref{eq: definition g_n}. In order to estimate $(A_df)''(y)$ in the case where $\alpha \in [1,2)$ we use therefore \eqref{eq: deriv secondq gn alpha >1} joint with dominated convergence theorem. It provides us
\begin{equation}
\sup_{u_2 \in [t_i, t_{i+1}]} |\mathbb{E}[(A_d A_d f_{i,n})(X^\theta_{u_2})|X^\theta_{t_i} = x]| = R(\theta, \Delta_{n,i}^{- 4 \beta}, x).
\label{eq: AdAdf alpha >1} 
\end{equation}
Using \eqref{eq: AcAcf alpha >1}, \eqref{eq: AcAdf alpha >1}, \eqref{eq: AdAcf alpha >1} and \eqref{eq: AdAdf alpha >1} we put the pieces together and so we obtain
$$ \sup_{u_2 \in [t_i, t_{i+1}]} | \mathbb{E}[A^2f_{i,n}(X^\theta_{u_2})| X^\theta_{t_i} = x] | = R(\theta,\Delta_{n,i}^{1 - \alpha\beta - 4\beta - \epsilon}, x) + R(\theta, \Delta_{n,i}^{-4\beta}, x).$$
We replace it in the Dynkin formula \eqref{eq: dynkin varphi avec A} getting
$$\mathbb{E}[\varphi_{\Delta_{n,i}^\beta}(X_{t_{i+1}}^\theta - X_{t_i}^\theta)|X_{t_i}^\theta = x] = 1 + \Delta_{n,i}R(\theta, \Delta_{n,i}^{-\alpha\beta}, x) + \frac{\Delta_{n,i}^2}{2}R(\theta, \Delta_{n,i}^{(1 - \alpha\beta - 4\beta - \epsilon) \land (-4\beta)}, x).$$
Using the definition of $\Delta_{n,i}$ and \eqref{propriety power R} it is
\begin{equation}
1 + R(\theta, \Delta_{n,i}^{(1 - \alpha\beta)\land (3 - \alpha\beta - 4\beta - \epsilon) \land (2-4\beta)}, x).
\label{eq: theorem 3 final} 
\end{equation}
Since $\epsilon$ is arbitrarily small, for each choice of $\alpha$ and $\beta$ there exists $\epsilon$ such that $3 - \alpha\beta - 4\beta - \epsilon$ is greater than $2-4\beta$ and \eqref{eq: dl phi alpha >1} follows.
\end{proof}

\subsection{Proof of Theorem \ref{th: dL mtheta alpha <1} }

\begin{proof}
We observe that
\begin{equation}
m_{\theta, \Delta_{n,i}} (x) : = \frac{\mathbb{E}[X_{t_{i+1}}^\theta \varphi_{\Delta_{n,i}^\beta}(X_{t_{i+1}}^\theta - X_{t_i}^\theta)|X_{t_i}^\theta = x]}{\mathbb{E}[\varphi_{\Delta_{n,i}^\beta}(X_{t_{i+1}}^\theta - X_{t_i}^\theta)|X_{t_i}^\theta = x]} = x + \frac{\mathbb{E}[g_{i,n}(X_{t_{i+1}}^\theta )|X_{t_i}^\theta = x]}{\mathbb{E}[\varphi_{\Delta_{n,i}^\beta}(X_{t_{i+1}}^\theta - X_{t_i}^\theta)|X_{t_i}^\theta = x]},
\label{eq: reformulation m} 
\end{equation}
with $g_{i,n}(y)= (y - x)\varphi_{\Delta_{n,i}^\beta}(y -x)$. \\ \\
We have already found a development for the denominator of (\ref{eq: reformulation m}) given by (\ref{eq: dl phi}), we use again the Dynkin's formula \eqref{eq: Dynkin formula beginning} for $k=1$ in order to find a development for the numerator.
By the building, $g_{i,n}(x)= 0$, $g_{i,n}'(x) = 1$ and $g_{i,n}''(x) =0$, so we get 
$$A_cg_{i,n}(x) = b(x, \theta)$$ and 
$$A_dg_{i,n}(x)= \int_{\mathbb{R} \backslash \left \{0 \right \}} [g_{i,n}(x+ z \gamma(x)) -z\gamma(x)] F(z) dz = \int_{\mathbb{R} \backslash \left \{0 \right \}} z \gamma(x)[\varphi_{\Delta_{n,i}^\beta}( z \gamma(x)) - 1] F(z) dz $$
where we have used, in the last equality, simply the definition of $g_{i,n}$. \\
Substituting in the Dynkin's formula we get 
$$\mathbb{E}[g_{i,n}(X_{t_{i+1}}^\theta )|X_{t_i}^\theta = x] = \Delta_{n,i}(b(x, \theta) + \int_{\mathbb{R} \backslash \left \{0 \right \}} z \gamma(x)[\varphi_{\Delta_{n,i}^\beta}( z \gamma(x)) - 1] F(z) dz ) + $$
\begin{equation}
+ \int_{t_i}^{t_{i +1}} \int_{t_i}^{u_1} \mathbb{E}[A^2g_{i,n}(X_{u_2})| X_{t_i} = x]du_2 du_1 . 
\label{eq: dynikin theorem 2} 
\end{equation}
In order to show that the last term is negligible, we have to estimate $(A^2g_{i,n})(y)$ using the decomposition in continuous and discrete part of the generator, as we have already done. \\
Since $g_{i,n}(y)= (y-x)\varphi_{\Delta_{n,i}^\beta}(y-x)$, we have 
$$g_{i,n}^{(h)}(y)= \sum_{k=0}^h \binom{h}{k} \frac{\partial ^k}{\partial y^k} (y-x) \frac{\partial ^{h-k}}{\partial y^{h-k}}(\varphi_{\Delta_{n,i}^\beta} (y-x)),$$
with $\binom{h}{k}$ binomial coefficients. So we get, observing that the derivatives of $(y-x)$ after the second order are zero, the following useful control for $h \ge 1$:
\begin{equation}
|g_{i,n}^{(h)}(y)| \le |\varphi_{\Delta_{n,i}^\beta}^{(h)} (y-x)| \Delta_{n,i}^{-\beta h} |y-x| + 
|\varphi_{\Delta_{n,i}^\beta}^{(h-1)} (y-x) |\Delta_{n,i}^{-\beta (h -1)}|h|.
\label{eq: estimation derivatives g} 
\end{equation}
By the definition of $\varphi$ as a smooth version of the indicator function, we know that it exists $c > 0$ such that if $\frac{|y-x|}{\Delta_{n,i}^\beta} > c$, then $\varphi$ and its derivatives are zero when evaluated at the point $\frac{(y-x)}{\Delta_{n,i}^\beta}$. \\
So we can say that $|\varphi_{\Delta_{n,i}^\beta}^{(h)} (y-x)||y-x| \le c |\varphi_{\Delta_{n,i}^\beta}^{(h)} (y-x)| \Delta_{n,i}^\beta $ 
and consequently
\begin{equation}
|g_{i,n}^{(h)}(y)| \le c |\varphi_{\Delta_{n,i}^\beta}^{(h)}(y-x)| \Delta_{n,i}^{-\beta (h -1)} +c |\varphi_{\Delta_{n,i}^\beta}^{(h-1)} (y-x)| \Delta_{n,i}^{-\beta (h -1)}. 
\label{eq: estim 2 derivatives g} 
\end{equation}
Reasoning as in the proof of Theorem \ref{th: dl den alpha <1}, we start with $(A^2_cg_{i,n})(y)$ and we get that it is $\sum_{j=1}^4 h_j(y,\theta) g_{i,n}^{(j)}(y)$ where again, for each $j \in \left \{ 1, 2, 3, 4 \right \}$, $h_j$ is a function of $a$, $b$ and their derivatives up to second order. \\
We substitute in $\mathbb{E}[(A^2_cg_{i,n})(X^\theta_{u_2})| X^\theta_{t_i} = x]$, getting
$\sum_{j=1}^4 \mathbb{E}[h_j(X^\theta_{u_2},\theta) g_{i,n}^{(j)}(X^\theta_{u_2})| X^\theta_{t_i} = x]$.
Using the estimation (\ref{eq: estimation derivatives g}) we obtain
$$\sup_{u_2 \in [t_i, t_{i+1}]}|\mathbb{E}[|A^2_cg_{i,n}(X^\theta_{u_2})|| X^\theta_{t_i} = x]| \le$$
$$\le \sup_{u_2 \in [t_i, t_{i+1}]}| \sum_{j=1}^4 c \Delta_{n,i}^{-\beta(j-1)}\mathbb{E}[|h_j(X^\theta_{u_2},\theta)|( |\varphi_{\Delta_{n,i}^\beta}^{(j)}(X^\theta_{u_2} - X^\theta_{t_i})| + |\varphi_{\Delta_{n,i}^\beta}^{(j- 1)}(X^\theta_{u_2} - X^\theta_{t_i})|)| X^\theta_{t_i} = x]|. $$
We observe that we can see $\sup_{u_2 \in [t_i, t_{i+1}]}|\mathbb{E}[|h_1(X^\theta_{u_2},\theta)||\varphi_{\Delta_{n,i}^\beta}(X^\theta_{u_2} - X^\theta_{t_i})|| X^\theta_{t_i} = x]|$ as $R(\theta, \Delta_{n,i}^0, x)\,= \,R(\theta, 1,x)$ and we use the Proposition \ref{prop: truc moche h} on the other terms, getting
\begin{align}
\sup_{u_2 \in [t_i, t_{i+1}]}|\mathbb{E}[|A^2_cg_{i,n}(X^\theta_{u_2})|| X^\theta_{t_i} = x] | & \le R(\theta, \Delta_{n,i}^{1 - \alpha \beta -\epsilon}, x) + R(\theta, 1,x) + \sum_{j=2}^4 c \Delta_{n,i}^{-\beta(j-1)}R(\theta, \Delta_{n,i}^{1 - \alpha \beta -\epsilon}, x) = \nonumber \\
&  = R(\theta, \Delta_{n,i}^{1 - \alpha \beta -\epsilon -3\beta}, x) + R(\theta, 1,x).
\label{eq: A^2c theorem 3} 
\end{align}
Let us now consider $A_c(A_dg_{i,n})(y)$ \\
\begin{equation}
A_c(A_dg_{i,n})(y)= A_c(\int_\mathbb{R} [g_{i,n}(\cdot + z \gamma(\cdot)) - g_{i,n}(\cdot)- z \gamma(\cdot) g_{i,n}'(\cdot)] F(z)dz)_{(y)} .
\label{eq: AcAdg} 
\end{equation}
Let us denote
\begin{equation}
h_{i,n}(y,z):= g_{i,n}(y + z\gamma(y)) - g_{i,n}(y) - z \gamma(y) g_{i,n}'(y).
\label{eq: definition hni} 
\end{equation}
We observe that
\begin{equation}
\frac{\partial h_{i,n}}{\partial y} (y,z) = g_{i,n}'(y + z\gamma(y)) - g_{i,n}'(y) + z\gamma'(y)(g_{i,n}'(y + z\gamma(y))- g_{i,n}'(y)) - z\gamma(y) g_{i,n}''(y), 
\label{eq: estimation derivative hn} 
\end{equation}
$$\frac{\partial^2 h_{i,n}}{\partial y^2} (y,z) = g_{i,n}''(y + z\gamma(y))(1 +z\gamma'(y))^2 + g_{i,n}'(y + z\gamma(y))z\gamma''(y) +$$
\begin{equation}
- g_{i,n}''(y) - g_{i,n}'''(y)\gamma(y)z - 2g_{i,n}''(y)z \gamma'(y) - g_{i,n}'(y)z\gamma''(y).
\label{eq: estim deriv2 hn} 
\end{equation}
Using the estimation (\ref{eq: estim 2 derivatives g}), we have
\begin{equation}
|g_{i,n}'(y)| \le c \qquad
|g_{i,n}''(y)| \le c\Delta_{n,i}^{-\beta} \qquad
|g_{i,n}'''(y)| \le c\Delta_{n,i}^{-2\beta}
\label{eq: estimations g',g'' and g'''} 
\end{equation}
Hence
\begin{equation}
|\frac{\partial h_{i,n}}{\partial y} (y,z) | \le \left \| g_{i,n}'' \right \|_\infty P(y)(|z|+|z|^2) \le (|z|+|z|^2) P(y) \Delta_{n,i}^{-\beta},
\label{eq: estimation derivative hni} 
\end{equation}
and similarly
$$|\frac{\partial^2 h_{i,n}}{\partial y^2} (y,z)| \le \Delta_{n,i}^{-2\beta} P(y)(|z|+ |z|^2 + |z|^3).$$
Since functions $a^2$ and $b$ have polynomial growth, we obtain 
$$|A_ch_{i,n}(y,z)| \le \Delta_{n,i}^{-2\beta}P(y)(|z|+ |z|^2+ |z|^3).$$
Using dominated convergence theorem we get
$$|A_c(\int_\mathbb{R} h_{i,n}(\cdot, z) F(z)dz)_{(y)}| \le \Delta_{n,i}^{-2\beta} P(y) \int_\mathbb{R} (|z|+ |z|^2 +|z|^3) F(z)dz $$
and so, using also the third point of Lemma \ref{lemma: Moment inequalities} and \eqref{eq: AcAdg}, we get
\begin{equation}
\sup_{u_2 \in [t_i, t_{i+1}]}|\mathbb{E}[A_c(A_dg_{n,i})(X^\theta_{u_2})| X^\theta_{t_i} = x] |= R(\theta, \Delta_{n,i}^{- 2 \beta}, x).  
\label{eq: AcAdg finale } 
\end{equation}
We reason on the same way on $A_d(A_cg_{n,i})(y)$: \\
\begin{equation}
A_d(A_cg_{n,i})(y)= \int_\mathbb{R} [A_cg_{n,i}(y + z \gamma(y)) - A_cg_{n,i}(y) - z \gamma(y)(A_cg_{n,i})'(y) ]F(z)dz.
\label{eq: AdAcg} 
\end{equation}
It is, in module, upper bounded by
$$c\int_0^1 \int_\mathbb{R}[|(A_cg_{n,i})'(y + z \gamma(y)s)|+ |(A_cg_{n,i})'(y)|]|z||\gamma(y)|F(z)dsdz.$$
In order to estimate it we observe that, $\forall y'$,
$$(A_cg_{n,i})'(y')= (aa'g_{n,i}'' + \frac{1}{2}a^2 g_{n,i}''' + b'g_{n,i}' + bg_{n,i}'')(y').$$
Using (\ref{eq: estimations g',g'' and g'''}) and the polynomial growth of $a$, $b$ and their derivatives, we get
$$|(A_cg_{n,i})'(y')| \le c + P(y')\Delta_{n,i}^{- \beta} + P(y')\Delta_{n,i}^{- 2\beta} \le P(y')\Delta_{n,i}^{- 2\beta}.$$
It yields
$$c\int_0^1 \int_\mathbb{R}[|(A_cg_{n,i})'(y + z \gamma(y)s)|+ |(A_cg_{n,i})'(y)|]|z||\gamma(y)|F(z)dsdz \le $$
$$ \le \Delta_{n,i}^{-2\beta} \int_0^1 \int_\mathbb{R} (P(y + z\gamma(y)s) + P(y)]|z||\gamma(y)|F(z)dsdz \le \Delta_{n,i}^{-2\beta} \int_\mathbb{R} P(y)P(z)|z| F(z) dz \le c \Delta_{n,i}^{-2\beta}P(y),$$
where we have used the first point of Assumptions 3 and the second of Assumption 4.
Hence $| A_d(A_cg)(y)| \le \Delta_{n,i}^{-2\beta}P(y).$ \\
Taking the expected value and using the third point of Lemma \ref{lemma: Moment inequalities}, we obtain
$$\sup_{u_2 \in [t_i, t_{i+1}]}|\mathbb{E}[A_d(A_cg_{n,i})(X^\theta_{u_2})| X^\theta_{t_i} = x]| =  R(\theta, \Delta_{n,i}^{- 2 \beta}, x).$$
In conclusion, we consider $A^2_d(g_{n,i})(y)$
\begin{equation}
A^2_d(g_{n,i})(y)= \int_\mathbb{R} [A_dg_{n,i}(y + z \gamma(y)) - A_dg_{n,i}(y) - z \gamma(y)(A_dg_{n,i})'(y) ] F(z)dz.
\label{eq: AdAdg} 
\end{equation}
Again it is, in module, upper bounded by
\begin{equation}
c \int_0^1 \int_\mathbb{R}[|(A_dg_{n,i})'(y + z \gamma(y)s)|+ |(A_dg_{n,i})'(y)|]|z||\gamma(y)|F(z)ds\,dz
\label{eq: estimation AdAdg} 
\end{equation}
But
\begin{equation}
A_dg_{n,i}(y')= \int_\mathbb{R} [g_{n,i}(y'+ z \gamma(y'))-g_{n,i}(y') - z \gamma(y')g_{n,i}'(y') ]F(z)dz = \int_\mathbb{R} h_{i,n} (y', z)F(z) dz, 
\label{eq: int hn come Adg} 
\end{equation}
with $h_{i,n}$ defined in \eqref{eq: definition hni}. Using control equation (\ref{eq: estimation derivative hni}) and dominated convergence theorem, we get that \eqref{eq: int hn come Adg} is upper bounded by $P(y')\Delta_{n,i}^{-\beta}$. \\
It follows from (\ref{eq: AdAdg}) and (\ref{eq: estimation AdAdg}) that
$$|A_d^2g_{n,i}(y)| \le c \Delta_{n,i}^{-\beta} P(y) \int_\mathbb{R}P(z) F(z) dz $$ and it turns, using again the third point of Lemma \ref{lemma: Moment inequalities},
$$\sup_{u_2 \in [t_i, t_{i+1}]}|\mathbb{E}[(A_d^2g_{n,i})(X^\theta_{u_2})|X^\theta_{t_i} = x] | = R(\theta,\Delta_{n,i}^{-\beta}, x).$$
Pieces things together we get
$$\sup_{u_2 \in [t_i, t_{i+1}]}|\mathbb{E}[A^2g_{n,i}(X^\theta_{u_2})| X^\theta_{t_i} = x]| =R(\theta, \Delta_{n,i}^{1 - \alpha \beta -\epsilon -3 \beta}, x) + R(\theta, \Delta_{n,i}^{- 2\beta}, x) + R(\theta, \Delta_{n,i}^{-\beta}, x)=$$
$$= R(\theta, \Delta_{n,i}^{-2 \beta}, x),$$
where $R(\theta, \Delta_{n,i}^{1 - \alpha \beta -\epsilon -3 \beta}, x)$ is negligible compared to $R(\theta, \Delta_{n,i}^{-2 \beta}, x)$ because, for each choice of $\alpha$ and $\beta$, we can find an $\epsilon$ arbitrarily small such that $1 - \alpha \beta -\epsilon -\beta$ is more than $0$.
We substitute it in Dynkin's formula and we obtain
$$\mathbb{E}[g_{n,i}(X_{t_{i+1}}^\theta )|X_{t_i}^\theta = x] = $$
\begin{equation}
= \Delta_{n,i}(b(x, \theta) + \int_{\mathbb{R} \backslash \left \{0 \right \}} z \gamma(x)[\varphi_{\Delta_{n,i}^\beta}( z \gamma(x)) - 1] F(z) dz ) + \frac{\Delta_{n,i}^2}{2} R(\theta, \Delta_{n,i}^{-2 \beta}, x). 
\label{eq: Dynkin g} 
\end{equation}
We use the definition of $\Delta_{n,i}$ and the property \eqref{propriety power R} on $R$, then we substitute in (\ref{eq: Dynkin g}) getting (\ref{eq: dl num}).\\
We now want to prove \eqref{eq: dl m}. From the expansion \eqref{eq: dl num} and the property \eqref{eq: definition R} of $R$, there exists $k_0 > 0$ such that for $|x| \le \Delta_{n,i}^{k_0}$, $\mathbb{E}[\varphi_{\Delta_{n,i}^\beta}(X_{t_{i+1}}^\theta -  X_{t_i}^\theta)|X_{t_i}^\theta = x] \ge \frac{1}{2}$ $\forall n, i \le n$: we are avoiding the possibility that the denominator is in the neighborhood of $0$. Using \eqref{eq: reformulation m}, \eqref{eq: Dynkin g} and \eqref{eq: dl phi} we have that
\begin{equation}
m_\theta(x)= x + \frac{\Delta_{n,i}(b(x, \theta) + \int_{\mathbb{R} \backslash \left \{0 \right \}} z \gamma(x)[\varphi_{\Delta_{n,i}^\beta}( z \gamma(x)) - 1] F(z) dz ) + R(\theta, \Delta_{n,i}^{2- 2 \beta}, x)}{1 + R(\theta, \Delta_{n,i}^{(1 - \alpha \beta) \land (2- 3 \beta)}, x) }.
\label{eq: ref m final} 
\end{equation}
Now we can use that $R$ in the denominator is a rest function and so we obtain
\begin{equation}
\frac{1}{1 +  R(\theta, \Delta_{n,i}^{(1 - \alpha \beta) \land (2- 3 \beta)}, x)  } \sim 1 -  R(\theta, \Delta_{n,i}^{(1 - \alpha \beta) \land (2- 3 \beta)}, x).
\label{eq: limite notevole R} 
\end{equation}
Replacing \eqref{eq: limite notevole R} in \eqref{eq: ref m final} we get
$$m_\theta(x)= x + [\Delta_{n,i}(b(x, \theta) + \int_{\mathbb{R} \backslash \left \{0 \right \}} z \gamma(x)[\varphi_{\Delta_{n,i}^\beta}( z \gamma(x)) - 1] F(z) dz ) + R(\theta, \Delta_{n,i}^{2- 2 \beta}, x)](1 - R(\theta, \Delta_{n,i}^{(1 - \alpha \beta) \land (2- 3 \beta)}, x)).$$
The expansion \eqref{eq: dl m} follows.
\end{proof}

\subsection{Proof of Theorem \ref{th: dL mtheta alpha  > 1}}
\begin{proof}
Let us now consider an expansion of \eqref{eq: reformulation m} in the case where $\alpha$ is in $[1,2)$. Again, we skip the study of $\alpha = 1$ to avoid the emergence of logarithmic functions; as it is embedded in the study of $\alpha > 1$ with the choice of $\alpha $ arbitrarily close to $1$. \\
We start observing that \eqref{eq: dynikin theorem 2} and \eqref{eq: estim 2 derivatives g} still hold; we want to show that even in this case the last term of \eqref{eq: dynikin theorem 2} is negligible compared to the others. Again, we consider its decomposition in continuous and discrete part. \\
Concerning $A^2_c g_{i,n}$, \eqref{eq: A^2c theorem 3} is still true. Let us now consider $A_c(A_d g_{i,n})(y)$ as written in \eqref{eq: AcAdg}. We act as in the proof of Theorem \ref{th: dl den alpha >1}, using Taylor development up to second order, on the function $h_{i,n}$ defined in \eqref{eq: definition hni}. Hence we obtain the following estimation:
$$|h_{i,n}(y, z)| \le \left \| g_{i,n}'' \right \|_{\infty} \frac{|z|^2 \gamma(y)^2}{2}$$
and in the same way, using also \eqref{eq: estimations g',g'' and g'''},
$$|\frac{\partial h_{i,n}}{\partial y} (y,z) | \le \left \| g_{i,n}'' \right \|_\infty |z|^2 |\gamma(y) \gamma'(y)| + \left \| g_{i,n}''' \right \|_\infty |z|^2 \frac{\gamma^2(y)}{2}|1 + \gamma'(y)z| \le$$
\begin{equation}
\le |z|^2 P(y) |\Delta_{n,i}|^{-\beta} +  |\Delta_{n,i}|^{-2\beta} P(y)(|z|^2 + |z|^3),
\label{eq: estimation derivative hni alpha>1} 
\end{equation}
\begin{equation}
|\frac{\partial^2 h_{i,n}}{\partial y^2} (y,z)| \le |\Delta_{n,i}^{-\beta}| |z|^2 P(y) + |\Delta_{n,i}|^{-2\beta} P(y)(|z|^2 + |z|^3) +|\Delta_{n,i}|^{-3\beta} P(y)(|z|^2 + |z|^3).  
\label{eq: estim deriv2 hni alpha >1} 
\end{equation}
Since $a^2$ and $b$ have polynomial growth, \eqref{eq: estim deriv2 hni alpha >1} provides us an estimation on $|A_ch_{i,n}(\cdot,z)(y)|$. Using dominated convergence theorem, \eqref{eq: AcAdg}, the estimation of $|A_ch_{i,n}(\cdot,z)(y)|$ obtained from \eqref{eq: estim deriv2 hni alpha >1} and the fact that both $\int_\mathbb{R} (|z|^2 + |z|^3) F(z) dz$ and $\int_\mathbb{R} (|z|^2 + |z|^3) F(z) dz$ are finite, we get
\begin{equation}
\sup_{u_2 \in [t_i, t_{i+1}]}|\mathbb{E}[A_c(A_dg_{n,i})(X^\theta_{u_2})| X^\theta_{t_i} = x] |= R(\theta, \Delta_{n,i}^{- \beta}, x) + R(\theta, \Delta_{n,i}^{- 2 \beta}, x) + R(\theta, \Delta_{n,i}^{- 3 \beta}, x)= R(\theta, \Delta_{n,i}^{- 3 \beta}, x).  
\label{eq: AcAdg alpha >1} 
\end{equation}
We now consider $A_d(A_c g_{i,n})(y)$. Using \eqref{eq: AdAcg} and the development to the second order of the function $A_cg_{i,n}(y + z\gamma(y))$ we obtain
\begin{equation}
|A_d(A_c g_{i,n})(y)| \le c \int_\mathbb{R}\int_0^1|(A_c g_{i,n})''(y + s \, z \gamma(y))||z|^2 |\gamma^2(y)| F(z) ds dz.
\label{eq: estim AdAcg alpha >1} 
\end{equation}
We observe that $(A_cg_{i,n})''(y) = [b''g_{i,n}' + 2b'g_{i,n}'' + bg_{i,n}''' + (a')^2g_{i,n}'' + a(a''g_{i,n}'' + a' g_{i,n}''') + 2aa'g_{i,n}''' + \frac{1}{2} a^2 g_{i,n}^{(4)}](y).$ Using \eqref{eq: estimations g',g'' and g'''}, to which we add $|g_{i,n}^{(4)}(y)| \le c \Delta_{n,i}^{- 3 \beta} $, we get
\begin{equation}
|(A_cg_{i,n})''(y)| \le c\,P(y)\Delta_{n,i}^{-3\beta}.
\label{eq: estimation (Acg)'' alpha >1} 
\end{equation}
Using \eqref{eq: estim AdAcg alpha >1} and \eqref{eq: estimation (Acg)'' alpha >1} it yields
\begin{equation}
\sup_{u_2 \in [t_i, t_{i+1}]} |\mathbb{E}[(A_d A_c g_{i,n})(X^\theta_{u_2})|X^\theta_{t_i} = x]| = R(\theta, \Delta_{n,i}^{- 3 \beta}, x).
\label{eq: AdAcg alpha >1} 
\end{equation}
To conclude, we consider $A_dA_dg_{i,n}$. Using \eqref{eq: AdAdg} and the development up to the second order we get
$$|A_d(A_d g_{i,n})(y)| \le c \int_\mathbb{R}\int_0^1|(A_dg_{i,n})''(y + s \, z \gamma(y))||z|^2 |\gamma^2(y)| F(z) ds dz.
\label{eq: estim AdAdg alpha >1} $$
We recall that \eqref{eq: int hn come Adg} still holds, with $h_{i,n}$ defined in \eqref{eq: definition hni}. In order to estimate $(A_dg_{i,n})''(y)$ in the case where $\alpha \in [1,2)$ we use therefore \eqref{eq: estim deriv2 hni alpha >1} joint with dominated convergence theorem. It provides us
\begin{equation}
\sup_{u_2 \in [t_i, t_{i+1}]} |\mathbb{E}[(A_d A_d g_{i,n})(X^\theta_{u_2})|X^\theta_{t_i} = x]| = R(\theta, \Delta_{n,i}^{- 3 \beta}, x).
\label{eq: AdAdg alpha >1} 
\end{equation}
Using \eqref{eq: A^2c theorem 3}, \eqref{eq: AcAdg alpha >1}, \eqref{eq: AdAcg alpha >1} and \eqref{eq: AdAdg alpha >1} we put the pieces together and so we obtain
$$ \sup_{u_2 \in [t_i, t_{i+1}]} | \mathbb{E}[A^2f_{i,n}(X^\theta_{u_2})| X^\theta_{t_i} = x] | = R(\theta,\Delta_{n,i}^{1 - \alpha\beta - 3\beta - \epsilon}, x) + R(\theta, 1, x) + R(\theta, \Delta_{n,i}^{-3\beta}, x) = R(\theta, \Delta_{n,i}^{-3\beta}, x).$$
Indeed, since $\epsilon$ is arbitrarily small, for each choice of $\alpha$ and $\beta$ we can find $\epsilon$ such that $1 - \alpha\beta - 3\beta - \epsilon > - 3 \beta$. 
We substitute in the Dynkin formula \eqref{eq: dynikin theorem 2} and so we get
$$\mathbb{E}[g_{n,i}(X_{t_{i+1}}^\theta )|X_{t_i}^\theta = x] = $$
\begin{equation}
= \Delta_{n,i}(b(x, \theta) + \int_{\mathbb{R} \backslash \left \{0 \right \}} z \gamma(x)[\varphi_{\Delta_{n,i}^\beta}( z \gamma(x)) - 1] F(z) dz ) + \frac{\Delta_{n,i}^2}{2} R(\theta, \Delta_{n,i}^{-3\beta}, x). 
\label{eq: Dynkin g alpha>1} 
\end{equation}
We use the definition of $\Delta_{n,i}$ and the property \eqref{propriety power R} on $R$, then we substitute in (\ref{eq: Dynkin g alpha>1}) getting \eqref{eq: dl num alpha >1}.\\
In order to prove \eqref{eq: dl m alpha>1}, we observe again that from the expansion \eqref{eq: dl num alpha >1} and the property \eqref{eq: definition R} of $R$, there exists $k_0 > 0$ such that for $|x| \le \Delta_{n,i}^{k_0}$, $\mathbb{E}[\varphi_{\Delta_{n,i}^\beta}(X_{t_{i+1}}^\theta -  X_{t_i}^\theta)|X_{t_i}^\theta = x] \ge \frac{1}{2}$ $\forall n, i \le n$. Using \eqref{eq: reformulation m}, \eqref{eq: Dynkin g alpha>1} and \eqref{eq: dl phi alpha >1} we have that
\begin{equation}
m_\theta(x)= x + \frac{\Delta_{n,i}(b(x, \theta) + \int_{\mathbb{R} \backslash \left \{0 \right \}} z \gamma(x)[\varphi_{\Delta_{n,i}^\beta}( z \gamma(x)) - 1] F(z) dz ) + R(\theta, \Delta_{n,i}^{2- 3 \beta}, x)}{1 + R(\theta, \Delta_{n,i}^{(1 - \alpha \beta) \land (2- 4 \beta)}, x) }.
\label{eq: ref m final alpha >1} 
\end{equation}
Now $R$ in the denominator is a rest function and so
\begin{equation}
\frac{1}{1 + R(\theta, \Delta_{n,i}^{(1 - \alpha \beta) \land (2- 4 \beta)}, x)} \sim 1 - R(\theta, \Delta_{n,i}^{(1 - \alpha \beta) \land (2- 4 \beta)}, x). \label{eq: R final m alpha >1} 
\end{equation}
We now replace \eqref{eq: R final m alpha >1} in \eqref{eq: ref m final alpha >1} and we observe that multiplying by $R$ we obtain negligible functions, hence we get \eqref{eq: dl m alpha>1}.
\end{proof}

Let us now prove the development of $m_{\theta, \Delta_{n,i}}$ in the particular case with finite intensity that makes possible to approximate explicitly the contrast function.

\subsection{Proof of Proposition \ref{prop: dl m intensita finita}}

\begin{proof}
We want to use again Dynkin's formula \eqref{eq: Dynkin formula beginning}.
We consider the decomposition of the generator: $A = A_c + A_d$ and, by the Remark 1 and the fact that we are in the finite intensity case, we can take $A_df(x) = \int_\mathbb{R} \lambda [f(x + \gamma (x) z) - f(x)] F_0(z) dz $, where $F(z) = \lambda F_0(z)$ and $\int_\mathbb{R} F_0(z) dz = 1 $. \\
Concerning the denominator, we denote again $f_{i,n}(y) : = \varphi_{\Delta_{n,i}^\beta}(y-x)$ and, in order to calculate $A^kf_{i,n}(y)$ we introduce the following set of functions:
$$\mathcal{F}^p := \left \{ g(y) \, s. \, t. \, g(y) = \sum_{k = 0}^p \varphi^{(k)}((y - x) \Delta_{n,i}^{-\beta}) \Delta_{n,i}^{- k \beta}(\sum_{j = 0}^k h_{k,j}(y) \Delta_{n,i}^{\beta j})  \right \}$$
where, $\forall k, j$, $\forall l \ge 0 $ $\exists c $ such that $| \frac{\partial^l}{\partial y^l} h_{k,j}(y)| \le c (1 + |y|^c)$ and $\forall k,j$ $h_{k,j}$ is $\mathcal{C}^\infty$. We observe that, if $g \in \mathcal{F}^p $, then $g' \in \mathcal{F}^{p+1}$, $bg$ and $a^2g$ are in $\mathcal{F}^p$ and therefore if $g \in \mathcal{F}^p $, then $Ag \in \mathcal{F}^{p +2}$. \\
We now want to show that, for $g \in \mathcal{F}^p$, $A_d$ acts like $- \lambda I_d$ up to an error term. Indeed,
\begin{equation}
A_dg(y) = \int_\mathbb{R} \lambda [g(y + \gamma(y)z) - g(y)] F_0(z) dz = \lambda \int_\mathbb{R} g(y + \gamma(y)z) F_0(z) dz - \lambda g(y). 
\label{eq: Adg caso finito} 
\end{equation}
Let us start considering $g(y) = \varphi^{(k)}((y - x) \Delta_{n,i}^{-\beta})h(y)$, where $k \le p$ and $h \in \mathcal{C}^\infty$ is such that $\forall l \ge 0$ $\exists c $: $|\frac{\partial ^l}{\partial y^l}h(y)| \le c(1 + |y|^c)$. Then,
$$\int_\mathbb{R} g(y + \gamma(y)z) F_0(z) dz = \int_\mathbb{R} \varphi^{(k)}((y+ \gamma(y)z - x) \Delta_{n,i}^{-\beta})\,h(y + \gamma(y)z) F_0(z) dz.$$
With the change of variable $u := (y+ \gamma(y)z - x) \Delta_{n,i}^{-\beta} $ it becomes equal to
\begin{equation}
\frac{\Delta_{n,i}^\beta}{\gamma(y)} \int_\mathbb{R} \varphi^{(k)}(u) h(x + u \Delta_{n,i}^\beta)F_0(\frac{x - y}{\gamma(y)} + \frac{\Delta_{n,i}^\beta u}{\gamma(y)}) du.
\label{eq: varphi contro F} 
\end{equation}
We define $\tilde{F}(x, y, s ) : = \frac{h(x + s)}{\gamma(y)}F_0(\frac{x - y}{\gamma(y)} + \frac{s}{\gamma(y)}) $ and we develop it up to the M-order, getting
$$\tilde{F}(x, y, \Delta_{n,i}^\beta u ) = \sum_{j = 0}^M \frac{\partial^j \tilde{F} }{\partial s^j} (x, y, 0) (\Delta_{n,i}^\beta u)^j + \int_0^1 \frac{\partial^{M+1}}{\partial s^{M+1}} \tilde{F}(x, y, t \Delta_{n,i}^\beta u) \frac{(1 - t)^M}{M !}(\Delta_{n,i}^\beta u)^{M + 1} dt.$$
Replacing the development in \eqref{eq: varphi contro F} and recalling that by the definition of $\varphi$ we have $\int_\mathbb{R} u^j \varphi^{(k)}(u) du = 0 $,  we get
\begin{equation}
\int_\mathbb{R} \varphi^{(k)}(u)\tilde{F}(x, y, \Delta_{n,i}^\beta u ) du = \sum_{j = 0}^M 0 + \int_\mathbb{R}\int_0^1 \varphi^{(k)}(u)\frac{\partial^{M+1}}{\partial s^{M+1}} \tilde{F}(x, y, t \Delta_{n,i}^\beta u) \frac{(1 - t)^M}{M !}(\Delta_{n,i}^\beta u)^{M + 1} dt \, du.
\label{eq: 48.5 dans le papier modifié} 
\end{equation}
We observe that it is $|\frac{\partial ^{l_1 + l_2 + l_3}}{\partial s^{l_1} \partial x^{l_2} \partial y^{l_3}}\tilde{F}(x, y, s)| \le c(1 + |x|^c +|y|^c + |s|^c)$. Therefore, since the support of $\varphi^{(k)}$ is compact, we get
\begin{equation}
\int_\mathbb{R}\int_0^1 \varphi^{(k)}(u)\frac{\partial^{M+1}}{\partial s^{M+1}} \tilde{F}(x, y, t \Delta_{n,i}^\beta u) \frac{(1 - t)^M}{M !}(\Delta_{n,i}^\beta u)^{M + 1} dt \, du \le c (\Delta_{n,i}^\beta)^{M+1}(1 + |x|^c + |y|^c). 
\label{eq: estimation particular g} 
\end{equation}
Hence using \eqref{eq: varphi contro F} and \eqref{eq: estimation particular g} on $|\int_\mathbb{R} g(y + \gamma(y)z) F_0(z) dz|$ and the differentiation of \eqref{eq: 48.5 dans le papier modifié} on $|\frac{\partial^l}{\partial y^l}\int_\mathbb{R} g(y + \gamma(y)z) F_0(z) dz|$ we get that both of them are upper bounded by $c(1 + |x|^c + |y|^c) \Delta_{n,i}^{\beta(M+2)}$, where in the second case the constant $c$ depends on $l$. \\
Turning to a general function $g \in \mathcal{F}^p$, the estimations above become
\begin{equation}
|\int_\mathbb{R} g(y + \gamma(y)z) F_0(z) dz| \le c(1 + |x|^c + |y|^c) \Delta_{n,i}^{\beta(M+2)} \Delta_{n,i}^{- \beta p}
\label{eq: estim g} 
\end{equation}
and, $\forall l \ge 1$,
\begin{equation}
|\frac{\partial^l}{\partial y^l} \int_\mathbb{R} g(y + \gamma(y)z) F_0(z) dz| \le c_l(1 + |x|^{c_l} + |y|^{c_l}) \Delta_{n,i}^{\beta(M+2)} \Delta_{n,i}^{- \beta p}.
\label{eq: estim derivate g} 
\end{equation}
We introduce the set of functions
$$\mathcal{R}^p := \left \{ r(x, y, \Delta_{n,i}^p) \mbox{ such that } \forall l \ge 0 \, \, \exists c_l \, |\frac{\partial^l}{\partial y^l} r(x, y, \Delta_{n,i}^p) |\le c_l(1 + |x|^{c_l} + |y|^{c_l}) \Delta_{n,i}^p   \right \}.$$
Hence, using \eqref{eq: Adg caso finito}, \eqref{eq: estim g} and \eqref{eq: estim derivate g} we have proved that, $\forall g \in \mathcal{F}^p$,
\begin{equation}
A_d g (y) = - \lambda g(y) + r(x, y, \Delta_{n,i}^{\beta (M+2 - p)}).
\label{eq: 51. 5 dans le papier modifié} 
\end{equation}
We observe that if a function $r$ is in $\mathcal{R}^p$, then both $A_d r$ and $A_c r$ are in $\mathcal{R}^p$. We can therefore now calculate for $f_{i,n}(y) = \varphi((y - x) \Delta_{n,i}^{- \beta}) $, $f_{i,n} \in \mathcal{F}^0$,
\begin{equation}
A_{i_1}f_{i,n}(y) =
\begin{cases}
A_cf_{i,n}(y) \qquad \mbox{if } i_1 = c \\
A_d f_{i,n} (y) = - \lambda f_{i,n}(y) + r(x, y, \Delta_{n,i}^{\beta (M+2)}) \qquad \mbox{if } i_1 = d,
\end{cases}
\end{equation}
We want to show, by recurrence, that
\begin{equation}
A_{i_N} \circ ... \circ A_{i_1} (f_{i,n})(y) = A_c^{l(i_1, ... , i_N)} f_{i,n} (y) (- \lambda)^{N - l(i_1, ... , i_N) } + r(x, y, \Delta_{n,i}^{\beta(M +2) - 2 \beta l(i_1, ... , i_N)}),
\label{eq: recurrence} 
\end{equation}
with $l(i_1, ... , i_N)$ the number of $c$ in $\left \{ i_1, ... , i_N \right \}$. 
Let us consider the base case
\begin{equation}
A_{i_2} \circ A_{i_1} f_{i,n}(y) =
\begin{cases}
A_c^2f_{i,n}(y) \qquad \mbox{if } i_2 = i_1 = c \\
A_c(- \lambda f_{i,n}(y) + r(x, y, \Delta_{n,i}^{\beta (M+2)})) = -\lambda A_cf_{i,n}(y) + r(x, y, \Delta_{n,i}^{\beta (M+2)}) \qquad \mbox{if } i_2 = c, \, i_1 = d \\
-\lambda A_cf_{i,n}(y) + r(x, y, \Delta_{n,i}^{\beta (M+2) - 2 \beta}) \qquad \mbox{if } i_2 = d, \, i_1 = c \\
A_d(- \lambda f_{i,n}(y) + r(x, y, \Delta_{n,i}^{\beta (M+2)})) = \lambda^2 f_{i,n}(y) + r(x, y, \Delta_{n,i}^{\beta (M+2)})
\quad \mbox{if } i_2 = i_1 = d,  
\end{cases}
\end{equation}
where in the third case we have used $A_cf_{i,n} \in \mathcal{F}^2$.
So we have
$$A_{i_2} \circ A_{i_1} f_{i,n}(y) = A_c^{l(i_1, i_2)}f_{i,n} (y) (- \lambda)^{2 - l(i_1, i_2)} + r(x, y, \Delta_{n,i}^{\beta(M +2) - 2 \beta l(i_1, i_2)}),$$
as we wanted. For the inductive step, we assume that \eqref{eq: recurrence} holds, now 
$$A_{i_{N + 1}} \circ A_{i_N} \circ ... \circ A_{i_1} (f_{i,n})(y) = $$
\begin{equation}
=
\begin{cases}
A_c \circ  A_c^{l(i_1, ... , i_N)} f_{i,n} (y) (- \lambda)^{N - l(i_1, ... , i_N) } + r(x, y, \Delta_{n,i}^{\beta(M +2) - 2 \beta l(i_1, ... , i_N)}) \quad \mbox{if } i_{N + 1} = c, \\ 
(- \lambda)A_c^{l(i_1, ... , i_N)} f_{i,n} (y) (- \lambda)^{N - l(i_1, ... , i_N) } + r(x, y, \Delta_{n,i}^{\beta(M +2) - 2 \beta l(i_1, ... , i_N)}) \quad \mbox{if } i_{N + 1} = d,
\end{cases}
\end{equation}
where in the first case we have used that $A_c r(x, y, \Delta_{n,i}^h) \in \mathcal{R}^h $, $\forall h$, and in the second case that $A_d r(x, y, \Delta_{n,i}^h) \in \mathcal{R}^h$ and that $A_c^{l(i_1, ... , i_N)} f_{i,n} \in \mathcal{F}^{2 l(i_1, ... , i_N)}$ while using \eqref{eq: 51. 5 dans le papier modifié}. \\
It is equal to $A_c^{l(i_1, ... , i_N, i_{N + 1})} f_{i,n} (y) (- \lambda)^{N + 1 - l(i_1, ... , i_N, i_{N + 1}) } + r(x, y, \Delta_{n,i}^{\beta(M +2) - 2 \beta l(i_1, ... , i_N, i_{N + 1})})$ and therefore the recurrence is proved.
We can now calculate $A^kf_{i,n}(x)$ in the Dynkin's formula \eqref{eq: Dynkin formula beginning} using \eqref{eq: recurrence}:
$$A^k f_{i,n}(x) = \sum_{(i_1, ... , i_k) \in \left \{ c, d \right \}^k} (A_{i_k} \circ ... \circ A_{i_1}) f_{i,n}(x) = $$
\begin{equation}
= \sum_{(i_1, ... , i_k) \in \left \{ c, d \right \}^k} A_c^{l(i_1, ... , i_k)} f_{i,n} (x) (- \lambda)^{k - l(i_1, ... , i_k) } + r(x, x, \Delta_{n,i}^{\beta(M +2) - 2 \beta l(i_1, ... , i_k)}).
\label{eq: sostituzione Akf} 
\end{equation}
Recalling that $A_c^l f_{i,n}(x) = 0 $ $\forall l \ge 1$, \eqref{eq: sostituzione Akf} becomes
$(- \lambda)^k f_{i,n}(x) + r(x, x, \Delta_{n,i}^{\beta(M +2) - 2 \beta k}). $ \\
Therefore, the principal term in the development of the denominator of $m_{\theta, \Delta_{n,i}}(x)$ from Dynkin's formula up to order $N$ is 
$$\sum_{k =0}^N \frac{\Delta_{n,i}^k}{k!} A^kf_{i,n}(x) = \sum_{k =0}^N \frac{\Delta_{n,i}^k}{k!} (- \lambda)^k f_{i,n}(x) + r(x, x, \Delta_{n,i}^{\beta(M +2) - 2 \beta k + k}).$$
Let us now consider the term of rest in the Dynkin's formula \eqref{eq: Dynkin formula beginning}. Observing that
$$|A_c^{N + 1}f_{i,n}(y)| \le \Delta_{n,i}^{- 2 \beta(N + 1)}(1 + |y|^c)$$
using \eqref{eq: recurrence} and the definition of the function $r$, we get that
\begin{equation}
|A^{N + 1} f_{i,n} (y)| \le c(\Delta_{n,i}^{- 2 \beta(N + 1)} + \Delta_{n,i}^{\beta(M +2) - 2 \beta (N +1)} ) (1 + |y|^ c ).
\label{eq: termine di resto Dynkin} 
\end{equation}
Therefore
\begin{equation}
\mathbb{E}[|A^{N + 1} f_{i,n} (X_{u_{n + 1}})| | X_{t _i} = x ] \le c(\Delta_{n,i}^{- 2 \beta(N + 1)} + \Delta_{n,i}^{\beta(M +2) - 2 \beta (N +1)} ) (1 + |x|^ c ).
\label{eq: esperance terme du reste} 
\end{equation}
Replacing in \eqref{eq: Dynkin formula beginning} it yields
$$|\int_{t_i}^{t_{i+1}} \int_{t_i}^{u_1} ... \int_{t_i}^{u_N}\mathbb{E}[A^{N + 1} f_{i,n} (X_{u_{n + 1}}) | X_{t _i} = x ] du_{N + 1}... du_2 du_1| \le $$
$$ \le c\Delta_{n,i}^{N + 1}(\Delta_{n,i}^{- 2 \beta(N + 1)} + \Delta_{n,i}^{\beta(M +2) - 2 \beta (N +1)} ) (1 + |x|^ c ). $$
Since $\Delta_{n,i}^{\beta(M +2) - 2 \beta (N +1)} $ is negligible compared to $\Delta_{n,i}^{- 2 \beta(N + 1)}$, it is enough to have $(N + 1)(1 - 2 \beta) \ge \lfloor \beta(M + 2) \rfloor$ in order to get the following development of the denominator $d_{\Delta_{n,i}}(x)$ of $m_{\theta, \Delta_{n,i}}(x)$:
$$d_{\Delta_{n,i}}(x) = \sum_{k = 0}^N \frac{\Delta_{n,i}^k}{k!}(- \lambda)^k f_{i,n}(x) + r(x, x, \Delta_{n,i}^{\beta(M+2) + (1 - 2 \beta)k}) + r(x, x, \Delta_{n,i}^{(1 - 2 \beta)(N + 1)}) =$$
$$= \sum_{k = 0}^{\lfloor \beta(M + 2) \rfloor} \frac{\Delta_{n,i}^k}{k!}(- \lambda)^k + r(x, x, \Delta_{n,i}^{\beta(M+2)}), $$
where we have also used that, by the definition of $f_{i,n}$, $f_{i,n}(x) = 1$ and in the sum we have considered only the terms up to $k = \lfloor \beta(M+2) \rfloor$ because the others are rest terms. \\
Let us now study the numerator $n_{\Delta_{n,i}}(x)$ of $m_{\theta, \Delta_{n,i}}(x)$: acting like in the proof of Theorem \ref{th: dL mtheta alpha <1} we consider $\bar{g}(y): = (y - x) \varphi ((y - x) \Delta_{n,i}^{ -\beta})$. Let us introduce, in place of $\mathcal{F}^p$, the set $\tilde{\mathcal{F}}^p$.
$$\tilde{\mathcal{F}}^p : = \left \{ \tilde{g}(y) \, \, s.t. \, \tilde{g}(y) = \sum_{k = 0}^p \varphi^{(k)}((y - x) \Delta_{n,i}^{- \beta}) \Delta_{n,i}^{- k \beta} (\sum_{j = 0}^k h_{k,j}(x,y) \Delta_{n,i}^{\beta j})  \right \}$$
where, $\forall k, j$, $\forall l \ge 0$, $\exists c_l$ such that $| \frac{\partial^l}{\partial y^l} h_{k,j}(x,y)| \le c_l (1 + |x|^{c_l} + |y|^{c_l})$. We observe that, as it was for $\mathcal{F}^p$, if $\tilde{g} \in \tilde{\mathcal{F}}^p$ then $A\tilde{g} \in \tilde{\mathcal{F}}^{p + 2}$ and, for all $\tilde{g} \in \tilde{\mathcal{F}}^p$,
\begin{equation}
A_d \tilde{g}(y) = -\lambda \tilde{g} (y) + r(x,y, \Delta_{n,i}^{\beta(M + 2 - p)}).
\label{eq: adg in tildeF} 
\end{equation}
It turns that the same relation as \eqref{eq: recurrence} holds with $\bar{g}$ in place of $f_{i,n}$.
Hence we get
\begin{equation}
A^k \bar{g}(y) = (A_c + A_d)^k \bar{g}(y) = \sum_{(i_1, ... , i_k) \in \left \{ c,d \right \}^k} A_c^{l(i_1, ... , l_k)} \bar{g} (y)(- \lambda)^{k - l(i_1, ... , i_k)} + r(x,y, \Delta_{n,i}^{\beta(M + 2 ) - 2 \beta l(i_1, ... ,  i_k)}) =
\label{eq: Akg development intesite finie} 
\end{equation}
$$ = \sum_{l = 0}^k \binom{k}{l} (- \lambda)^{k - l} A^l_c \bar{g}(y) +  r(x,y, \Delta_{n,i}^{\beta(M + 2 ) - 2 \beta k}),$$
where $l(i_1, ... , i_k)$ is the number of $c$ in $\left \{ i_1, ..., i_k \right \}$ and $\binom{k}{l}$ are the binomial coefficients. Now, concerning the continuous part of the generator, since it is local and $\bar{g}(y) = (y - x)$ in the neighborhood of $x$, we find $A_c^l \bar{g}(x) = A_K^{(l)}(x)$, which are exactly the coefficients found in the case without jump studied by Kessler in \hyperref[csl:15]{(Kessler, 1997)}.\\
By \eqref{eq: Akg development intesite finie}, the principal term in the development of the numerator is therefore
$$\sum_{k = 0}^N \frac{\Delta_{n,i}^k}{k!} A^k g(x) = \sum_{k = 0}^N \frac{\Delta_{n,i}^k}{k!}(\sum_{l = 0}^k \binom{k}{l} (- \lambda)^{k - l} A^{(l)}_K(x) +  r(x,x, \Delta_{n,i}^{\beta(M + 2 ) - 2 \beta k})) = $$
\begin{equation}
= \sum_{k = 0}^N \frac{\Delta_{n,i}^k}{k!}(\sum_{l = 0}^k \binom{k}{l} (- \lambda)^{k - l} A^{(l)}_K(x)) + r(x,x, \Delta_{n,i}^{\beta(M + 2 )}). 
\label{eq: Akg primo} 
\end{equation}
Changing the order of summation and introducing $k' := k-l $ we get that the first term of the previous equation is equal to
$$\sum_{l = 0}^N \sum_{k = l}^N \frac{\Delta_{n,i}^k}{k!} \binom{k}{l} (- \lambda)^{k - l} A^{(l)}_K(x) = \sum_{l = 0}^N \frac{\Delta_{n,i}^l}{l!} A^{(l)}_K (x) \sum_{k' = 0 }^{N - l} \Delta_{n,i}^{k'} (- \lambda)^{k'}\frac{l!}{(k' + l)!} \binom{l + k'}{l} = $$
\begin{equation}
= \sum_{l = 0}^N \frac{\Delta_{n,i}^l}{l!} A^{(l)}_K (x) \sum_{k' = 0 }^{N - l}\frac{ \Delta_{n,i}^{k'} (- \lambda)^{k'}}{k'!},
\label{eq: acg risolto} 
\end{equation}
where in the last equality we have used the definition of binomial coefficients.
Concerning the rest term in the Dynkin's formula, we use again \eqref{eq: termine di resto Dynkin} and \eqref{eq: esperance terme du reste} with $\bar{g}$ in place of $f_{i,n}$ and it turns again
\begin{equation}
|\int_{t_i}^{t_{i+1}} \int_{t_i}^{u_1} ... \int_{t_i}^{u_N}\mathbb{E}[A^{N + 1} \bar{g} (X_{u_{N + 1}}) | X_{t _i} = x ] du_{N + 1}... du_2 du_1| \le r(x,x, \Delta_{n,i}^{(1 - 2 \beta ) (N + 1)}). 
\label{eq: rest in g} 
\end{equation}
Hence, using \eqref{eq: Akg primo}, \eqref{eq: acg risolto} and \eqref{eq: rest in g} we have the following development:
\begin{equation}
n_{\Delta_{n,i}}(x) = \sum_{l = 0}^N \frac{\Delta_{n,i}^l}{l!} A^{(l)}_K(x) \sum_{k' = 0 }^{N - l}\frac{ \Delta_{n,i}^{k'} (- \lambda)^{k'}}{k'!} + r(x,x, \Delta_{n,i}^{\beta (M + 2)}) + r(x,x, \Delta_{n,i}^{(1 - 2 \beta ) (N + 1)}).
\label{eq: dl numeratore intenita finita} 
\end{equation}
If $(N + 1)(1 - 2 \beta) \ge \beta (M + 2)$, it entails
$$ n_{\Delta_{n,i}}(x) = \sum_{l = 0}^{\lfloor \beta(M + 2) \rfloor} \frac{\Delta_{n,i}^l}{l!} A^{(l)}_K(x) \sum_{k' = 0 }^{\lfloor \beta(M + 2) \rfloor} \frac{ \Delta_{n,i}^{k'} (- \lambda)^{k'}}{k'!} + r(x,x, \Delta_{n,i}^{\beta (M + 2)}).$$
Acting as in the proof of the development of $m_{\theta}$ given in Theorem \ref{th: dL mtheta alpha <1} we can say that it exists $k_0 > 0$ such that, for $|x| \le \Delta_{n,i}^{ - k_0} $, the development of $m_{\theta, \Delta_{n,i}}(x)$ is
\begin{equation}
x + \frac{n_{\Delta_{n,i}}(x)}{d_{\Delta_{n,i}}(x)} = x + \sum_{l = 0}^{\lfloor \beta(M + 2) \rfloor} \frac{\Delta_{n,i}^l}{l!} A^{(l)}_K (x) + r(x,x, \Delta_{n,i}^{\beta (M + 2)}).
\label{eq: dl tout ordre m} 
\end{equation}
 The expansion \eqref{eq: dl m intensita finita} follows after remarking that $A^{(0)}_K(x)=0$.
\end{proof}

\subsection{Contrast convergence}\label{Ss:Contrast_conv}
Before proving the contrast convergence, let us define $r(\theta,x)$ as the particular rest function that turns out from the development of $m_{\theta, \Delta_{n,i}}$:
\begin{equation}
r(\theta, x) := m_{\theta, \Delta_{n,i}}(x) - x - \Delta_{n,i} \, b(x, \theta) -  \Delta_{n,i} \, \int_{\mathbb{R} \backslash \left \{0 \right \}} z\, \gamma(x)\,[1 - \varphi_{\Delta_{n,i}^\beta}(\gamma(x)z)]\,F(z)dz.
\label{eq: definition r} 
\end{equation}
We recall that $r(\theta, x)$ is $R(\theta, \Delta_{n,i}^{1 + \delta}, x)$ with $\delta >0$ as defined below equation \eqref{eq: link betwen m and zeta}.\\
In order to prove the consistency and asymptotic normality of the estimator, the first step is the following Lemma: \\
\begin{lemma}
Suppose that Assumptions 1-5 and $A_\beta$ are satisfied. Then
\begin{equation}
\frac{U_n(\theta) - U_n(\theta_0)}{t_n} \xrightarrow{\mathbb{P}} \int_\mathbb{R} \frac{(b(x, \theta) - b(x, \theta_0))^2}{a^2(x)} \pi(dx)
\label{eq: contrast convergence} 
\end{equation}
\label{lemma: convergence contrast} 
\end{lemma}
\begin{proof}
By the definition, 
$$U_n(\theta)= \sum_{i =0}^{n -1} \frac{(X_{t_{i+1}} - m_\theta(X_{t_i}))^2}{a^2(X_{t_i})\Delta_{n,i}} \, \varphi_{\Delta_{n,i}^\beta}(X_{t_{i+1}} - X_{t_i})1_{\left \{|X_{t_i}| \le \Delta_{n,i}^{-k} \right \} }.$$
We want to reformulate the contrast function, in order to compensate for the terms not depending on $\theta$ in the difference $U_n(\theta) - U_n(\theta_0)$. \\
The dynamic of the process $X$ is known and so we can write
$$X_{t_{i +1}} = X_{t_i} + \int_{t_i}^{t_{i+1}} b(X_s, \theta_0) ds + \int_{t_i}^{t_{i+1}} a(X_s) dW_s + \int_{t_i}^{t_{i+1}} \int_\mathbb{R} \gamma(X_{s^-})z \tilde{\mu}(ds,dz).$$
We have proved the development (\ref{eq: dl m}) of $m_\theta$, too. We can substitute both of them in $U_n(\theta)$, getting
$$ U_n(\theta)= \sum_{i=0}^{n-1} \frac{1}{a^2(X_{t_i})\Delta_{n,i}}[X_{t_i} + \int_{t_i}^{t_{i+1}} b(X_s, \theta_0) ds + \int_{t_i}^{t_{i+1}} a(X_s) dW_s + \int_{t_i}^{t_{i+1}} \int_\mathbb{R} \gamma(X_{s^-})z \tilde{\mu}(ds,dz)
- X_{t_i}+ $$
$$ + \Delta_{n,i}(-  b(X_{t_i}, \theta) + \int_{\mathbb{R} \backslash \left \{0 \right \}} z\, \gamma(X_{t_i})\,(1 - \varphi_{\Delta_{n,i}^\beta}(\gamma(X_{t_i})z))\,F(z)dz) + r(\theta, X_{t_i} )]^2\varphi_{\Delta_{n,i}^\beta}(X_{t_{i+1}} - X_{t_i})1_{\left \{|X_{t_i}| \le \Delta_{n,i}^{-k} \right \} } \, =$$
$$= \sum_{i=0}^{n-1} \frac{1}{a^2(X_{t_i})\Delta_{n,i}} (\int_{t_i}^{t_{i+1}} b(X_s, \theta_0) ds 
- \Delta_{n,i}\, b(X_{t_i}, \theta) + \zeta_i + r(\theta, X_{t_i} ))^2 \varphi_{\Delta_{n,i}^\beta}(X_{t_{i+1}} - X_{t_i})1_{\left \{|X_{t_i}| \le \Delta_{n,i}^{-k} \right \} }, $$
we recall the definition of $\zeta_i := \int_{t_i}^{t_{i+1}} a(X_s)dW_s + \int_{t_i}^{t_{i+1}} \int_{\mathbb{R} \backslash \left \{0 \right \}} \gamma(X_{s^-})z \tilde{\mu}(ds,dz) \,+ $ \\
$+  \Delta_{n,i}\, \int_{\mathbb{R} \backslash \left \{0 \right \}} z \, \gamma(X_{t_i})\,[1 - \varphi_{\Delta_{n,i}^\beta}(\gamma(X_{t_i})z)]\,F(z)dz,$ as in (\ref{eq: def zeta}); we point out that $\zeta_i$ does not depend on $\theta$. \\
In the same way
$$U_n(\theta_0) = \sum_{i=0}^{n-1} \frac{1}{a^2(X_{t_i})\Delta_{n,i}} (\int_{t_i}^{t_{i+1}} b(X_s, \theta_0) ds 
- \Delta_{n,i}\, b(X_{t_i}, \theta_0) + \zeta_i +  r(\theta_0, X_{t_i} ))^2 \varphi_{\Delta_{n,i}^\beta}(X_{t_{i+1}} - X_{t_i}) 1_{\left \{|X_{t_i}| \le \Delta_{n,i}^{-k} \right \} } $$
and so $$\frac{U_n(\theta) - U_n(\theta_0) }{t_n} = \frac{1}{t_n} \sum_{i=0}^{n-1} \frac{\varphi_{\Delta_{n,i}^\beta}(X_{t_{i+1}} - X_{t_i}) 1_{\left \{|X_{t_i}| \le \Delta_{n,i}^{-k} \right \} }}{a^2(X_{t_i})\Delta_{n,i}}[\Delta_{n,i}^2(b(X_{t_i}, \theta)^2 - b(X_{t_i}, \theta_0)^2) + $$
\begin{equation}
+ 2\Delta_{n,i}\int_{t_i}^{t_{i+1}} b(X_s, \theta_0) ds(b(X_{t_i}, \theta_0) - b(X_{t_i}, \theta)) + A_i + B_i + C_i + D_i + E_i],
\label{eq: difference contraste convergence} 
\end{equation}
with $$A_i = 2\zeta_i \Delta_{n,i} (b(X_{t_i}, \theta_0) - b(X_{t_i}, \theta)), \qquad B_i = 2 \zeta_i(r(\theta, X_{t_i} ) - r(\theta_0, X_{t_i} )),$$
$$C_i = 2 \Delta_{n,i}(r(\theta_0, X_{t_i})b(X_{t_i}, \theta_0) - r(\theta, X_{t_i} )b(X_{t_i}, \theta)), \qquad D_i = r(\theta, X_{t_i} )^2 - r(\theta_0, X_{t_i} )^2, $$
$$E_i = 2 \int_{t_i}^{t_{i+1}} b(X_s, \theta_0) ds (r(\theta, X_{t_i} ) - r(\theta_0, X_{t_i} )). $$
Our goal is to show that the contribution of $A_i$, $B_i$, $C_i$, $D_i$ and $E_i$ go to zero in probability as $n \rightarrow \infty$ and to prove that the other terms converge to $\int_\mathbb{R} \frac{(b(x, \theta) - b(x, \theta_0))^2}{a^2(x)} \pi(dx)$. \\
We observe that the rest function $r(\theta,x)$ is present in all the terms that have to converge to $0$ but $A_i$, on which we use a different motivation to obtain the convergence:
$$\frac{1}{t_n} \sum_{i=0}^{n-1} \frac{\varphi_{\Delta_{n,i}^\beta}(X_{t_{i+1}} - X_{t_i})1_{\left \{|X_{t_i}| \le \Delta_{n,i}^{-k} \right \} }}{a^2(X_{t_i})\Delta_{n,i}}A_i =  \frac{1}{t_n} \sum_{i=0}^{n-1} \varphi_{\Delta_{n,i}^\beta}(X_{t_{i+1}} - X_{t_i}) f_{i,n}(X_{t_i}, \theta) 1_{\left \{|X_{t_i}| \le \Delta_{n,i}^{-k} \right \} } \zeta_i, $$
with $f_{i,n}(X_{t_i}, \theta): = \frac{2}{a^2(X_{t_i})}\,( b(X_{t_i}, \theta_0) -  b(X_{t_i}, \theta)) $. \\
In order to apply Proposition \ref{prop: LT2} we observe that, by the assumptions done on the coefficients, $f_{i,n}$ has polynomial growth.
We therefore get the convergence to zero in probability, using Proposition \ref{prop: LT2}. \\
We want to show that $\frac{1}{t_n} \sum_{i=0}^{n-1} \frac{\varphi_{\Delta_{n,i}^\beta}(X_{t_{i+1}} - X_{t_i})1_{\left \{|X_{t_i}| \le \Delta_{n,i}^{-k} \right \} }}{a^2(X_{t_i})\Delta_{n,i}}B_i \xrightarrow{\mathbb{P}} 0 $ and so we observe that, by the definition of the function $r$ and by \eqref{propriety power R} we have that
\begin{equation}
r(\theta,X_{t_i})1_{\left \{|X_{t_i}| \le \Delta_{n,i}^{-k} \right \}}= R(\theta, \Delta_{n,i}^{1 + \delta},X_{t_i})= \Delta_{n,i}^{1 + \delta}R(\theta, 1, X_{t_i}).
\label{eq: puissance R} 
\end{equation}
Hence
$$\frac{1}{t_n} \sum_{i=0}^{n-1} \frac{\varphi_{\Delta_{n,i}^\beta}(X_{t_{i+1}} - X_{t_i})1_{\left \{|X_{t_i}| \le \Delta_{n,i}^{-k} \right \} }}{a^2(X_{t_i})\Delta_{n,i}}B_i = $$
$$ = \frac{1}{t_n} \sum_{i=0}^{n-1} \Delta_{n,i}^{\delta} \zeta_i \frac{\varphi_{\Delta_{n,i}^\beta}(X_{t_{i+1}} - X_{t_i})1_{\left \{|X_{t_i}| \le \Delta_{n,i}^{-k} \right \} }(R(\theta,1, X_{t_i}) - R(\theta_0,1, X_{t_i} ))}{a^2(X_{t_i})}$$
To prove the convergence, we have to show that
\begin{equation}
\frac{1}{t_n} \sum_{i= 0}^{n-1} |\mathbb{E}[\Delta_{n,i}^{\delta} f_{i,n}(X_{t_i}, \theta) 1_{\left \{|X_{t_i}| \le \Delta_{n,i}^{-k} \right \} }\zeta_i \varphi_{\Delta_{n,i}^\beta}(X_{t_{i+1}} - X_{t_i})| \mathcal{F}_{t_i}] |\xrightarrow{\mathbb{P}} 0,
\label{eq: convergence Bi} 
\end{equation}
and
$$ \frac{1}{(t_n)^2} \sum_{i= 0}^{n-1}\mathbb{E}[\Delta_{n,i}^{2 \delta} f_{i,n}^2(X_{t_i}, \theta) 1_{\left \{|X_{t_i}| \le \Delta_{n,i}^{-k} \right \} } \zeta_i^2 \varphi^2_{\Delta_{n,i}^\beta}(X_{t_{i+1}} - X_{t_i})| \mathcal{F}_{t_i}] \xrightarrow{\mathbb{P}} 0, $$
with $$f_{i,n}(X_{t_i}, \theta) = \frac{R(\theta,1, X_{t_i}) - R(\theta_0,1, X_{t_i} )}{a^2(X_{t_i})}.$$
By the measurability of $X_{t_i}$ with respect to $\mathcal{F}_{t_i}$, by the fact that $|\Delta_{n,i}| \le
\Delta_n$ and that $t_n = 0(n \Delta_n)$ we get
$$\frac{1}{t_n} \sum_{i= 0}^{n-1} |\mathbb{E}[\Delta_{n,i}^{\delta} f_{i,n}(X_{t_i}, \theta) 1_{\left \{|X_{t_i}| \le \Delta_{n,i}^{-k} \right \} } \zeta_i \varphi_{\Delta_{n,i}^\beta}(X_{t_{i+1}} - X_{t_i})| \mathcal{F}_{t_i}] | = $$
$$= \frac{1}{t_n} \sum_{i= 0}^{n-1}\Delta_{n,i}^{\delta}| f_{i,n}(X_{t_i}, \theta)||\mathbb{E}[1_{\left \{|X_{t_i}| \le \Delta_{n,i}^{-k} \right \} }\zeta_i \varphi_{\Delta_{n,i}^\beta}(X_{t_{i+1}} - X_{t_i})| \mathcal{F}_{t_i}] | \le $$
$$\le \Delta_n^{\delta} \frac{c}{n\Delta_n} \sum_{i= 0}^{n-1} |f_{i,n}(X_{t_i}, \theta)| |\mathbb{E}[1_{\left \{|X_{t_i}| \le \Delta_{n,i}^{-k} \right \} } \zeta_i \varphi_{\Delta_{n,i}^\beta}(X_{t_{i+1}} - X_{t_i})| \mathcal{F}_{t_i}] | .$$
We recall that $\delta$ is positive. Using \eqref{eq: convergence expected value prop 2}, we get the convergence \eqref{eq: convergence Bi} in $L^1$ and thus in probability. \\
In the same way, 
$$ \frac{1}{(t_n)^2} \sum_{i= 0}^{n-1}\mathbb{E}[\Delta_{n,i}^{2\delta} f_{i,n}^2(X_{t_i}, \theta)1_{\left \{|X_{t_i}| \le \Delta_{n,i}^{-k} \right \} } \zeta_i^2 \varphi^2_{\Delta_{n,i}^\beta}(X_{t_{i+1}} - X_{t_i})| \mathcal{F}_{t_i}] \le $$
$$\le \Delta_n^{2\delta} \frac{c}{n^2\Delta^2_n} \sum_{i= 0}^{n-1} f_{i,n}^2(X_{t_i}, \theta) \mathbb{E}[1_{\left \{|X_{t_i}| \le \Delta_{n,i}^{-k} \right \} } \zeta_i^2 \varphi^2_{\Delta_{n,i}^\beta}(X_{t_{i+1}} - X_{t_i})| \mathcal{F}_{t_i}],$$
that goes to zero in probability using \eqref{eq: conv squared prop 2}. \\
$$\frac{1}{t_n} \sum_{i=0}^{n-1} \frac{\varphi_{\Delta_{n,i}^\beta}(X_{t_{i+1}} - X_{t_i})1_{\left \{|X_{t_i}| \le \Delta_{n,i}^{-k} \right \} }}{a^2(X_{t_i})\Delta_{n,i}}C_i = \frac{2}{t_n} \sum_{i=0}^{n-1} \Delta_{n,i}^{1 + \delta} f_{i,n}(X_{t_i}, \theta) \varphi_{\Delta_{n,i}^\beta}(X_{t_{i+1}} - X_{t_i}),$$
with $f_{i,n}(X_{t_i}, \theta): = \frac{R(\theta_0,1, X_{t_i})b(X_{t_i}, \theta_0) - R(\theta,1, X_{t_i} )b(X_{t_i}, \theta)}{a^2(X_{t_i})} 1_{\left \{|X_{t_i}| \le \Delta_{n,i}^{-k} \right \} } $, where we have used \eqref{eq: puissance R}. \\
In module, it is upper bounded by $\Delta_n^{\delta} \frac{c}{n} \sum_{i=0}^{n-1} | f_{i,n}(X_{t_i}, \theta) \varphi_{\Delta_{n,i}^\beta}(X_{t_{i+1}} - X_{t_i})|.$ \\
We observe that the exponent on $\Delta_n$ is positive so it goes to zero as $n \rightarrow \infty $ and that $|\varphi_{\Delta_{n,i}^\beta}(X_{t_{i+1}} - X_{t_i})| \le c$. 
By the polynomial growth of $f_{i,n}$ and the third point of Lemma \ref{lemma 2.1 GLM}, we get that 
$\frac{1}{n} \sum_{i=0}^{n-1} | f_{i,n}(X_{t_i}, \theta)| $ is bounded in $L^1$. It yields the convergence in probability that we were looking for. \\
Let us consider $D_i$. Using triangle inequality, we can just prove the convergence of the following:
$$|\frac{1}{t_n} \sum_{i=0}^{n-1} \frac{\varphi_{\Delta_{n,i}^\beta}(X_{t_{i+1}} - X_{t_i})1_{\left \{|X_{t_i}| \le \Delta_{n,i}^{-k} \right \} }}{a^2(X_{t_i})\Delta_{n,i}} r(\theta, X_{t_i} )^2 | = $$
$$ = |\frac{1}{t_n} \sum_{i=0}^{n-1} \Delta_{n,i}^{1 + 2\delta} \frac{\varphi_{\Delta_{n,i}^\beta}(X_{t_{i+1}} - X_{t_i})1_{\left \{|X_{t_i}| \le \Delta_{n,i}^{-k} \right \} }}{a^2(X_{t_i})} R(\theta,1, X_{t_i} )^2 | \le $$
$$\le \Delta_n^{2\delta} \frac{1}{n} \sum_{i=0}^{n-1} | f_{i,n}(X_{t_i}, \theta) \varphi_{\Delta_{n,i}^\beta}(X_{t_{i+1}} - X_{t_i})|, $$
$f_{i,n}(X_{t_i}, \theta) = \frac{ R(\theta,1, X_{t_i} )^2}{a^2(X_{t_i})}$, using also the indicator is always upper bounded by $1$. \\
Also this time the exponent on $\Delta_n$ is positive. We can use the boundedness of $ | \varphi_{\Delta^\beta_{n,i}}|$, the polynomial growth of $f_{i,n}$ and third point of Lemma \ref{lemma 2.1 GLM} in order to get that $ \frac{1}{n} \sum_{i=0}^{n-1}| f_{i,n}(X_{t_i}, \theta) \varphi_{\Delta_{n,i}^\beta}(X_{t_{i+1}} - X_{t_i})|$ is bounded in $L^1$ . It turns
$$|\frac{1}{t_n} \sum_{i=0}^{n-1} \frac{\varphi_{\Delta_{n,i}^\beta}(X_{t_{i+1}} - X_{t_i})1_{\left \{|X_{t_i}| \le \Delta_{n,i}^{-k} \right \} }}{a^2(X_{t_i})\Delta_{n,i}} r(\theta, X_{t_i} )^2 | \xrightarrow{\mathbb{P}} 0. $$
Considering $E_i$, we use again the triangle inequality in order to prove only the convergence to zero of the following:
\begin{equation}
|\frac{1}{t_n} \sum_{i=0}^{n-1} 2 \frac{\varphi_{\Delta_{n,i}^\beta}(X_{t_{i+1}} - X_{t_i})1_{\left \{|X_{t_i}| \le \Delta_{n,i}^{-k} \right \} }}{a^2(X_{t_i})\Delta_{n,i}} \int_{t_i}^{t_{i+1}} b(X_s, \theta_0) ds \, r(\theta, X_{t_i}) |. 
\label{eq: Ei in contrast conv} 
\end{equation}
In the sequel it will be useful to substitute $\int_{t_i}^{t_{i+1}} b(X_s, \theta_0) ds$ with $\Delta_{n,i}\, b(X_{t_i}, \theta_0)$. \\
\begin{equation}
\int_{t_i}^{t_{i+1}} b(X_s, \theta_0) ds = \int_{t_i}^{t_{i+1}} [b(X_s, \theta_0)- b(X_{t_i}, \theta_0) ] ds + \Delta_{n,i}b(X_{t_i}, \theta_0).
\label{eq: rewrite int b} 
\end{equation}
In order to show that the first term is negligible compared to $\Delta_{n,i}$, we consider the following expected value:
$$\sup_{u \in [0, \Delta_{n,i}]} \mathbb{E}[|b(X_{t_i + u}, \theta_0) - b(X_{t_i}, \theta_0)||\mathcal{F}_{t_i}] \le
\sup_{u \in [0, \Delta_{n,i}]} \mathbb{E}[\left \| \frac{\partial b}{\partial x }\right \|_\infty |X_{t_i + u} - X_{t_i}|| \mathcal{F}_{t_i}] \le $$
$$\le c\sup_{u \in [0, \Delta_{n,i}]} \mathbb{E}[|X_{t_i + u} - X_{t_i}||\mathcal{F}_{t_i}].$$
In the last inequality we have used that the derivative of $b$ is supposed bounded. \\
Using Holder inequality we get that it is, for each $p \ge 2$, upper bounded by
$$c\sup_{u \in [0, \Delta_{n,i}]} (\mathbb{E}[|X_{t_i + u} - X_{t_i}|^p| \mathcal{F}_{t_i}])^\frac{1}{p} \le$$
\begin{equation}
\le c \sup_{u \in [0, \Delta_{n,i}]}(|t_i + u - t_i|(1 + |X_{t_i}|^p ))^\frac{1}{p} = R(\theta, \Delta_{n,i}^{\frac{1}{p}}, X_{t_i}).
\label{eq: estim substitution b} 
\end{equation}
Where, in the last inequality, we have used the second point of Lemma \ref{lemma: Moment inequalities}. \\
For $p = 2$, $\mathbb{E}[|b(X_{t_i + u}, \theta_0) - b(X_{t_i}, \theta_0)||\mathcal{F}_{t_i}]  \le R(\theta, \Delta_{n,i}^{\frac{1}{2}}, X_{t_i})$ and therefore
\begin{equation}
\int_{t_i}^{t_{i+1}} \mathbb{E}[|b(X_s, \theta_0)- b(X_{t_i}, \theta_0)||\mathcal{F}_{t_i}] ds \le  \, \int_{t_i}^{t_{i+1}}R(\theta, \Delta_{n,i}^{\frac{1}{2}}, X_{t_i})ds =\, R(\theta, \Delta_{n,i}^{\frac{3}{2}}, X_{t_i}),
\label{eq: expected value int b as R} 
\end{equation}
negligible compared to $\Delta_{n,i}$, that is the order of the second term of \eqref{eq: rewrite int b}. \\
Using \eqref{eq: puissance R} and \eqref{eq: rewrite int b}, \eqref{eq: Ei in contrast conv} can be reformulated as
\begin{align}
 |\frac{1}{t_n} \sum_{i=0}^{n-1} 2 \frac{\varphi_{\Delta_{n,i}^\beta}(X_{t_{i+1}} - X_{t_i})1_{\left \{|X_{t_i}| \le \Delta_{n,i}^{-k} \right \} }\Delta_{n,i}^{\delta}R(\theta, 1, X_{t_i})}{a^2(X_{t_i})} [\Delta_{n,i} \,b(X_{t_i}, \theta_0) + \nonumber \\
 +\int_{t_i}^{t_{i+1}}[ b(X_s, \theta_0) - b(X_{t_i}, \theta_0)] ds ]|.
\label{eq: reformulation Ei} 
\end{align}
The first term is upper bounded by
$$\Delta_n^{\delta}\frac{1}{n} \sum_{i=0}^{n-1} | f_{i,n}(X_{t_i}, \theta) \varphi_{\Delta_{n,i}^\beta}(X_{t_{i+1}} - X_{t_i})|,$$
where $f_{i,n}(X_{t_i}, \theta)= \frac{2b(X_{t_i}, \theta_0) R(\theta,1, X_{t_i} )}{a^2(X_{t_i})}$. \\
Again, the exponent on $\Delta_n$ is positive and $\frac{1}{n} \sum_{i=0}^{n-1} | f_{i,n}(X_{t_i}, \theta) \varphi_{\Delta_{n,i}^\beta}(X_{t_{i+1}} - X_{t_i})|$ is bounded in $L^1$ using the boundedness of $\varphi_{\Delta_{n,i}^\beta}$, the polynomial growth of $f_{i,n}$ and the third point of Lemma \ref{lemma 2.1 GLM}. \\
Concerning the second term of \eqref{eq: reformulation Ei}, we observe it is upper bounded by
$$\Delta_n^{\delta}\frac{1}{n\Delta_n} \sum_{i=0}^{n-1} | f_{i,n}(X_{t_i}, \theta) \int_{t_i}^{t_{i+1}}[ b(X_s, \theta_0) - b(X_{t_i}, \theta_0)] ds \, \varphi_{\Delta_{n,i}^\beta}(X_{t_{i+1}} - X_{t_i})|,$$
where $f_{i,n}(X_{t_i}, \theta)= \frac{2R(\theta,1, X_{t_i} )}{a^2(X_{t_i})}$.
The exponent on $\Delta_n$ is still positive and $\frac{1}{n\Delta_n} \sum_{i=0}^{n-1} | f_{i,n}(X_{t_i}, \theta) \int_{t_i}^{t_{i+1}}[ b(X_s, \theta_0) - b(X_{t_i}, \theta_0)] ds \, \varphi_{\Delta_{n,i}^\beta}(X_{t_{i+1}} - X_{t_i})|$ is bounded in $L^1$. Indeed, 
$$\frac{1}{n\Delta_n} \sum_{i=0}^{n-1} \mathbb{E}[| f_{i,n}(X_{t_i}, \theta) \int_{t_i}^{t_{i+1}}[ b(X_s, \theta_0) - b(X_{t_i}, \theta_0)] ds \, \varphi_{\Delta_{n,i}^\beta}(X_{t_{i+1}} - X_{t_i})|] \le $$
$$ \le \frac{c}{n\Delta_n} \sum_{i=0}^{n-1} \mathbb{E}[ | f_{i,n}(X_{t_i}, \theta) \int_{t_i}^{t_{i+1}}[ b(X_s, \theta_0) - b(X_{t_i}, \theta_0)] ds |]=$$
\begin{equation}
= \frac{c}{n\Delta_n} \sum_{i=0}^{n-1} \mathbb{E}[ | f_{i,n}(X_{t_i}, \theta) \mathbb{E}[\int_{t_i}^{t_{i+1}}[ b(X_s, \theta_0) - b(X_{t_i}, \theta_0)] ds \,| \mathcal{F}_{t_i}]|] = \frac{c}{n\Delta_n} \sum_{i=0}^{n-1} \mathbb{E}[ | f_{i,n}(X_{t_i}, \theta) R(\theta, \Delta_{n,i}^\frac{3}{2}, X_{t_i})|],
\label{eq: bound in L1} 
\end{equation}
where we have used the definition of conditional expectation and \eqref{eq: expected value int b as R}. \\
From \eqref{propriety power R}, we can upper bound \eqref{eq: bound in L1} by
$\Delta_n^\frac{1}{2} \frac{1}{n} \sum_{i=0}^{n-1} \mathbb{E}[ | f_{i,n}(X_{t_i}, \theta) R(\theta, 1, X_{t_i})|]. $ \\ 
The exponent of $\Delta_n$ is clearly positive and $\frac{1}{n} \sum_{i=0}^{n-1} \mathbb{E}[ | f_{i,n}(X_{t_i}, \theta) R(\theta, 1, X_{t_i})|]$ is bounded  using again the polynomial growth of both $f_{n,i}$ and $R$ and the third point of Lemma \ref{lemma 2.1 GLM}. \\
We have obtained the wanted convergence. \\
\\
Let us now consider the main terms of \eqref{eq: difference contraste convergence}: we will show that they converge to $\int_\mathbb{R} \frac{(b(x, \theta) - b(x, \theta_0))^2}{a^2(x)} \pi(dx)$. \\
In order to do it, we want to replace $\int_{t_i}^{t_{i+1}} b(X_s, \theta_0) ds$ with $\Delta_{n,i}\, b(X_{t_i}, \theta_0)$ in \eqref{eq: difference contraste convergence}, getting:
$$\Delta_{n,i}^2[b(X_{t_i}, \theta)^2-  b(X_{t_i}, \theta_0)^2] + 2\Delta_{n,i}^2b(X_{t_i}, \theta_0)[ b(X_{t_i}, \theta_0) -  b(X_{t_i}, \theta)] = \Delta_{n,i}^2 [ b(X_{t_i}, \theta) -  b(X_{t_i}, \theta_0)]^2.$$
Hence, we can reformulate \eqref{eq: difference contraste convergence} adding and subtracting $\Delta_{n,i}b(X_{t_i}, \theta_0)$. We obtain
$$ \frac{U_n(\theta) - U_n(\theta_0)}{t_n} = \frac{1}{t_n} \sum_{i=0}^{n-1} \frac{\varphi_{\Delta_{n,i}^\beta}(X_{t_{i+1}} - X_{t_i})1_{\left \{|X_{t_i}| \le \Delta_{n,i}^{-k} \right \} }}{a^2(X_{t_i})}\Delta_{n,i}[ b(X_{t_i}, \theta) -  b(X_{t_i}, \theta_0)]^2 +$$
\begin{equation}
+ \frac{1}{t_n} \sum_{i=0}^{n-1} \frac{2 \varphi_{\Delta_{n,i}^\beta}(X_{t_{i+1}} - X_{t_i})1_{\left \{|X_{t_i}| \le \Delta_{n,i}^{-k} \right \} }}{a^2(X_{t_i})} (\int_{t_i}^{t_{i+1}} [b(X_s, \theta_0) - b(X_{t_i}, \theta_0)] ds ) \,[b(X_{t_i}, \theta_0) - b(X_{t_i}, \theta)] + R_i,
\label{eq: difference constraste reformulated} 
\end{equation}
where $R_i$ represents the rest terms, for which we have already shown the convergence to $0$ in probability.
The second term of \eqref{eq: difference constraste reformulated} goes to $0$ in $L^1$, in fact
$$\mathbb{E}[|\frac{1}{t_n} \sum_{i=0}^{n-1} f(X_{t_i}, \theta)\varphi_{\Delta_{n,i}^\beta}(X_{t_{i+1}} - X_{t_i})(\int_{t_i}^{t_{i+1}} (b(X_s, \theta_0) ds- b(X_{t_i}, \theta_0))ds|] =$$
$$ = \mathbb{E}[|\frac{1}{t_n} \sum_{i=0}^{n-1} f(X_{t_i}, \theta)\mathbb{E}[\varphi_{\Delta_{n,i}^\beta}(X_{t_{i+1}} - X_{t_i})(\int_{t_i}^{t_{i+1}} (b(X_s, \theta_0) - b(X_{t_i}, \theta_0))ds | \mathcal{F}_{t_i}]|],$$
With $f(X_{t_i}, \theta) := \frac{2 (b(X_{t_i}, \theta_0) - b(X_{t_i}, \theta)) }{ a^2(X_{t_i})} 1_{\left \{|X_{t_i}| \le \Delta_{n,i}^{-k} \right \} }$. \\
Using that $\varphi_{\Delta_{n,i}^\beta}(X_{t_{i+1}} - X_{t_i})$ is bounded by a constant and the estimation \eqref{eq: expected value int b as R}, we get that it is upper bounded by 
$$\mathbb{E}[|\frac{1}{n\Delta_n} \sum_{i=0}^{n-1} f(X_{t_i}, \theta)R(\theta, \Delta_{n,i}^{\frac{3}{2}}, X_{t_i})|] \le \Delta_n^\frac{1}{2} \frac{1}{n} \sum_{i=0}^{n-1} \mathbb{E}[|f(X_{t_i}, \theta)R(\theta, 1, X_{t_i})|],$$
where in the last inequality we have used \eqref{propriety power R}, the triangle inequality and that $|\Delta_{n,i}| \le \Delta_n$.
Using the third point of Lemma \ref{lemma 2.1 GLM}, we obtain that $\frac{1}{n} \sum_{i=0}^{n-1} |f(X_{t_i}, \theta)R(\theta, 1, X_{t_i})|$ is bounded in $L^1$ and so the convergence wanted. \\
To conclude, we use the second point of Proposition \ref{prop: LT1} on the first term of \eqref{eq: difference constraste reformulated}. It yields 
$$\frac{1}{t_n} \sum_{i=0}^{n-1} \frac{\varphi_{\Delta_{n,i}^\beta}(X_{t_{i+1}} - X_{t_i})1_{\left \{|X_{t_i}| \le \Delta_{n,i}^{-k} \right \} }}{a^2(X_{t_i})}\Delta_{n,i} (b(X_{t_i}, \theta_0) - b(X_{t_i}, \theta))^2 \xrightarrow{\mathbb{P}} \int_\mathbb{R} \frac{(b(x, \theta) - b(x, \theta_0))^2}{a^2(x)} \pi(dx).$$
Therefore, $$\frac{U_n (\theta) - U_n (\theta_0)}{t_n}\xrightarrow{\mathbb{P}} \int_\mathbb{R} \frac{(b(x, \theta) - b(x, \theta_0))^2}{a^2(x)} \pi(dx).$$
\end{proof}

\begin{remark}
We observe that the contrast function does not converge: $\forall \theta \in \Theta$
$$\lim_{n \rightarrow \infty} \frac{U_n(\theta)}{t_n} =\infty.$$
It happens because, in the expansion $$X_{t_{i+1}} - m_\theta(X_{t_i}) = \zeta_i + \int_{t_i}^{t_{i+1}} b(X_s, \theta_0) ds - \Delta_{n,i}\, b(X_{t_i}, \theta) + R(\theta, \Delta_{n,i}^{1 + \delta}, X_{t_i}),$$ $\zeta_i$ is of the order $\Delta_n^{\frac{1}{2}}$ while the order of the part dependent on $\theta$ is $\Delta_n$. \\
That is the reason why we consider the difference between $U_n(\theta)$ and $U_n(\theta_0)$: stressing that $\zeta_i$ does not depend on $\theta$, we get that in the difference it does not contribute anymore. \\
The asymptotic behavior of $(U_n(\theta) - U_n(\theta_0))$ is therefore governed by the part depending on $\theta$. \\

\end{remark}

\subsection{Consistency of the estimator}
In order to prove the consistency of $\hat{\theta}_n$, we need that the convergence (\ref{eq: contrast convergence}) takes place in probability uniformly in the parameter $\theta$, we want therefore to show the uniformity of the convergence in $\theta$. \\
Let $S_n(\theta): = \frac{U_n(\theta) - U_n(\theta_0)}{t_n}$ ; we regard this as a random element taking values in $(C(\Theta), \left \| . \right \|_\infty)$. It suffices to prove the tightness of this sequence, to do it we need an explicit approximation of $\dot{m}_{\theta,h}$. Such an approximation, together with the approximation of $\ddot{m}_{\theta,h}$, will be also useful to study the asymptotic behavior of the derivatives of the contrast function. In the following proposition we study their asymptotic expansions as $\Delta_{n,i} \rightarrow 0$ : \\
\begin{proposition}
 \label{P:expansion_m_dot} 
Suppose that Assumptions 1 to 4 and 7 hold, with $\alpha \in (0,2)$, $\alpha \neq 1$ and $\beta \in (0, \frac{1}{1 + \alpha} - \epsilon)$. Then, for $|y| \le h^{- k_0}$ (where $k_0$ is the same as in Theorem \ref{th: dL mtheta alpha <1} or \ref{th: dL mtheta alpha  > 1}, according to $\alpha <1$ or $\alpha > 1$),
\begin{equation}
\dot{m}_{\theta,h}(y)= h \dot{b}(y, \theta) + R(\theta, h^{\frac{3}{2} \land (2 - \alpha\beta -\epsilon - \beta)}, y )
\label{eq: dl dot mtheta} 
\end{equation}
and
\begin{equation}
\ddot{m}_{\theta,h}(y)= h \ddot{b}(y, \theta) + R(\theta, h^{\frac{3}{2} \land (2 - \alpha\beta -\epsilon - \beta)}, y ).
\label{eq: dl ddot mtheta} 
\end{equation}
\label{prop: dl dotm ddotm} 
\end{proposition}

\begin{remark}
It is also possible to show that
\begin{equation}
|\dddot{m}_{\theta,h}(y)| = R(\theta, h,y).  
\label{eq: estimation 3dot m } 
\end{equation}
\end{remark}
The proposition above will be proved in the Appendix \ref{S:Proof_derivatives}, where we will also justify \eqref{eq: estimation 3dot m }. We can now show the tightness of $S_n(\theta)$: \\
\begin{lemma}
Suppose that Assumptions 1 - 8 and $A_\beta$ are satisfied. Then 
$$S_n(\theta): = \frac{U_n(\theta) - U_n(\theta_0)}{t_n}$$ is a tight sequence in $(C(\Theta), \left \| . \right \|_\infty)$.
\label{lemma: tightness Sn}
\begin{proof}
In the proof we use the notation of Section $5.3$ and especially of the proof of Lemma \ref{lemma: convergence contrast}.  
Since the sum of tight sequences is also tight, we can see $S_n(\theta)$ as $S_{n1}(\theta) + S_{n2}(\theta)$, where
$$S_{n1}(\theta): = \frac{1}{t_n} \sum_{i=0}^{n-1} \frac{\varphi_{\Delta_{n,i}^\beta}(X_{t_{i+1}} - X_{t_i}) 1_{\left \{|X_{t_i}| \le \Delta_{n,i}^{-k} \right \} }}{a^2(X_{t_i})\Delta_{n,i}}[\Delta_{n,i}^2(b(X_{t_i}, \theta)^2 - b(X_{t_i}, \theta_0)^2) + $$
$$ + 2\Delta_{n,i}\int_{t_i}^{t_{i+1}} b(X_s, \theta_0) ds(b(X_{t_i}, \theta_0) - b(X_{t_i}, \theta)) + C_i + D_i + E_i] +$$
$$+ \frac{2}{t_n} \sum_{i=0}^{n-1} \frac{ 1_{\left \{|X_{t_i}| \le \Delta_{n,i}^{-k} \right \} } \mathbb{E}[\zeta_i \varphi_{\Delta_{n,i}^\beta}(X_{t_{i+1}} - X_{t_i})| \mathcal{F}_{t_i}] }{a^2(X_{t_i})\Delta_{n,i}} ( \Delta_{n,i}(b(X_{t_i}, \theta) - b(X_{t_i}, \theta_0)) + (r(\theta, X_{t_i}) - r(\theta_0, X_{t_i}))),$$
$$S_{n2}(\theta):= \frac{2}{t_n} \sum_{i=0}^{n-1} \frac{ 1_{\left \{|X_{t_i}| \le \Delta_{n,i}^{-k} \right \} }}{a^2(X_{t_i})\Delta_{n,i}}[ \zeta_i \varphi_{\Delta_{n,i}^\beta}(X_{t_{i+1}} - X_{t_i}) + $$
$$ - \mathbb{E}[\zeta_i \varphi_{\Delta_{n,i}^\beta}(X_{t_{i+1}} - X_{t_i})| \mathcal{F}_{t_i}]] ( \Delta_{n,i}(b(X_{t_i}, \theta) - b(X_{t_i}, \theta_0)) + (r(\theta, X_{t_i}) - r(\theta_0, X_{t_i}))),$$
and show the tightness of the two sequences individually, using two different criteria. \\
In order to prove that $S_{n1}$ is tight, we want to show that $ \sup_n \mathbb{E}[\sup_{\theta \in \Theta}| \frac{\partial}{ \partial \theta} S_{n1}(\theta)|] < \infty.$
As concerns $S_{n2}(\theta)$, according to Theorem 20 in Appendix 1 from Ibragimov and Has' Minskii \hyperref[csl:21]{(Ibragimov \& Has' Minskii, 2013)}, we should verify the following: for some positive constant $H$ independent of $n$, ,
\begin{equation}
\mathbb{E}[(S_{n2}(\theta))^2] \le H \qquad \forall \theta \in \Theta ,
\label{eq: 1 criterion tightness} 
\end{equation}
\begin{equation}
\mathbb{E}[(S_{n2}(\theta_1)- S_{n2}(\theta_2))^2] \le H(\theta_1 - \theta_2)^2 \qquad \forall \theta_1, \theta_2 \in \Theta.
\label{eq: 2 criterion tightness} 
\end{equation}
The derivative that we want to estimate is, using the expressions of $C_i$, $D_i$ and $E_i$,
\begin{equation}
\frac{\partial S_{n1} (\theta) }{\partial \theta}= \frac{1}{t_n} \sum_{i=0}^{n-1} \frac{\varphi_{\Delta_{n,i}^\beta}(X_{t_{i+1}} - X_{t_i}) 1_{\left \{|X_{t_i}| \le \Delta_{n,i}^{-k} \right \} }}{a^2(X_{t_i})\Delta_{n,i}}[2\Delta_{n,i}^2b(X_{t_i}, \theta) \dot{b}(X_{t_i}, \theta) + 
\label{eq: derivative Sn1} 
\end{equation}
$$ + 2\Delta_{n,i}\int_{t_i}^{t_{i+1}} b(X_s, \theta_0) ds(- \dot{b}(X_{t_i}, \theta)) -2\Delta_{n,i}(\dot{b}r)(X_{t_i}, \theta)-2\Delta_{n,i}(b\dot{r})(\theta, X_{t_i}) + 2(\dot{r}r)(\theta, X_{t_i}) +$$
$$ +2\dot{r}(\theta, X_{t_i})\int_{t_i}^{t_{i+1}} b(X_s, \theta_0) ds ] + \frac{2}{t_n} \sum_{i=0}^{n-1} \frac{\mathbb{E}[\zeta_i \varphi_{\Delta_{n,i}^\beta}(X_{t_{i+1}} - X_{t_i})| \mathcal{F}_{t_i}] 1_{\left \{|X_{t_i}| \le \Delta_{n,i}^{-k} \right \} }}{a^2(X_{t_i})\Delta_{n,i}}(\dot{r}(\theta, X_{t_i}) + \Delta_{n,i}\dot{b}(X_{t_i}, \theta)).  $$
Using triangle inequality, we can just estimate each term in $L^1$ norm. \\
Using the polynomial growth of both $b$ and $\dot{b}$, the fact that $\varphi$ and the indicator function are bounded, that $a^2$ is bigger than a constant from Assumption 5 and that $|\Delta_{n,i}| \le \Delta_n$, we get the first term of \eqref{eq: derivative Sn1} is upper bounded by
$$\mathbb{E}[\sup_{\theta \in \Theta}| \frac{1}{n} \sum_{i=0}^{n-1} (1 + |X_{t_i}|^c ) |],$$
that is bounded by the third point of Lemma \ref{lemma 2.1 GLM}. \\
On the second term of \eqref{eq: derivative Sn1} we can use that $\varphi$ and the indicator function are bounded, that $a^2$ is bigger than a constant from Assumption 5, that both $b$ and $\dot{b}$ have polynomial growth, from the integral we get a $|\Delta_{n,i}|$ (using \eqref{eq: rewrite int b} and \eqref{eq: expected value int b as R}) that is smaller than $\Delta_n$ and so we have just to use the third point of Lemma \ref{lemma 2.1 GLM} in order to say that the moments of $X$ are bounded. Hence
$$\mathbb{E}[\sup_{\theta \in \Theta} |\frac{1}{t_n} \sum_{i=0}^{n-1}\frac{\varphi_{\Delta_{n,i}^\beta}(X_{t_{i+1}} - X_{t_i}) 1_{\left \{|X_{t_i}| \le \Delta_{n,i}^{-k} \right \} }}{a^2(X_{t_i})\Delta_{n,i}}2\Delta_{n,i}\int_{t_i}^{t_{i+1}} b(X_s, \theta_0) ds(- \dot{b}(X_{t_i}, \theta)) |] \le c.$$
Concerning the third and the fourth terms of \eqref{eq: derivative Sn1}, we use again that $\varphi$ and the indicator function are bounded, that $a^2$ is bigger than a constant from Assumption 5 and that $\dot{b}$ has polynomial growth. 
We recall that
\begin{equation}
r(\theta, X_{t_i}) 1_{\left \{|X_{t_i}| \le \Delta_{n,i}^{-k} \right \} } = R(\theta, \Delta_{n,i}^{1 + \delta }, X_{t_i})= \Delta_{n,i}^{1 + \delta } R(\theta, 1, X_{t_i}), 
\label{eq: r x indicatrice} 
\end{equation}
using \eqref{eq: puissance R}.
By the definition \eqref{eq: definition r} and the development \eqref{eq: dl dot mtheta} of $\dot{m}_{\theta}$ we get also the following estimation:
\begin{equation}
\sup_{\theta \in \Theta}|\dot{r}(\theta, x) |\le \Delta_{n,i}(1 + |x|^c).
\label{eq: estimation dot r} 
\end{equation}
We obtain in this way a $|\Delta_{n,i}|$ that is always smaller than $\Delta_n$ and so we can simplify the $\Delta_n$ in the denominator.  Now we use the third point of Lemma \ref{lemma 2.1 GLM} and we get also this time that the expectation is bounded. \\
Also on the fifth we use that $\varphi$ and the indicator function are bounded, $a^2$ is bigger than a constant from Assumption 5, \eqref{eq: r x indicatrice} and \eqref{eq: estimation dot r} on $\dot{r}$. Therefore the fifth term of \eqref{eq: derivative Sn1} is upper bounded by $\Delta_{n}^{\delta}\mathbb{E}[| \frac{1}{n} \sum_{i=0}^{n-1} (1 + |X_{t_i}|^c ) |].$ \\
Since the exponent on $\Delta_n$ is positive and by the third point of Lemma \ref{lemma 2.1 GLM}, it is upper  bounded by a constant. \\
As concerns the expected value of the sixth term of \eqref{eq: derivative Sn1}, we use again that $\varphi$ and the indicator function are both bounded, $a^2$ is bigger than a constant from Assumption 5 and \eqref{eq: estimation dot r} on $\dot{r}$. Moreover, we get a $|\Delta_{n,i}|$ from the integral (using \eqref{eq: rewrite int b} and \eqref{eq: expected value int b as R}). The third point of Lemma \ref{lemma 2.1 GLM} is sufficient to assure the boundedness of the considered expectation. \\
Let us now consider 
$$\mathbb{E}[\sup_{\theta \in \Theta} |\frac{2}{t_n} \sum_{i=0}^{n-1}\frac{\varphi_{\Delta_{n,i}^\beta}(X_{t_{i+1}} - X_{t_i}) 1_{\left \{|X_{t_i}| \le \Delta_{n,i}^{-k} \right \} }}{a^2(X_{t_i})}\dot{b}(X_{t_i}, \theta)\mathbb{E}[\zeta_i \varphi_{\Delta_{n,i}^\beta}(X_{t_{i+1}} - X_{t_i})| \mathcal{F}_{t_i}]|].$$
By the boundedness of $\varphi$, the Assumption 5 on $a$ and the polynomial growth of $\dot{b}$, it is upper bounded by
$$\mathbb{E}[\sup_{\theta \in \Theta} |\frac{1}{n\Delta_n} \sum_{i=0}^{n-1}\mathbb{E}[ 1_{\left \{|X_{t_i}| \le \Delta_{n,i}^{-k} \right \} }\zeta_i \varphi_{\Delta_{n,i}^\beta}(X_{t_{i+1}} - X_{t_i})| \mathcal{F}_{t_i}](1 + |X_{t_i}|^c )|] \le $$
$$\le \mathbb{E}[\sup_{\theta \in \Theta} |\frac{1}{n\Delta_n} \sum_{i=0}^{n-1}R(\theta, \Delta_{n,i}^{(1 + \delta) \land \frac{3}{2}})(1 + |X_{t_i}|^c )|] \le c\Delta_n^{\delta \land \frac{1}{2}},  $$
where we have used \eqref{eq: convergence expected value prop 2}, $|\Delta_{n,i}| \le \Delta_n$ and the third point of Lemma \ref{lemma 2.1 GLM}. Since the exponent on $\Delta_n$ is positive, it is bounded by a constant.\\
In order to conclude the proof of the $S_{n1}$'s tightness, we observe that by the boundedness of both $\varphi$ and the indicator function, the Assumption 5 on $a$ and \eqref{eq: estimation dot r} on $\dot{r}$ we get
$$\mathbb{E}[\sup_{\theta \in \Theta}| \frac{2}{t_n} \sum_{i=0}^{n-1}\frac{\varphi_{\Delta_{n,i}^\beta}(X_{t_{i+1}} - X_{t_i}) 1_{\left \{|X_{t_i}| \le \Delta_{n,i}^{-k} \right \} }}{a^2(X_{t_i})\Delta_{n,i}}\dot{r}(\theta, X_{t_i})\mathbb{E}[\zeta_i \varphi_{\Delta_{n,i}^\beta}(X_{t_{i+1}} - X_{t_i})| \mathcal{F}_{t_i}]|] \le $$
$$\le \mathbb{E}[| \frac{c}{n\Delta_n} \sum_{i=0}^{n-1}\mathbb{E}[ 1_{\left \{|X_{t_i}| \le \Delta_{n,i}^{-k} \right \} } \zeta_i \varphi_{\Delta_{n,i}^\beta}(X_{t_{i+1}} - X_{t_i})| \mathcal{F}_{t_i}](1 + |X_{t_i}|^c )], $$
on which we can act exactly like above, getting the wanted boundedness. \\ \\
Let us now consider $S_{n2}$. In order to prove \eqref{eq: 2 criterion tightness}, we observe that
\begin{equation}
\mathbb{E}[(S_{n2}(\theta_1) - S_{n2}(\theta_2))^2] \le  \frac{c}{n^2 \Delta_n^2} \mathbb{E}[(\sum_{i=0}^{n-1}\frac{ 1_{\left \{|X_{t_i}| \le \Delta_{n,i}^{-k} \right \} }}{a^2(X_{t_i})\Delta_{n,i}}[ \zeta_i \varphi_{\Delta_{n,i}^\beta}(X_{t_{i+1}} - X_{t_i}) + 
\label{eq: Sn2 martingale} 
\end{equation}
$$ - \mathbb{E}[\zeta_i \varphi_{\Delta_{n,i}^\beta}(X_{t_{i+1}} - X_{t_i})| \mathcal{F}_{t_i}]] ( \Delta_{n,i}(b(X_{t_i}, \theta_2) - b(X_{t_i}, \theta_1)) + r(\theta_1, X_{t_i}) - r(\theta_2, X_{t_i})))^2]$$
By the building the sum is a square integrable martingale. The Pythagoras'\ theorem on a square integrable martingale yields that \eqref{eq: Sn2 martingale} is equal to
\begin{equation}
\frac{c}{n^2 \Delta_n^2} \sum_{i=0}^{n-1} \mathbb{E}[\frac{ 1_{\left \{|X_{t_i}| \le \Delta_{n,i}^{-k} \right \} }}{a^4(X_{t_i})\Delta^2_{n,i}}[ \zeta_i \varphi_{\Delta_{n,i}^\beta}(X_{t_{i+1}} - X_{t_i}) + 
\label{eq: equal Sn2 martingale} 
\end{equation}
$$- \mathbb{E}[\zeta_i \varphi_{\Delta_{n,i}^\beta}(X_{t_{i+1}} - X_{t_i})| \mathcal{F}_{t_i}]]^2 ( \Delta_{n,i}(b(X_{t_i}, \theta_2) - b(X_{t_i}, \theta_1)) + r(\theta_1, X_{t_i}) - r(\theta_2, X_{t_i}))^2].$$
We now observe that
$$( \Delta_{n,i}(b(X_{t_i}, \theta_2) - b(X_{t_i}, \theta_1)) + r(\theta_1, X_{t_i}) - r(\theta_2, X_{t_i}))^2 \le c \Delta^2_{n,i}(b(X_{t_i}, \theta_2) - b(X_{t_i}, \theta_1))^2 +$$
$$ + c(r(\theta_1, X_{t_i}) - r(\theta_2, X_{t_i}))^2 \le c \Delta^2_{n,i} \dot{b}(X_{t_i}, \theta_u)^2(\theta_1 - \theta_2)^2 +c \dot{r}(\theta_u,X_{t_i})^2(\theta_1 - \theta_2)^2, $$
where $\theta_u \in [\theta_1, \theta_2]$. Using \eqref{eq: estimation dot r}, it is upper bounded by
\begin{equation}
c \Delta^2_{n,i} [\dot{b}(X_{t_i}, \theta_u)^2 + (1 + |X_{t_i}|^c)^2](\theta_1 - \theta_2)^2 .
\label{eq: estimation dotb ans dotr in sn2} 
\end{equation}
Replacing \eqref{eq: estimation dotb ans dotr in sn2} in \eqref{eq: equal Sn2 martingale}, using that the indicator function is bounded by a constant, the Assumption 5 on $a$ and that $\dot{b}$ has polynomial growth, we get that \eqref{eq: equal Sn2 martingale} is upper bounded by
$$\frac{c}{n^2 \Delta_n^2}\sum_{i=0}^{n-1} \mathbb{E}[(\zeta_i \varphi_{\Delta_{n,i}^\beta}(X_{t_{i+1}} - X_{t_i})- \mathbb{E}[\zeta_i \varphi_{\Delta_{n,i}^\beta}(X_{t_{i+1}} - X_{t_i})| \mathcal{F}_{t_i}])^2(1 + |X_{t_i}|^c)^2](\theta_1 - \theta_2)^2 =$$
\begin{equation}
= \frac{c}{n^2 \Delta_n^2}\sum_{i=0}^{n-1} \mathbb{E}[\mathbb{E}[(\zeta_i \varphi_{\Delta_{n,i}^\beta}(X_{t_{i+1}} - X_{t_i})- \mathbb{E}[\zeta_i \varphi_{\Delta_{n,i}^\beta}(X_{t_{i+1}} - X_{t_i})| \mathcal{F}_{t_i}])^2| \mathcal{F}_{t_i}](1 + |X_{t_i}|^c)^2](\theta_1 - \theta_2)^2, 
\label{eq: eq martingale cond Sn2} 
\end{equation}
by the definition of conditional expected value and the measurability of $X_{t_i}$. \\
We observe that $\mathbb{E}[(\zeta_i \varphi_{\Delta_{n,i}^\beta}(X_{t_{i+1}} - X_{t_i})- \mathbb{E}[\zeta_i \varphi_{\Delta_{n,i}^\beta}(X_{t_{i+1}} - X_{t_i})| \mathcal{F}_{t_i}])^2| \mathcal{F}_{t_i}]$ is the conditional variance of $\zeta_i \varphi$ and so it is always smaller then $\mathbb{E}[\zeta_i^2 \varphi^2_{\Delta_{n,i}^\beta}(X_{t_{i+1}} - X_{t_i})| \mathcal{F}_{t_i}]$ that is, using \eqref{eq: conv squared prop 2}, $R(\theta, \Delta_{n,i}, X_{t_i})$. We get that \eqref{eq: eq martingale cond Sn2} is upper bounded by $$\frac{1}{n^2 \Delta_n^2}\sum_{i=0}^{n-1} \mathbb{E}[R(\theta, \Delta_{n,i}, X_{t_i})(1 + |X_{t_i}|^c)^2](\theta_1 - \theta_2)^2 \le \frac{1}{n \Delta_n}c(\theta_1 - \theta_2)^2, $$
where in the last inequality we have used \eqref{propriety power R} in order to say that $R(\theta, \Delta_{n,i}, X_{t_i})= \Delta_{n,i} R(\theta, 1,X_{t_i}) $, the fact that $|\Delta_{n,i}| \le \Delta_n$, the natural polynomial growth of the function derived from its definition \eqref{eq: definition R} and the third point of Lemma \ref{lemma 2.1 GLM} in order to assure the boundedness of the expected value. \\
Hence, recalling that $n \Delta_n \rightarrow \infty$, we get \eqref{eq: 2 criterion tightness} since $\frac{1}{n \Delta_n}c(\theta_1 - \theta_2)^2 \le c(\theta_1 - \theta_2)^2$. \\
\\
Concerning \eqref{eq: 1 criterion tightness}, we act exactly like we have already done in order to prove \eqref{eq: 2 criterion tightness}, getting
$\mathbb{E}[(S_{n2}(\theta))^2] \le c(\theta - \theta_0)^2$. $\Theta$ is a compact set and so $\Theta$'s diameter $d: = \sup_{\theta_1, \theta_2 \in \Theta } |\theta_1 - \theta_2| $ is $< \infty$. We therefore deduce \eqref{eq: 1 criterion tightness}: 
$c(\theta - \theta_0)^2 \le  cd^2 \le c.$ \\
The tightness of $S_n(\theta) = \frac{U_n(\theta) - U_n(\theta_0)}{t_n}$ follows.
\end{proof}
\end{lemma}

We are now ready to show the consistence of the estimator $\hat{\theta}_n := arg\min_{\theta \in \Theta} U_n(\theta)$. \\
We want to prove that $\hat{\theta}_n \xrightarrow{\mathbb{P}} \theta_0$ when $ n \rightarrow \infty$, that is equivalent to show that $\forall \left \{ \hat{\theta}_{n_k} \right \} \subset \hat{\theta}_n$, $\exists \left \{ \hat{\theta}_{n_{k_j}} \right \} \subset \left \{ \hat{\theta}_{n_k} \right \} $ such that $\hat{\theta}_{n_{k_j}} \rightarrow \theta_0$ a.s. \\
Let $ \left \{ \hat{\theta}_{n_k} \right \}$ be a subsequence of $ \left \{ \hat{\theta}_{n} \right \}$. By the uniform convergence in probability of the contrast function given by Lemma \ref{lemma: convergence contrast} and Lemma \ref{lemma: tightness Sn}, we get the a.s. convergence along some subsequence of $n_k$, denoted $n_{k_j}$:
$$\sup_{\theta \in \Theta}|\frac{U_{n_{k_j}}(\theta) - U_{n_{k_j}}(\theta_0)}{t_{n_{k_j}}} - l(\theta, \theta_0)| \xrightarrow{a. s.} 0, \quad n_{k_j} \rightarrow\infty,$$
where $l(\theta, \theta_0) = \int_\mathbb{R} \frac{(b(x, \theta) - b(x, \theta_0))^2}{a^2(x)} \pi(dx)  \ge 0$. \\
Now, for fixed $\omega \in \Omega$, thanks to the compactness of $\Theta$, there exists a subsequence of $n_{k_j}$, that we still denote $n_{k_j}$, and a $\theta_\infty$ such that $\hat{\theta}_{n_{k_j}} \rightarrow \theta_\infty$. \\
Since the mapping $\theta \mapsto l(\theta, \theta_0)$ is continuous, we have
$l(\hat{\theta}_{n_{k_j}}, \theta_0) \rightarrow l(\theta_\infty, \theta_0).$\\
Then, by the definition of $\hat{\theta}_n$ as the argmin of $U_n(\theta)$, we have
$$0 \ge \frac{U_{n_{k_j}}(\hat{\theta}_{n_{k_j}}) - U_{n_{k_j}}(\theta_0)}{t_{n_{k_j}}} \rightarrow l(\theta_\infty, \theta_0) \ge 0 $$
and so $l(\theta_\infty, \theta_0)=0$.
The Assumption 6 of identifiability leads that $\theta_\infty = \theta_0$. \\
This implies that any convergent subsequence of $\hat{\theta}_n$ tends to $\theta_0$; this means the consistency of $\hat{\theta}_n$.

\subsection{Contrast's derivatives convergence}
We are now ready to show the convergence of the derivative of the contrast function through the following lemma: \\
\begin{lemma}
Suppose that Assumptions 1 - 8 and $A_\beta$ are satisfied. Then 
$$\frac{\dot{U}_n(\theta_0)}{\sqrt[]{t_n}} \xrightarrow{\mathcal{L}} N(0, 4 \int_\mathbb{R}(\frac{\dot{b}(x, \theta_0)}{a(x)})^2 \, \pi(dx)). $$
\label{lemma dot Un convergence} 
\end{lemma}
\begin{proof}
We recall that 
$$U_n(\theta_0)= \sum_{i =0}^{n -1} \frac{(X_{t_{i+1}} - m_{\theta_0}(X_{t_i}))^2}{a^2(X_{t_i})\Delta_{n,i}} \, \varphi_{\Delta_{n,i}^\beta}(X_{t_{i+1}} - X_{t_i})1_{\left \{|X_{t_i}| \le \Delta_{n,i}^{-k} \right \} },$$
hence
\begin{equation}
\dot{U}_n(\theta_0)= 2 \sum_{i =0}^{n -1} \frac{(X_{t_{i+1}} - m_{\theta_0}(X_{t_i})) \dot{m}_{\theta_0}(X_{t_i})}{a^2(X_{t_i})\Delta_{n,i}} \, \varphi_{\Delta_{n,i}^\beta}(X_{t_{i+1}} - X_{t_i})1_{\left \{|X_{t_i}| \le \Delta_{n,i}^{-k} \right \} }.
\label{eq: dotUn(theta0)} 
\end{equation}
It means that
\begin{align}
\frac{\dot{U}_n(\theta_0)}{\sqrt[]{t_n}} = \frac{2}{\sqrt[]{t_n}}\sum_{i =0}^{n -1} (X_{t_{i+1}} - m_{\theta_0}(X_{t_i})) \frac{\dot{b}(X_{t_i}, \theta_0)}{a^2(X_{t_i})}\varphi_{\Delta_{n,i}^\beta}(X_{t_{i+1}} - X_{t_i})1_{\left \{|X_{t_i}| \le \Delta_{n,i}^{-k} \right \} } + \nonumber \\
+ \frac{2}{\sqrt[]{t_n}}\sum_{i =0}^{n -1} (X_{t_{i+1}} - m_{\theta_0}(X_{t_i})) \frac{R(\theta_0, \Delta_{n,i}^{\frac{1}{2} \land (1 - \alpha \beta -\epsilon - \beta)}, X_{t_i})}{a^2(X_{t_i})}\varphi_{\Delta_{n,i}^\beta}(X_{t_{i+1}} - X_{t_i})1_{\left \{|X_{t_i}| \le \Delta_{n,i}^{-k} \right \} }, 
\label{eq: contrast deriv conv} 
\end{align}
where we have used the development \eqref{eq: dl dot mtheta} of $\dot{m}_\theta(X_{t_i})$. \\
We now use Proposition \ref{prop: LT3} on the first term of \eqref{eq: contrast deriv conv}, getting that it converges in distribution to a Gaussian random variable with mean $0$ and variance $\int_\mathbb{R} \frac{4 \dot{b}^2(x, \theta_0)}{a^4(x)}a^2(x) \pi(dx) = \int_\mathbb{R} 4(\frac{\dot{b}(x, \theta_0)}{a(x)})^2 \pi(dx)$, as we wanted.
In order to get the thesis we want to show that the second term of \eqref{eq: contrast deriv conv} goes to zero in probability as $ t_n \rightarrow \infty$. In order to do this, we we want to use Lemma 9 of \hyperref[8 GLM]{(Genon Catalot \& Jacod 1993))} and so we have to prove the following:
\begin{equation}
\frac{2}{\sqrt[]{t_n}}\sum_{i =0}^{n -1}\mathbb{E}[(X_{t_{i+1}} - m_{\theta_0}(X_{t_i})) \frac{R(\theta_0, \Delta_{n,i}^{\frac{1}{2} \land (1 - \alpha \beta -\epsilon - \beta)}, X_{t_i})}{a^2(X_{t_i})}\varphi_{\Delta_{n,i}^\beta}(X_{t_{i+1}} - X_{t_i})1_{\left \{|X_{t_i}| \le \Delta_{n,i}^{-k} \right \} }| \mathcal{F}_{t_i}] \rightarrow 0
\label{eq: critere 1 conv dot un } 
\end{equation}
\begin{equation}
\frac{4}{t_n}\sum_{i =0}^{n -1} \mathbb{E}[ ((X_{t_{i+1}} - m_{\theta_0}(X_{t_i})) \frac{R(\theta_0, \Delta_{n,i}^{\frac{1}{2}\land (1 - \alpha \beta -\epsilon - \beta)}, X_{t_i})}{a^2(X_{t_i})}\varphi_{\Delta_{n,i}^\beta}(X_{t_{i+1}} - X_{t_i})1_{\left \{|X_{t_i}| \le \Delta_{n,i}^{-k} \right \} })^2| \mathcal{F}_{t_i}] \rightarrow 0
\label{eq: critere 2 conv dot un } 
\end{equation}
Using the measurability and the fact that
\begin{equation}
\mathbb{E}[(X_{t_{i+1}} - m_{\theta_0}(X_{t_i}))\varphi_{\Delta_{n,i}^\beta}(X_{t_{i+1}} - X_{t_i})| \mathcal{F}_{t_i}] =0
\label{eq: mtg X-mtheta} 
\end{equation}
we get \eqref{eq: critere 1 conv dot un }.
Let us consider \eqref{eq: critere 2 conv dot un }. Using the Assumption 5 on $a$, the measurability of $R$ and the expression \eqref{eq: x - m prop 2} we can upper bound it with
$$\frac{c}{n \Delta_n} \sum_{i =0}^{n -1}R(\theta_0, \Delta_{n,i}^{1 \land 2(1 - \alpha \beta -\epsilon - \beta)}, X_{t_i})R(\theta_0, \Delta_{n,i}, X_{t_i}) \le \Delta_n^{1 \land 2(1 - \alpha \beta -\epsilon - \beta)} \frac{1}{n } \sum_{i =0}^{n -1} R(\theta_0, 1, X_{t_i}), $$
that goes to zero in norm $1$ by the polynomial growth of $R$, the third point of Lemma \ref{lemma 2.1 GLM} and $A_\beta$. Therefore it converges to zero also in probability. \\
It follows that $$\frac{2}{\sqrt[]{t_n}}\sum_{i =0}^{n -1} (X_{t_{i+1}} - m_{\theta_0}(X_{t_i})) \frac{R(\theta_0, \Delta_{n,i}^{\frac{1}{2} \land (1 - \alpha \beta -\epsilon - \beta)}, X_{t_i})}{a^2(X_{t_i})}\varphi_{\Delta_{n,i}^\beta}(X_{t_{i+1}} - X_{t_i})1_{\left \{|X_{t_i}| \le \Delta_{n,i}^{-k} \right \} } \xrightarrow{\mathbb{P}} 0,$$ as we wanted.
\end{proof}

Concerning the second derivative of the contrast function, we have the following convergence: \\
\begin{lemma}
Suppose that Assumptions 1 - 8 and $A_\beta$ hold. Then
$$\frac{\ddot{U}_n(\theta_0)}{t_n} \xrightarrow{\mathbb{P}} -2\int_\mathbb{R}(\frac{\dot{b}(x, \theta_0)}{a(x)})^2 \, \pi(dx). $$
\label{lemma: convergence ddot Un } 
\end{lemma}

\begin{proof}
Derivating twice the expression of $U_n$ we get
\begin{align}
\ddot{U}_n(\theta_0) = -2 \sum_{i =0}^{n -1} \frac{\dot{m}_{\theta_0}^2(X_{t_i})}{a^2(X_{t_i}) \Delta_{n,i}} \varphi_{\Delta_{n,i}^\beta}(X_{t_{i+1}} - X_{t_i})1_{\left \{|X_{t_i}| \le \Delta_{n,i}^{-k} \right \} } + \nonumber \\
+ 2 \sum_{i =0}^{n -1} \frac{(X_{t_{i+1}} - m_{\theta_0}(X_{t_i})) \ddot{m}_{\theta_0}(X_{t_i})}{a^2(X_{t_i})\Delta_{n,i}} \, \varphi_{\Delta_{n,i}^\beta}(X_{t_{i+1}} - X_{t_i})1_{\left \{|X_{t_i}| \le \Delta_{n,i}^{-k} \right \} }
\label{eq: ddotUn(theta0)} 
\end{align}
First of all we show that the second term of \eqref{eq: ddotUn(theta0)}, divided by $n \Delta_n$, goes to zero in probability. We use again Lemma 9 of \hyperref[8 GLM]{(Genon Catalot \& Jacod 1993)}. Hence, our goal is to prove the following:
\begin{equation}
\frac{2}{t_n} \sum_{i =0}^{n -1} \mathbb{E}[\frac{(X_{t_{i+1}} - m_{\theta_0}(X_{t_i})) \ddot{m}_{\theta_0}(X_{t_i})}{a^2(X_{t_i})\Delta_{n,i}} \, \varphi_{\Delta_{n,i}^\beta}(X_{t_{i+1}} - X_{t_i})1_{\left \{|X_{t_i}| \le \Delta_{n,i}^{-k} \right \} }| \mathcal{F}_{t_i}] \rightarrow 0
\label{eq: convergence E ddotUn} 
\end{equation}
\begin{equation}
\frac{4}{(t_n)^2} \sum_{i =0}^{n -1} \mathbb{E}[ (\frac{(X_{t_{i+1}} - m_{\theta_0}(X_{t_i})) \ddot{m}_{\theta_0}(X_{t_i})}{a^2(X_{t_i})\Delta_{n,i}} \, \varphi_{\Delta_{n,i}^\beta}(X_{t_{i+1}} - X_{t_i})1_{\left \{|X_{t_i}| \le \Delta_{n,i}^{-k} \right \} })^2| \mathcal{F}_{t_i}] \rightarrow 0
\label{eq: convergence squared E ddotUn} 
\end{equation}
As we acted in the last proof, we use \eqref{eq: mtg X-mtheta} in order to get \eqref{eq: convergence E ddotUn}. \\
Concerning \eqref{eq: convergence squared E ddotUn}, using Assumption 5 on $a$, the measurability of $R$, the development \eqref{eq: dl ddot mtheta} of $\ddot{m}_{\theta_0}(X_{t_i})$ and the expression \eqref{eq: x - m prop 2} we can upper bound it with
$$\frac{c}{n^2\Delta^2_n}\sum_{i =0}^{n -1} [R(\theta_0, \Delta_{n,i}, X_{t_i}) \frac{\Delta_{n,i}^2 \ddot{b}^2(X_{t_i}, \theta_0) + R(\theta_0, \Delta_{n,i}^{3 \land 2(2 - \alpha \beta -\epsilon - \beta)}, X_{t_i})}{\Delta^2_{n,i}} ] \le$$
$$\le \frac{c}{n^2\Delta_n} \sum_{i =0}^{n -1} R(\theta_0, 1, X_{t_i}),$$
where in the last inequality we have used the polynomial growth of $\ddot{b}$, the property \eqref{propriety power R} on $R$ and that $|\Delta_{n,i}| \le \Delta_n$ . Since $n \Delta_n \rightarrow \infty$ and 
$\frac{1}{n}\sum_{i =0}^{n -1} R(\theta_0, 1, X_{t_i}) $ is bounded in $L^1$, we get the convergence en probability wanted. \\
Let us now consider the first term of \eqref{eq: ddotUn(theta0)}. Using the development \eqref{eq: dl dot mtheta} we get
\begin{equation}
\frac{-2}{t_n} \sum_{i =0}^{n -1} \frac{(\Delta_{n,i}\dot{b}(X_{t_i}, \theta_0) + R(\theta_0, \Delta_{n,i}^{\frac{3}{2}\land (2 - \beta - \beta \alpha -\epsilon)}, X_{t_i}))^2}{a^2(X_{t_i}) \Delta_{n,i}} \varphi_{\Delta_{n,i}^\beta}(X_{t_{i+1}} - X_{t_i})1_{\left \{|X_{t_i}| \le \Delta_{n,i}^{-k} \right \} }.
\label{eq: ddotUn(theta0) first term} 
\end{equation}
Hence, we obtain three terms by expanding the square. Using on the first Proposition \ref{prop: LT1}, we get the convergence
\begin{equation}
\frac{-2}{t_n} \sum_{i =0}^{n -1} \frac{\Delta_{n,i}\dot{b}^2(X_{t_i}, \theta_0)}{a^2(X_{t_i})}\varphi_{\Delta_{n,i}^\beta}(X_{t_{i+1}} - X_{t_i})1_{\left \{|X_{t_i}| \le \Delta_{n,i}^{-k} \right \} } \xrightarrow{\mathbb{P}} -2 \int_\mathbb{R} \frac{\dot{b}^2(x, \theta_0)}{a^2(x)} \pi(dx).
\label{eq: first term con ddotUn} 
\end{equation}
The second term of \eqref{eq: ddotUn(theta0) first term} is 
$$\frac{-4}{t_n} \sum_{i =0}^{n -1} \frac{2\dot{b}(X_{t_i}, \theta_0)R(\theta_0, \Delta_{n,i}^{\frac{3}{2}\land (2 - \beta - \beta \alpha -\epsilon)}, X_{t_i})}{a^2(X_{t_i})} \varphi_{\Delta_{n,i}^\beta}(X_{t_{i+1}} - X_{t_i})1_{\left \{|X_{t_i}| \le \Delta_{n,i}^{-k} \right \} }.$$
Using Assumption 5 on $a$, the fact that both $\varphi$ and the indicator function are bounded, the polynomial growth of both $\dot{b}$ and $R$  and the third point of Lemma \ref{lemma 2.1 GLM} we get that its $L^1$ norm is upper bounded by $c \Delta_n^{\frac{1}{2} \land (1 - \beta - \beta \alpha -\epsilon)}$. Since the exponent on $\Delta_n$ is positive, the convergence in norm $L^1$ and therefore in probability follows. \\
Concerning the last term of \eqref{eq: ddotUn(theta0) first term}, using again Assumption 5 on $a$, the fact that both $\varphi$ and the indicator function are bounded, the polynomial growth of $R$  and the third point of Lemma \ref{lemma 2.1 GLM} we get that its $L^1$ norm is upper bounded by $c \Delta_n^{1 \land (2 - 2 \beta - 2 \beta \alpha -2\epsilon)}$. Once again, since the exponent on $\Delta_n$ is positive, the convergence in norm $L^1$ and therefore in probability follows. \\
It yields 
$$\frac{\ddot{U}_n(\theta_0)}{t_n} \xrightarrow{\mathbb{P}} -2 \int_\mathbb{R} \frac{\dot{b}^2(x, \theta_0)}{a^2(x)} \pi(dx).$$
\end{proof}

\subsection{Asymptotic normality of the estimator}\label{Ss:Asymptotic_nor}
In order to show the asymptotic normality of the estimator we need the following lemma: \\
\begin{lemma}
Suppose that Assumptions 1 - 8 and $A_\beta$ hold. Then
\begin{equation}
\frac{1}{t_n} \sup_{t \in [0,1]}|\ddot{U}_n(\theta_0 + t (\hat{\theta}_n - \theta_0)) - \ddot{U}_n(\theta_0)|\xrightarrow{\mathbb{P}} 0,
\label{eq: conv unif ddot Un} 
\end{equation}
where $\hat{\theta}_n$ is the estimator defined in \eqref{eq: def thetan}.
\begin{proof}
Let us define
\begin{equation}
\tilde{\theta}_n : = \theta_0 + t (\hat{\theta}_n - \theta_0).
\label{eq: def tilde theta} 
\end{equation}
 Using \eqref{eq: ddotUn(theta0)},
\begin{align}
\frac{\ddot{U}_n(\tilde{\theta}_n) - \ddot{U}_n(\theta_0)}{t_n} = -\frac{2}{t_n} \sum_{i =0}^{n -1} \frac{( \dot{m}_{\tilde{\theta}_n}^2 (X_{t_i}) - \dot{m}_{\theta_0}^2(X_{t_i}))}{a^2(X_{t_i}) \Delta_{n,i}} \varphi_{\Delta_{n,i}^\beta}(X_{t_{i+1}} - X_{t_i})1_{\left \{|X_{t_i}| \le \Delta_{n,i}^{-k} \right \} } + \nonumber \\
+ \frac{2}{t_n} \sum_{i =0}^{n -1} \frac{(X_{t_{i+1}} - m_{\theta_0}(X_{t_i})) (\ddot{m}_{\tilde{\theta}_n}(X_{t_i})-\ddot{m}_{\theta_0}(X_{t_i}))}{a^2(X_{t_i})\Delta_{n,i}} \, \varphi_{\Delta_{n,i}^\beta}(X_{t_{i+1}} - X_{t_i})1_{\left \{|X_{t_i}| \le \Delta_{n,i}^{-k} \right \} } + \nonumber \\
+ \frac{2}{t_n} \sum_{i =0}^{n -1} \frac{(m_{\theta_0}(X_{t_i}) - m_{\tilde{\theta}_n}(X_{t_i}))\ddot{m}_{\tilde{\theta}_n}(X_{t_i})}{a^2(X_{t_i})\Delta_{n,i}} \, \varphi_{\Delta_{n,i}^\beta}(X_{t_{i+1}} - X_{t_i})1_{\left \{|X_{t_i}| \le \Delta_{n,i}^{-k} \right \} }.
\label{eq: difference de la derivé du contraste} 
\end{align}
Concerning the first term of \eqref{eq: difference de la derivé du contraste}, we use the following estimation:
\begin{equation}
|\dot{m}_{\tilde{\theta}_n}^2 (X_{t_i}) - \dot{m}_{\theta_0}^2(X_{t_i})| \le 2 |\ddot{m}_{\theta_u}(X_{t_i})\dot{m}_{\theta_u}(X_{t_i})(\tilde{\theta}_n - \theta_0)|,
\label{eq: estimation difference mtheta} 
\end{equation}
where $\theta_u \in [\theta_0, \tilde{\theta}_n]$. We replace the development \eqref{eq: dl dot mtheta} and \eqref{eq: dl ddot mtheta} of $\dot{m}$ and $\ddot{m}$. Hence the first term of \eqref{eq: difference de la derivé du contraste} is, in module, upper bounded by
\begin{equation}
\frac{2}{n} \sum_{i =0}^{n -1}|2(\dot{b}(X_{t_i}, \theta_u) + R(\theta_u, \Delta_{n,i}^{\frac{1}{2} \land (1 - \beta - \beta \alpha -\epsilon)}, X_{t_i}))(\ddot{b}(X_{t_i}, \theta_u) + R(\theta_u, \Delta_{n,i}^{\frac{1}{2} \land (1 - \beta - \beta \alpha -\epsilon)}, X_{t_i}) )||\tilde{\theta}_n - \theta_0| =
\label{eq: estim 1 term derivé contraste} 
\end{equation}
$$= \frac{1}{n} \sum_{i =0}^{n -1}|R(\theta_u, 1, X_{t_i})||\tilde{\theta}_n - \theta_0| \le \frac{1}{n} \sum_{i =0}^{n -1}c(1 + |X_{t_i}|^c )|\hat{\theta}_n - \theta_0| ,$$
where we have used Assumption 5 on $a$, the boundedness of both $\varphi$ and the indicator function, the property \eqref{propriety power R} on $R$ that $|\Delta_{n,i} | \le \Delta_n$ and the definition \eqref{eq: def tilde theta} of $\tilde{\theta}_n$ joint with the fact that $|t| \le 1$. By the consistency of $\hat{\theta}_n$ that we have already proved, we get that the first term of \eqref{eq: difference de la derivé du contraste} converges to zero in probability uniformly in $t$, since 
the right hand side of \eqref{eq: estim 1 term derivé contraste}
is bounded in $L^1$ by the third point of Lemma \ref{lemma 2.1 GLM} and it does not depend on $t$.\\
On the third term of \eqref{eq: difference de la derivé du contraste} we use again the Assumption 5 on $a$, the fact that both $\varphi$ and the indicator function are bounded, the development \eqref{eq: dl ddot mtheta} of $\ddot{m}_\theta$ and the following estimation:
$|m_{\theta_0}(X_{t_i}) - m_{\tilde{\theta}_n}(X_{t_i})| \le | \dot{m}_{\theta_u}(X_{t_i})||\theta_0 - \tilde{\theta}_n| ,$
on which we can use the development \eqref{eq: dl dot mtheta} of $\dot{m}_\theta$. We can hence upper bound the third term with
\begin{equation}
\frac{2}{n} \sum_{i =0}^{n -1}2|(\dot{b}(X_{t_i}, \theta_u) + R(\theta_u, \Delta_{n,i}^{\frac{1}{2} \land (1 - \beta - \beta \alpha -\epsilon)}, X_{t_i}))(\ddot{b}(X_{t_i},\tilde{\theta}_n) + R(\tilde{\theta}_n, \Delta_{n,i}^{\frac{1}{2} \land (1 - \beta - \beta \alpha-\epsilon)}, X_{t_i}) )||\theta_0 - \tilde{\theta}_n | =
\label{eq: estim 3 term derivé contraste} 
\end{equation}
$$= \frac{1}{n} \sum_{i =0}^{n -1}|R(\theta, 1, X_{t_i})||\tilde{\theta}_n - \theta_0| \le \frac{1}{n} \sum_{i =0}^{n -1}c(1 + |X_{t_i}|^c )|\hat{\theta}_n - \theta_0| .$$
The consistency of $\hat{\theta}_n$ yields the convergence in probability uniformly in $t$ wanted, by the boundedness in $L^1$ of the sum, that does not depend on $t$. \\
It remains to prove the convergence to zero, uniformly in $t$, for the second term of \eqref{eq: difference de la derivé du contraste}; it is sufficient to prove that the following sequence $S_n(\theta)$ converges to zero uniformly with respect to $\theta$:
$$S_n(\theta): = \frac{2}{t_n} \sum_{i =0}^{n -1} \frac{(X_{t_{i+1}} - m_{\theta_0}(X_{t_i})) (\ddot{m}_{\theta}(X_{t_i})-\ddot{m}_{\theta_0}(X_{t_i}))}{a^2(X_{t_i})\Delta_{n,i}} \, \varphi_{\Delta_{n,i}^\beta}(X_{t_{i+1}} - X_{t_i})1_{\left \{|X_{t_i}| \le \Delta_{n,i}^{-k} \right \} }.$$
The pointwise convergence is already proved (it is enough to repeat the proof of \eqref{eq: convergence E ddotUn} and \eqref{eq: convergence squared E ddotUn} with $\ddot{m}_{\theta}(X_{t_i}) - \ddot{m}_{\theta_0}(X_{t_i})$ in place of $\ddot{m}_{\theta_0}(X_{t_i})$). In order to show that the convergence takes place uniformly in $\theta$, we prove the tightness of $S_n(\theta)$, using the criterion analogues to \eqref{eq: 1 criterion tightness} and \eqref{eq: 2 criterion tightness}. \\
Let us consider \eqref{eq: 2 criterion tightness} first. We observe that 
$$\mathbb{E}[(S_n(\theta_1)- S_n(\theta_2))^2] \le$$
\begin{equation}
\le \frac{c}{n^2\Delta^2_n} \mathbb{E}[( \sum_{i =0}^{n -1} \frac{(X_{t_{i+1}} - m_{\theta_0}(X_{t_i})) (\ddot{m}_{\theta_1}(X_{t_i})-\ddot{m}_{\theta_2}(X_{t_i}))}{a^2(X_{t_i})\Delta_{n,i}} \, \varphi_{\Delta_{n,i}^\beta}(X_{t_{i+1}} - X_{t_i})1_{\left \{|X_{t_i}| \le \Delta_{n,i}^{-k} \right \} })^2].
\label{eq: Sntheta_1 - Sntheta_2} 
\end{equation}
By the building the sum is a square integrable martingale. The Pythagoras' theorem on a square integrable martingale yields that \eqref{eq: Sntheta_1 - Sntheta_2} is equal to
\begin{equation}
\frac{c}{n^2\Delta^2_n} \mathbb{E}[\sum_{i =0}^{n -1} \frac{(X_{t_{i+1}} - m_{\theta_0}(X_{t_i}))^2 (\ddot{m}_{\theta_1}(X_{t_i})-\ddot{m}_{\theta_2}(X_{t_i}))^2}{a^4(X_{t_i})\Delta^2_{n,i}} \, \varphi^2_{\Delta_{n,i}^\beta}(X_{t_{i+1}} - X_{t_i})1_{\left \{|X_{t_i}| \le \Delta_{n,i}^{-k} \right \} } ].
\label{eq: equivalence Sntheta_1 - Sntheta_2 } 
\end{equation}
We now use the following estimation:
\begin{equation}
|\ddot{m}_{\theta_1}(X_{t_i})-\ddot{m}_{\theta_2}(X_{t_i})| \le |\dddot{m}_{\theta_u}(X_{t_i})||\theta_1 - \theta_2|.
\label{eq: accroissements ddotm} 
\end{equation}
Replacing \eqref{eq: accroissements ddotm} in \eqref{eq: Sntheta_1 - Sntheta_2} and using \eqref{eq: estimation 3dot m } on $\dddot{m}_{\theta_u}(X_{t_i})$, we can upper bound \eqref{eq: Sntheta_1 - Sntheta_2} with
$$\frac{c}{n^2\Delta^2_n} \mathbb{E}[\sum_{i =0}^{n -1} \frac{(X_{t_{i+1}} - m_{\theta_0}(X_{t_i}))^2 R(\theta_u, \Delta^2_{n,i}, X_{t_i})}{a^4(X_{t_i})\Delta^2_{n,i}} \, \varphi^2_{\Delta_{n,i}^\beta}(X_{t_{i+1}} - X_{t_i})1_{\left \{|X_{t_i}| \le \Delta_{n,i}^{-k} \right \} } ](\theta_1 - \theta_2)^2 \le$$
\begin{equation}
\le \frac{c}{n^2\Delta^2_n} \mathbb{E}[\sum_{i =0}^{n -1} f(X_{t_i}, \theta_u) \mathbb{E}[(X_{t_{i+1}} - m_{\theta_0}(X_{t_i}))^2 \varphi^2_{\Delta_{n,i}^\beta}(X_{t_{i+1}} - X_{t_i})1_{\left \{|X_{t_i}| \le \Delta_{n,i}^{-k} \right \} }| \mathcal{F}_{t_i}]](\theta_1 - \theta_2)^2,
\label{eq: conditional in normality} 
\end{equation}
with $f(X_{t_i}, \theta_u) = \frac{R(\theta_u, 1, X_{t_i})}{a^4(X_{t_i})}$ and where we have used the property \eqref{propriety power R} of the functions $R$ and the definition of conditional expected value. \\
Using \eqref{eq: x - m prop 2}, the property \eqref{propriety power R} and that $|\Delta_{n,i}| \le \Delta_n$ , we can upper bound \eqref{eq: conditional in normality} with \\
$\frac{4}{n^2\Delta_n} \sum_{i =0}^{n -1}  \mathbb{E}[f(X_{t_i}, \theta_u)R(\theta_0, 1, X_{t_i})](\theta_1 - \theta_2)^2.$ \\
By the Assumption 5 on $a$ and the polynomial growth of $R$ derived by its definition, $f$ has polynomial growth. Using the third point of Lemma \ref{lemma 2.1 GLM} we get that the expected value is bounded. Hence, since $n\Delta_n \rightarrow \infty $, it yields
\begin{equation}
\frac{4}{n^2\Delta_n} \sum_{i =0}^{n -1}  \mathbb{E}[f(X_{t_i}, \theta_u)R(\theta_0, 1, X_{t_i})] \le c,
\label{eq: estimation crit 2 tightness end} 
\end{equation}
therefore we obtain \eqref{eq: 2 criterion tightness} on $S_n$. \\
Concerning \eqref{eq: 1 criterion tightness}, we can act exactly in the same way, using \eqref{eq: estimation crit 2 tightness end} and the compactness of $\Theta$. The tightness of $S_n(\theta)$ follows.
\end{proof}
\end{lemma}

We are now ready to prove the asymptotic normality of the estimator. Using \eqref{eq: conv unif ddot Un} we have that
\begin{equation}
\frac{1}{t_n} \int_0^1 [\ddot{U}_n(\theta_0 + t(\hat{\theta}_n - \theta_0)) - \ddot{U}_n(\theta_0)]dt \xrightarrow{\mathbb{P}} 0.
\label{eq: conv int ddotUntheta} 
\end{equation}
We observe that
\begin{align}
\frac{1}{t_n} \int_0^1 [\ddot{U}_n(\theta_0 + t(\hat{\theta}_n - \theta_0))]dt \, \sqrt[]{t_n}(\hat{\theta}_n - \theta_0) = \nonumber \\
= \frac{1}{\sqrt[]{t_n}}\int_0^1 [\ddot{U}_n(\theta_0 + t(\hat{\theta}_n - \theta_0))]dt \, (\hat{\theta}_n - \theta_0) = \frac{1}{\sqrt[]{t_n}}(\dot{U}_n(\hat{\theta}_n) - \dot{U}_n(\theta_0)) = - \frac{\dot{U}_n(\theta_0)}{\sqrt[]{t_n}},
\end{align}
where in the last equality we have used that, on the set $\left \{ \hat{\theta}_n \in \mathop \Theta\limits^ \circ \right \}$, $\dot{U}_n(\hat{\theta}_n) =0 $ since $\hat{\theta}_n$ is a 
minimum. \\
Hence
\begin{equation}
\sqrt[]{t_n}(\hat{\theta}_n - \theta_0) = \frac{- \frac{\dot{U}_n(\theta_0)}{\sqrt[]{t_n}}}{\frac{1}{t_n} \int_0^1 [\ddot{U}_n(\theta_0 + t(\hat{\theta}_n - \theta_0))]dt}.
\label{eq: normality} 
\end{equation}
Using Lemma \ref{lemma dot Un convergence} we have the convergence in distribution of the numerator of \eqref{eq: normality} to $N(0, 4 \int_\mathbb{R} (\frac{\dot{b}(x,\theta_0)}{a(x)})^2\pi(dx))$ and, by the equation \eqref{eq: conv int ddotUntheta}, the denominator converges in probability to $-2\int_\mathbb{R} (\frac{\dot{b}(x,\theta_0)}{a(x)})^2\pi(dx) $. \\
Therefore $\sqrt[]{t_n}(\hat{\theta}_n - \theta_0)$ converges in distribution to $N(0, \frac{4 \int_\mathbb{R} (\frac{\dot{b}(x,\theta_0)}{a(x)})^2\pi(dx)}{4(\int_\mathbb{R} (\frac{\dot{b}(x,\theta_0)}{a(x)})^2 \pi(dx))^2})$, i. e. it is $N(0,( \int_\mathbb{R} (\frac{\dot{b}(x,\theta_0)}{a(x)})^2\pi(dx))^{-1} )$, as we wanted.

\subsection{Proof of Proposition \ref{P:normality_approximate_contrast}}
The proof of the proposition is essentially similar to the the proof of the asymptotic normality of
$\hat{\theta}_n$ given in Sections \ref{Ss:Contrast_conv}--\ref{Ss:Asymptotic_nor} and we skip it. The main difference comes from the fact that Proposition \ref{prop: LT3} holds true with $\widetilde{m}_{\theta_0}(X_{t_i})$ replacing
${m}_{\theta_0}(X_{t_i})$ under the condition that $\sqrt{n} \Delta_n^{\rho-1/2} \to 0$.

\section*{References}\sloppy

\phantomsection
\label{csl:7}A\"it-Sahalia, Y., \& Yu, J. (2006). {Saddlepoint Approximations for Continuous-Time Markov Processes, Journal of Econometrics, 134, 507–551}.

\phantomsection
\label{csl:17}Applebaum, D. L. processes, \& stochastic calculus. Cambridge university press. (2009).

\phantomsection
\label{csl:4}Barndorff-Nielsen, O. E., \& Shephard, N. (2001). {Non-Gaussian Ornstein-Uhlenbeck-Based Models and Some of Their Uses in Financial Economics. J. R. Stat. Soc., Ser. B, Stat. Methodol., 63, 167-241}.

\phantomsection
\label{csl:3}Bates, D. S. (1996). {Jumps and Stochastic Volatility: Exchange Rate Processes Implicit in Deutsche Mark. The Review of Financial Studies, 9(1), 69-107}.

\phantomsection
\label{Derivative EDS} Bichteler, K. (1987). Malliavin calculus for processes with jumps. Stochastics Monographs.

\phantomsection
\label{csl:6}Ditlevsen, S., \& Greenwood, P. (2013). {The Morris–lecar Neuron Model Embeds a Leaky Integrate-And-Fire Model. Journal of Mathematical Biology 67 239-259}.

\phantomsection
\label{csl:5}Eraker, B., Johannes, M., \& N, P. (2003). {The Impact of Jumps in Volatility and Returns. J. Finance, 58(3), 1269}.

\phantomsection
\label{Florens Zmirou} Florens Zmirou, Danielle. "Approximate discrete-time schemes for statistics of diffusion processes." Statistics: A Journal of Theoretical and Applied Statistics 20.4 (1989): 547-557.

\phantomsection
\label{8 GLM} Genon Catalot, V. and Jacod, J. (1993). On the estimation of the diffusion coefficient for multi- dimensional diffusion processes. Annales de l’institut Henri Poincaré (B) Probabilités et Statistiques, 29, 119-151.

\phantomsection
\label{csl:16}Gloter, A., Loukianova, D., \& Mai, H. (2018). {Jump Filtering and Efficient Drift Estimation for Lévy-Driven Sdes. The Annals of Statistics, 46(4), 1445}.

\phantomsection
\label{csl:21}Ibragimov, I. A., \& Has' Minskii, R. Z. (2013). {Statistical Estimation: Asymptotic Theory (vol. 16). Springer Science and Business Media}.

\phantomsection
\label{csl:24}Jacod, J., \& Protter, P. (2011). {Discretization of Processes (vol. 67). Springer Science and Business Media}.

\phantomsection
\label{11 GLM} Jacod, J., and Shiryaev, A. (2013). Limit theorems for stochastic processes (Vol. 288). Springer Science and Business Media.

\phantomsection
\label{csl:9}Jakobsen, N., \& S{\o}rensen, M. (2017). {Estimating Functions for Jump-Diffusions, Preprint}.

\phantomsection
\label{csl:15}Kessler, M. (1997). {Estimation of an Ergodic Diffusion from Discrete Observations. Scandinavian Journal of Statistics, 24(2), 211-229}.

\phantomsection
\label{csl:2} Kou, S.G. (2002). A Jump Diffusion Model for Option Pricing. Management Science, 48, 1086-1101.

\phantomsection
\label{csl:8}Li, C., \& Chen, D. (2016). {Estimating Jump-Diffusions Using Closed Form Likelihood Expansions. Journal of Econometrics, 195, 51–71}.

\phantomsection
\label{18 GLM}Masuda, H. (2007). Ergodicity and exponential beta mixing bounds for multidimensional diffusions with jumps. Stochastic processes and their applications, 117(1), 35-56.

\phantomsection
\label{19 GLM} Masuda, H. (2009). Erratum to: Ergodicity and exponential beta mixing bound for multidimensional diffusions with jumps, Stochastic Process. Appl. 117 (2007) 35–56. Stochastic Processes and their Applications, 119(2), 676-678.

\phantomsection
\label{csl:12}Masuda, H. (2013). {Convergence of Gaussian Quasi-Likelihood Random Fields for Ergodic Lévy Driven Sde Observed At High Frequency. Annals of Stat., 41(3), 1593}.

\phantomsection
\label{csl:1}Merton, R. C. (1976). {Option Pricing When Underlying Stock Returns Are Discontinuous. Journal of Financial Economics, 3, 125-144}.

\phantomsection
\label{csl:20}Nikolskii, S. M. (1977). {Approximation of Functions of Several Variables and Imbedding Theorems (russian). Sec. Ed., Moskva, Nauka 1977 English Translation of the First Ed., Berlin 1975}.

\phantomsection
\label{Protter GLM} Protter, P. E. (2005). Stochastic differential equations. In Stochastic integration and differential equations (pp. 249-361). Springer, Berlin, Heidelberg.

\phantomsection
\label{csl:11} Shimizu, Y. (2006). M Estimation for Discretely Observed Ergodic Diffusion Processes with Infinitely many Jumps. Statistical Inference for Stochastic Processes, 9, 179-225.

\phantomsection
\label{csl:10}Shimizu, Y., \& Yoshida, N. (2006). {Estimation of Parameters for Diffusion Processes with Jumps from Discrete Observations. Statistical Inference for Stochastic Processes, 9(3), 227-277}.

\phantomsection
\label{csl:14}Yoshida, N. (1992). {Estimation for Diffusion Processes from Discrete Observation. Journal of Multivariate Analysis, 41, 220-242}.

\vspace{0.5 cm}
LaMME, UMR CNRS 8071 \\
Université d'Evry Val d'Essonne \\
91037 Évry Cedex \\
France \\
E-mail: chiara.amorino@univ-evry.fr \\
E-mail: arnaud.gloter@univ-evry.fr

\appendix

\section{Appendix}
In this section we will prove the technical lemmas that we have used in order to show the main theorems.

\subsection{Proof of expansions of the derivatives of the function $m_{\theta, h}$}\label{S:Proof_derivatives}
In order to prove the explicit approximation of $\dot{m}_{\theta,h}$ and $\ddot{m}_{\theta,h}$ provided in Proposition \ref{prop: dl dotm ddotm}, the following lemma will be useful.
We point out that $X_t^\theta$ is $X_t^{\theta, x}$ and so the process starts in $0$: $X_0^{\theta,x} = x$. \\

\begin{lemma}
Suppose that Assumptions 1 to 4 and 7 hold. Let us define $\dot{X}^{\theta,x}_t : = \frac{\partial X^{\theta,x}_t}{\partial \theta}$ and $\ddot{X}^{\theta,x}_t : = \frac{\partial^2 X^{\theta,x}_t}{\partial \theta^2}$. Then, for all $p \ge 2$ $\exists c > 0$: $\forall h \le \Delta_n$ $\forall x$,
\begin{equation}
\mathbb{E}[|\frac{\dot{X}^{\theta,x}_h}{h}|^p ] \le c(1 + |x|^c ),
\label{eq: estimation dot x} 
\end{equation}
\begin{equation}
\mathbb{E}[|\frac{\ddot{X}^{\theta,x}_h}{h}|^p ] \le c(1 + |x|^c ). 
\label{eq: estimation ddot x} 
\end{equation}
\label{lemma: estimation dotx et ddotx} 
\end{lemma}
\begin{proof}
The dynamic of the process $X$ is known. The same applies to the processes $\dot{X}^{\theta,x}_t$ and $\ddot{X}^{\theta,x}_t$ (cf. (missing citation), section 5).
\begin{equation}
\dot{X}^{\theta,x}_h = \int_0^h (b'(X^{\theta,x}_s, \theta)\dot{X}^{\theta,x}_s + \dot{b}(X^{\theta,x}_s, \theta)) ds + \int_0^h a'(X^{\theta,x}_s)\dot{X}^{\theta,x}_s dW_s + \int_0^h \int_\mathbb{R} \gamma'(X^{\theta,x}_{s^-})\dot{X}^{\theta,x}_s z \tilde{\mu}(dz, ds)
\label{eq: dynamic dotx} 
\end{equation}
and
\begin{equation}
\ddot{X}^{\theta,x}_h = \int_0^h (b''(X^{\theta,x}_s, \theta)(\dot{X}^{\theta,x}_s)^2 + 2\dot{b}'(X^{\theta,x}_s, \theta)\dot{X}^{\theta,x}_s + b'(X^{\theta,x}_s, \theta)\ddot{X}^{\theta,x}_s + \ddot{b}(X^{\theta,x}_s, \theta)) ds +
\label{eq: dynamic ddotx} 
\end{equation}
$$+ \int_0^h (a''(X^{\theta,x}_s)(\dot{X}^{\theta,x}_s)^2 + a'(X^{\theta,x}_s)\ddot{X}^{\theta,x}_s)dW_s + 
\int_0^h \int_\mathbb{R} (\gamma''(X^{\theta,x}_{s^-})(\dot{X}^{\theta,x}_s)^2 + \gamma'(X^{\theta,x}_{s^-})\ddot{X}^{\theta,x}_s) z \tilde{\mu}(dz, ds).$$
From now on, we will drop the dependence of the starting point in order to make the notation easier. \\
Let us start with the proof of  \eqref{eq: estimation dot x}. We observe that, taking the $L^p$ norm of \eqref{eq: dynamic dotx}, we have the following estimation:
\begin{equation}
\mathbb{E}[|\dot{X}_h^\theta|^p] \le c\mathbb{E}[|\int_0^h (b'(X^\theta_s, \theta)\dot{X}^\theta_s + \dot{b}(X^\theta_s, \theta)) ds |^p] + c \mathbb{E}[| \int_0^h a'(X^\theta_s)\dot{X}^\theta_s dW_s |^p] + c \mathbb{E}[|\int_0^h \int_\mathbb{R} \gamma'(X^\theta_{s^-})\dot{X}^\theta_s z \tilde{\mu}(dz, ds)|^p]. 
\label{eq: expectation dinamyc dotX} 
\end{equation}
Concerning the first term of \eqref{eq: expectation dinamyc dotX}, 
$$\mathbb{E}[|\int_0^h (b'(X^\theta_s, \theta)\dot{X}^\theta_s + \dot{b}(X^\theta_s, \theta)) ds |^p] \le c\mathbb{E}[|\int_0^h b'(X^\theta_s, \theta)\dot{X}^\theta_s ds |^p] + c\mathbb{E}[|\int_0^h\dot{b}(X^\theta_s, \theta) ds |^p].$$
Then, using Jensen inequality on the first, we obtain
$$\mathbb{E}[|\int_0^h b'(X^\theta_s, \theta)\dot{X}^\theta_s ds |^p] = \mathbb{E}[h^p| \frac{1}{h}\int_0^h b'(X^\theta_s, \theta)\dot{X}^\theta_s ds |^p] \le $$
$$\le \mathbb{E}[h^{p-1} \int_0^h |b'(X^\theta_s, \theta)|^p |\dot{X}^\theta_s|^p ds ] = h^{p-1} \int_0^h \mathbb{E}[|b'(X^\theta_s, \theta)|^p |\dot{X}^\theta_s|^p] ds.$$
The derivatives of $b$ with respect to $x$ are supposed bounded, it yields
\begin{equation}
\mathbb{E}[|\int_0^h b'(X^\theta_s, \theta)\dot{X}^\theta_s ds |^p] \le c h^{p-1} \int_0^h \mathbb{E}[|\dot{X}^\theta_s|^p] ds.
\label{eq: estimation b expected norme p dotx } 
\end{equation}
Let us now consider the second term of \eqref{eq: expectation dinamyc dotX}. Using Burkholder-Davis-Gundy and Jensen inequalities we get
$$ \mathbb{E}[| \int_0^h a'(X^\theta_s)\dot{X}^\theta_s dW_s |^p] \le c \mathbb{E}[| \int_0^h (a'(X^\theta_s)\dot{X}^\theta_s)^2 ds |^\frac{p}{2}] = $$
$$= c \mathbb{E}[h^\frac{p}{2}|\frac{1}{h} \int_0^h (a'(X^\theta_s)\dot{X}^\theta_s)^2 ds |^\frac{p}{2}] \le
c h^{\frac{p}{2} - 1} \mathbb{E}[\int_0^h |a'(X^\theta_s)\dot{X}^\theta_s|^p].$$
Therefore
\begin{equation}
\mathbb{E}[| \int_0^h a'(X^\theta_s)\dot{X}^\theta_s dW_s |^p] \le c h^{\frac{p}{2} - 1}\int_0^h \mathbb{E}[|\dot{X}^\theta_s|^p] ds,
\label{eq: estimation a expected norme p dotx } 
\end{equation}
where we have used that the derivatives of $a$ are bounded. \\
The third term of \eqref{eq: expectation dinamyc dotX} can be estimed using Kunita inequality (cf. the Appendix of \hyperref[csl:24]{(Jacod \& Protter, 2011)}):
$$\mathbb{E}[|\int_0^h \int_\mathbb{R} \gamma'(X^\theta_{s^-})\dot{X}^\theta_s z \tilde{\mu}(dz, ds)|^p] \le$$
$$ \le \mathbb{E}[\int_0^h \int_\mathbb{R} |\gamma'(X^\theta_{s})\dot{X}^\theta_s|^p |z|^p \bar{\mu}(dz, ds)] + \mathbb{E}[|\int_0^h \int_\mathbb{R} (\gamma'(X^\theta_{s})\dot{X}^\theta_s)^2 z^2 \bar{\mu}(dz, ds)|^\frac{p}{2}] \le$$
$$\le \int_0^h\mathbb{E}[|\gamma'(X^\theta_{s})|^p|\dot{X}^\theta_s|^p] ( \int_\mathbb{R} |z|^p F(z)dz) ds + 
\mathbb{E}[|\int_0^h (\gamma'(X^\theta_{s})\dot{X}^\theta_s)^2 (\int_\mathbb{R}z^2 F(z)dz) ds|^\frac{p}{2}]\le$$
$$\le c\int_0^h\mathbb{E}[|\gamma'(X^\theta_{s})|^p|\dot{X}^\theta_s|^p] ds + 
c \mathbb{E}[|\int_0^h (\gamma'(X^\theta_{s})\dot{X}^\theta_s)^2 ds|^\frac{p}{2}], $$
where in the last two inequalities we have just used the definition of the compensated measure $\bar{\mu}$ and the third point of Assumption 4. \\
Since the derivatives of $\gamma$ are supposed bounded and by the Jensen inequality we get it is upper bounded by
$$c \int_0^h\mathbb{E}[|\dot{X}^\theta_s|^p] ds + c \mathbb{E}[h^{\frac{p}{2} - 1}\int_0^h|\gamma'(X^\theta_{s})|^p|\dot{X}^\theta_s|^p ds] \le c \int_0^h\mathbb{E}[|\dot{X}^\theta_s|^p] ds + c h^{\frac{p}{2} - 1} \int_0^h\mathbb{E}[|\dot{X}^\theta_s|^p] ds.  $$
Hence
\begin{equation}
\mathbb{E}[|\int_0^h \int_\mathbb{R} \gamma'(X^\theta_{s^-})\dot{X}^\theta_s z \tilde{\mu}(dz, ds)|^p] \le c(1 + h^{\frac{p}{2} - 1}) \int_0^h\mathbb{E}[|\dot{X}^\theta_s|^p] ds.
\label{eq: estimation gamma expected norme p dotx } 
\end{equation}
From \eqref{eq: estimation b expected norme p dotx }, \eqref{eq: estimation a expected norme p dotx } and \eqref{eq: estimation gamma expected norme p dotx }, we obtain
$$\mathbb{E}[|\dot{X}_h^\theta|^p] \le c\mathbb{E}[|\int_0^h\dot{b}(X^\theta_s, \theta) ds |^p] + c(1 + h^{\frac{p}{2} - 1}+ h^{p - 1}) \int_0^h\mathbb{E}[|\dot{X}^\theta_s|^p] ds.  $$
Let $M_h$ be $\mathbb{E}[|\dot{X}_h^\theta|^p]$, then the equation above can be seen as
$$M_h \le c\mathbb{E}[|\int_0^h\dot{b}(X^\theta_s, \theta) ds |^p] + c(1 + h^{\frac{p}{2} - 1}+ h^{p - 1}) \int_0^h M_s ds.$$
Using Gronwall lemma, it yields $M_h \le c\mathbb{E}[|\int_0^h\dot{b}(X^\theta_s, \theta) ds |^p] e^{ch(1 + h^{\frac{p}{2} - 1}+ h^{p - 1})}$. \\
By the polynomial growth of $\dot{b}$ and the third point of Lemma 1, $$\mathbb{E}[|\int_0^h\dot{b}(X^\theta_s, \theta) ds |^p] \le ch^p(1 + |X_0^{0,x}|^c)= ch^p(1 + |x|^c). $$ \\
Hence $\mathbb{E}[|\dot{X}_h^\theta|^p] \le c h^p(1 + |x|^c).$
\\
Our goal is now to prove \eqref{eq: estimation ddot x}. In order to do it, we take the $L^p$ norm of \eqref{eq: dynamic ddotx}, getting the following estimation:
\begin{align}
\mathbb{E}[|\ddot{X}^\theta_h|^p] \le \mathbb{E}[|\int_0^h (b''(X^\theta_s, \theta)(\dot{X}^\theta_s)^2 + 2\dot{b}'(X^\theta_s, \theta)\dot{X}^\theta_s + b'(X^\theta_s, \theta)\ddot{X}^\theta_s + \ddot{b}(X^\theta_s, \theta)) ds|^p]+ \nonumber \\
+ \mathbb{E}[|\int_0^h (a''(X^\theta_s)(\dot{X}^\theta_s)^2 + a'(X^\theta_s)\ddot{X}^\theta_s)dW_s |^p] + \mathbb{E}[|\int_0^h \int_\mathbb{R} (\gamma''(X^\theta_{s^-})(\dot{X}^\theta_s)^2 + \gamma'(X^\theta_{s^-})\ddot{X}^\theta_s) z \tilde{\mu}(dz, ds)|^p] 
\label{eq: expectation dinamyc ddotX} 
\end{align}
The first term of \eqref{eq: expectation dinamyc ddotX} is upper bounded by
$$\mathbb{E}[|\int_0^h (b''(X^\theta_s, \theta)(\dot{X}^\theta_s)^2ds|^p] + \mathbb{E}[|\int_0^h 2\dot{b}'(X^\theta_s, \theta)\dot{X}^\theta_s ds|^p] + \mathbb{E}[ |\int_0^h b'(X^\theta_s, \theta)\ddot{X}^\theta_s ds|^p] + \mathbb{E}[| \int_0^h \ddot{b}(X^\theta_s, \theta) ds|^p]\le$$
\begin{equation}
\le ch^{p - 1} \int_0^h \mathbb{E}[|\dot{X}^\theta_s|^{2p}] ds + c h^{p - 1}\int_0^h \mathbb{E}[|\dot{b}'(X^\theta_s, \theta)|^p|\dot{X}^\theta_s|^p] ds +  ch^{p - 1} \int_0^h \mathbb{E}[|\ddot{X}^\theta_s|^p] ds + \mathbb{E}[| \int_0^h \ddot{b}(X^\theta_s, \theta) ds|^p],
\label{eq: bddot} 
\end{equation}
where we have used Jensen inequality and that the derivatives of $b$ with respect to $x$ are supposed bounded. \\
By Holder inequality
$$\mathbb{E}[|\dot{b}'(X^\theta_s, \theta)|^p|\dot{X}^\theta_s|^p] \le (\mathbb{E}[|\dot{b}'(X^\theta_s, \theta)|^{pp_1}])^\frac{1}{p_1} (\mathbb{E}[|\dot{X}^\theta_s|^{pp_2}]^\frac{1}{p_2} \le c(h^{pp_2})^\frac{1}{p_2}(1 + |x|^c) = c h^p(1 + |x|^c),$$
where in the last inequality we have used the boundedness of $\dot{b}'$ and \eqref{eq: estimation dot x}.
Since $\ddot{b}$ has polynomial growth and by the third point of Lemma \ref{lemma: Moment inequalities}, $\mathbb{E}[| \int_0^h \ddot{b}(X^\theta_s, \theta)) ds|^p] \le ch^p(1 + |x|^c)$. 
Replacing in \eqref{eq: bddot} and using also on its first term \eqref{eq: estimation dot x} we obtain it is upper bounded by
\begin{equation}
\mathbb{E}[|\int_0^h (b''(X^\theta_s, \theta)(\dot{X}^\theta_s)^2 + 2\dot{b}'(X^\theta_s, \theta)\dot{X}^\theta_s + b'(X^\theta_s, \theta)\ddot{X}^\theta_s + \ddot{b}(X^\theta_s, \theta)) ds|^p] \le
\label{eq: estimation b expected norme p ddotx } 
\end{equation}
$$ \le c(1 + |x|^c)(h^{3p} + h^{2p}+ h^p )+  ch^{p - 1} \int_0^h \mathbb{E}[|\ddot{X}^\theta_s|^p] ds . $$
Let us now consider the second term of \eqref{eq: expectation dinamyc ddotX}. By Burkholder-Davis-Gundy and Jensen inequalities we get
$$\mathbb{E}[|\int_0^h (a''(X^\theta_s)(\dot{X}^\theta_s)^2 + a'(X^\theta_s)\ddot{X}^\theta_s)dW_s |^p]  \le
\mathbb{E}[|\int_0^h (a''(X^\theta_s)(\dot{X}^\theta_s)^2 + a'(X^\theta_s)\ddot{X}^\theta_s)^2ds |^\frac{p}{2}]\le $$
\begin{equation}
\le h^{\frac{p}{2} -1}\mathbb{E}[\int_0^h |a''(X^\theta_s)(\dot{X}^\theta_s)^2|^p + |a'(X^\theta_s)\ddot{X}^\theta_s|^pds] \le c h^{\frac{p}{2} + 2p}(1 + |x|^c) + ch^{\frac{p}{2} -1} \int_0^h \mathbb{E}[|\ddot{X}^\theta_s|^p] ds,
\label{eq: estimation a expected norme p ddotx } 
\end{equation}
where in the last inequality we have used that the derivatives of $a$ are supposed bounded and \eqref{eq: estimation dot x}. \\
Concerning the last term of \eqref{eq: expectation dinamyc ddotX}, by Kunita inequality it is upper bounded by
$$\mathbb{E}[\int_0^h \int_\mathbb{R} |\gamma''(X^\theta_{s})(\dot{X}^\theta_s)^2 + \gamma'(X^\theta_{s})\ddot{X}^\theta_s|^p |z|^p \bar{\mu}(dz, ds)] +\mathbb{E}[|\int_0^h \int_\mathbb{R} (\gamma''(X^\theta_{s})(\dot{X}^\theta_s)^2 + \gamma'(X^\theta_{s})\ddot{X}^\theta_s)^2 z^2 \bar{\mu}(dz, ds)|^\frac{p}{2}] \le$$
$$\le c \int_0^h \mathbb{E}[|\gamma''(X^\theta_{s})(\dot{X}^\theta_s)^2 + \gamma'(X^\theta_{s})\ddot{X}^\theta_s|^pds] + h^{\frac{p}{2} - 1}\int_0^h \mathbb{E}[|\gamma''(X^\theta_{s})(\dot{X}^\theta_s)^2 + \gamma'(X^\theta_{s})\ddot{X}^\theta_s|^pds], $$
having used Jensen inequality and the third point of Assumption 4 in order to say that $\int_\mathbb{R}|z|^p F(z) dz < c$. \\
Using \eqref{eq: estimation dot x} and the boundedness of the derivatives of $\gamma$, it is upper bounded by
\begin{equation}
c(1 + |x|^c)h^{2p +1} + c\int_0^h \mathbb{E}[|\ddot{X}^\theta_s|^p]ds + ch^{\frac{p}{2} + 2p}(1 + |x|^c) + ch^{\frac{p}{2} - 1}\int_0^h \mathbb{E}[|\ddot{X}^\theta_s|^p]ds.
\label{eq: estimation gamma expected norme p ddotx } 
\end{equation}
From \eqref{eq: estimation b expected norme p ddotx }, \eqref{eq: estimation a expected norme p ddotx } and \eqref{eq: estimation gamma expected norme p ddotx } we get
$$\mathbb{E}[|\ddot{X}^\theta_h|^p] \le c(1 + |x|^c)h^p(1+h^p + h^{2p} + h^{p + \frac{p}{2}} + h^{p +1}) + c(1+ h^{p-1}+h^{\frac{p}{2} -1})\int_0^h \mathbb{E}[|\ddot{X}^\theta_s|^p]ds.$$
Using Gronwall Lemma we obtain $\mathbb{E}[|\ddot{X}^\theta_h|^p] \le c(1 + |x|^c)h^p(1+h^p + h^{2p} + h^{p + \frac{p}{2}} + h^{p +1})$ and so $ \mathbb{E}[|\ddot{X}^\theta_h|^p] \le c(1 + |x|^c)h^p, $ as we wanted.
\end{proof}

\begin{remark}
Supposing that the same assumptions as in Lemma $5$ hold and acting as we have done in order to get the estimations \eqref{eq: estimation dot x} and \eqref{eq: estimation ddot x} it is possible to prove that, for all $p \ge 2$ $\exists c > 0$:  $\forall h \le \Delta_{n} $ , $\forall x$,
\begin{equation}
\mathbb{E}[|\frac{\partial^3}{\partial \theta^3}X_h^{\theta,x}|^p \frac{1}{h^p}] \le c(1 + |x|^c).
\label{eq: estimation 3dot x} 
\end{equation}
\end{remark}

\subsubsection{Proof of Proposition \ref{prop: dl dotm ddotm}}
\begin{proof}
As in the proof of Lemma \ref{lemma: estimation dotx et ddotx}, we drop the dependence on the starting point in order to make the notation easier. \\
We recall the definition of $m_{\theta,h}(x):$
$$m_{\theta, h}(x): =\frac{\mathbb{E}[X_h^\theta \varphi_{h^\beta}(X_{h}^\theta - X_0^\theta)|X_0^\theta = x]}{\mathbb{E}[\varphi_{h^\beta}(X_h^\theta - X_0^\theta)|X_0^\theta = x]} = \frac{\mathbb{E}[X_h^\theta \varphi_{h^\beta}(X_{h}^\theta - x)]}{\mathbb{E}[\varphi_{h^\beta}(X_h^\theta -x)]}.$$
Its derivative with respect to $\theta$ is
\begin{equation}
\frac{\mathbb{E}[\dot{X}_h^\theta \varphi_{h^\beta}(X_{h}^\theta - x)] + \mathbb{E}[X_h^\theta h^{- \beta} \dot{X}_{h}^\theta \varphi'_{h^\beta}(X_{h}^\theta - x)]}{\mathbb{E}[\varphi_{h^\beta}(X_h^\theta - x)]} - m_{\theta,h}(x)\frac{\mathbb{E}[h^{- \beta} \dot{X}_{h}^\theta \varphi'_{h^\beta}(X_{h}^\theta - x)]}{\mathbb{E}[\varphi_{h^\beta}(X_h^\theta - x)]}.
\label{eq: derivé mtheta} 
\end{equation}
On the second and on the third term of \eqref{eq: derivé mtheta} we divide and we multiply by $h$ and then we use Proposition \ref{prop: truc moche z}, taking $Z_1=\frac{\dot{X}_{h}^\theta}{h} X_{h}^\theta $ and $Z_2=\frac{\dot{X}_h^\theta}{h}$, respectively. 
We are allowed to do that because they are both bounded in $L^p$, with $p$ arbitrary high,  since we can use \eqref{eq: estimation dot x} on $Z_2$ and  Holder inequality, \eqref{eq: estimation dot x} and the third point of Lemma \ref{lemma: Moment inequalities} on $Z_1$. For $|x| \le h^{- k_0}$ we have
\begin{equation}
m_{\theta,h}(x)= x + \frac{\mathbb{E}[(X_h^\theta - x) \varphi_{h^\beta}(X_{h}^\theta - x)]}{\mathbb{E}[\varphi_{h^\beta}(X_h^\theta -x)]} = R(\theta, 1, x),
\label{eq: m come R(1) senza dl} 
\end{equation}
where we have used that $k_0$ turns out in the proof of theorems \ref{th: dL mtheta alpha <1} and \ref{th: dL mtheta alpha  > 1}, hence it has been chosen such that, for $|x| \le h^{-k_0}$ we have that $\mathbb{E}[\varphi_{h^\beta}(X_h^\theta - x)] \ge \frac{1}{2}$. Moreover the expected value is bounded as a result of the boundedness of $\varphi$ and the third point of Lemma \ref{lemma: Moment inequalities}.
It yields, for $\epsilon > 0$ arbitrary small,
\begin{equation}
\dot{m}_{\theta,h }= \frac{\mathbb{E}[\dot{X}_h^\theta \varphi_{h^\beta}(X_h^\theta - x)] + R(\theta, h^{2 - \alpha\beta -\epsilon - \beta}, x)}{\mathbb{E}[\varphi_{h^\beta}(X_h^\theta - x)]}.
\label{eq: dot mtheta with rest} 
\end{equation}
Let us now consider the first term. Replacing the dynamic of the process $\dot{X}_h^\theta$, we get
$$\mathbb{E}[\int_0^h (b'(X^\theta_s, \theta)\dot{X}^\theta_s + \dot{b}(X^\theta_s, \theta)) ds \, \varphi_{h^\beta}(X_h^\theta - x)] + \mathbb{E}[(\int_0^h a'(X^\theta_s)\dot{X}^\theta_s dW_s +$$
\begin{equation}
+ \int_0^h \int_\mathbb{R} \gamma'(X^\theta_{s^-})\dot{X}^\theta_s z \tilde{\mu}(dz, ds)) \, \varphi_{h^\beta}(X_h^\theta - x)] = \mathbb{E}[\int_0^h \dot{b}(X^\theta_s, \theta) ds \, \varphi_{h^\beta}(X_h^\theta - x)] + R(\theta, h^{\frac{3}{2}}, x).
\label{eq: dotx with term rest} 
\end{equation}
In fact, using Holder inequality, 
$$|\mathbb{E}[\int_0^h b'(X^\theta_s, \theta)\dot{X}^\theta_s ds \, \varphi_{h^\beta}(X_h^\theta - x)] |\le$$
$$ \le (\mathbb{E}[|\int_0^h b'(X^\theta_s, \theta)\dot{X}^\theta_s ds|^p])^\frac{1}{p} (\mathbb{E}[\varphi^q_{h^\beta}(X_h^\theta - x)])^\frac{1}{q} \le (c h^{p-1} \int_0^h \mathbb{E}[|\dot{X}^\theta_s|^p] ds)^\frac{1}{p},$$
where in the last inequality we have used that $\varphi$ is bounded and \eqref{eq: estimation b expected norme p dotx }. By \eqref{eq: estimation dot x}, it is upper bounded by $(c h^{2p}(1 + |x|^c))^\frac{1}{p}$. It turns
\begin{equation}
\mathbb{E}[\int_0^h b'(X^\theta_s, \theta)\dot{X}^\theta_s ds \, \varphi_{h^\beta}(X_h^\theta - x)]= R(\theta, h^2, x).
\label{eq: estim b' dotx varphi} 
\end{equation}
In the same way, from Holder inequality, \eqref{eq: estimation a expected norme p dotx } and the fact that $\varphi$ is bounded, we get
$|\mathbb{E}[\int_0^h a'(X^\theta_s)\dot{X}^\theta_s dW_s \, \varphi_{h^\beta}(X_h^\theta - x)] |\le (c h^{\frac{p}{2} -1} \int_{0}^h \mathbb{E}[|\dot{X}^\theta_s|^p] ds)^\frac{1}{p}.$
Using \eqref{eq: estimation dot x}, it yields
\begin{equation}
 \mathbb{E}[\int_0^h a'(X^\theta_s)\dot{X}^\theta_s dW_s \, \varphi_{h^\beta}(X_h^\theta - x)] = R(\theta, h^\frac{3}{2}, x).
\label{eq: estim a' dotx varphi} 
\end{equation}
Using again Holder inequality, the fact that $\varphi$ is bounded and \eqref{eq: estimation gamma expected norme p dotx } we obtain
$$|\mathbb{E}[\int_0^h \int_\mathbb{R} \gamma'(X^\theta_{s^-})\dot{X}^\theta_s z \tilde{\mu}(dz, ds) \, \varphi_{h^\beta}(X_h^\theta - x)] | \le (c(1 + h^{\frac{p}{2} -1}) \int_0^h \mathbb{E}[|\dot{X}^\theta_s|^p] ds)^\frac{1}{p}.$$
Using \eqref{eq: estimation dot x}, we obtain 
$\mathbb{E}[\int_0^h \int_\mathbb{R} \gamma'(X^\theta_{s^-})\dot{X}^\theta_s z \tilde{\mu}(dz, ds) \, \varphi_{h^\beta}(X_h^\theta -x)] = R(\theta, h^{1 + \frac{1}{p}}, x),$
where $p$ turns out from Holder inequality. We can choose $p=2$, getting
\begin{equation}
\mathbb{E}[\int_0^h \int_\mathbb{R} \gamma'(X^\theta_{s^-})\dot{X}^\theta_s z \tilde{\mu}(dz, ds) \, \varphi_{h^\beta}(X_h^\theta - x)] = R(\theta, h^\frac{3}{2}, x).
\label{eq: estim gamma' dotx varphi} 
\end{equation}
Using \eqref{eq: estim b' dotx varphi}, \eqref{eq: estim a' dotx varphi} and \eqref{eq: estim gamma' dotx varphi}
we have \eqref{eq: dotx with term rest}, as we wanted. \\
The first term of \eqref{eq: dotx with term rest} can be seen as 
$$\mathbb{E}[\int_0^h (\dot{b}(X^\theta_s, \theta)- \dot{b}(x, \theta)) ds \, \varphi_{h^\beta}(X_h^\theta - x)] + \mathbb{E}[\int_0^h \dot{b}(x, \theta) ds \, \varphi_{h^\beta}(X_h^\theta - x)].$$
Using Holder inequality and the fact that $\varphi$ is bounded we get
$$\mathbb{E}[\int_0^h (\dot{b}(X^\theta_s, \theta)- \dot{b}(x, \theta)) ds \, \varphi_{h^\beta}(X_h^\theta - x)] \le $$
$$ \le c (\mathbb{E}[(\int_0^h (\dot{b}(X^\theta_s, \theta)- \dot{b}(x, \theta)) ds )^p])^\frac{1}{p} \le c (\mathbb{E}[(\int_0^h \left \| \frac{\partial \dot{b}}{\partial x }\right \|_\infty |X_s^\theta - x| ds )^p])^\frac{1}{p}.$$
From Jensen inequality we get it is upper bounded by
$c(h^{p-1}\int_0^h \mathbb{E}[|X_s^\theta - x|^p]ds )^\frac{1}{p} \le c (h^{p+1}(1 + |x|^p))^\frac{1}{p}, $
where we have used the second point of Lemma \ref{lemma: Moment inequalities}. It yields
$$\mathbb{E}[\int_0^h (\dot{b}(X^\theta_s, \theta)- \dot{b}(x, \theta)) ds \, \varphi_{h^\beta}(X_h^\theta - x)] = R(\theta, h^{1 + \frac{1}{p}}, x). $$
Taking $p=2$, the equation \eqref{eq: dotx with term rest} becomes
\begin{equation}
\mathbb{E}[\int_0^h \dot{b}(x, \theta) ds \, \varphi_{h^\beta}(X_h^\theta - x)] + R(\theta, h^{\frac{3}{2}}, x) + R(\theta,h^{\frac{3}{2}}, x) = \mathbb{E}[h \dot{b}(x, \theta) \varphi_{h^\beta}(X_h^\theta - x)] + R(\theta, h^{\frac{3}{2}}, x). 
\label{eq: expected ddotx varphi} 
\end{equation}
Replacing in \eqref{eq: dot mtheta with rest}, we get
$$\dot{m}_{\theta,h}(x)= \frac{\mathbb{E}[h \dot{b}(x, \theta) \varphi_{h^\beta}(X_h^\theta - x)] + R(\theta, h^{\frac{3}{2}}, x) + R(\theta, h^{2 - \alpha\beta -\epsilon - \beta}, x)}{\mathbb{E}[\varphi_{h^\beta}(X_h^\theta - x)]} = h \dot{b}(x, \theta) + \frac{R(\theta, h^{\frac{3}{2} \land (2 - \alpha\beta -\epsilon - \beta)}, x)}{\mathbb{E}[\varphi_{h^\beta}(X_h^\theta - x)]}.$$
We use the developments \eqref{eq: dl phi} and \eqref{eq: dl phi alpha >1} on the denominator; in both of them the function $R$ is negligible compared to $1$ without any condition on $\alpha$ and $\beta$. Hence for $|x| \le h^{-k_0}$ we get the expression (\ref{eq: dl dot mtheta}). \\
\\
In order to prove \eqref{eq: dl ddot mtheta}, we have to compute the second derivative of $m_{\theta,h}(x)$. From now on we will write only $\varphi^{(k)}$ for $\varphi^{(k)}_{h^\beta}(X_h^\theta - x)$, $k \ge 0$.
$$\ddot{m}_{\theta,h}(x) = \frac{\mathbb{E}[\ddot{X}^\theta_h\varphi] + h^{- \beta} \mathbb{E}[(\dot{X}^\theta_h)^2 \varphi']}{\mathbb{E}[\varphi]} - \frac{h^{-\beta}\mathbb{E}[\dot{X}^\theta_h\varphi]\mathbb{E}[\dot{X}^\theta_h\varphi']}{(\mathbb{E}[\varphi])^2} + h^{-\beta}\frac{\mathbb{E}[(\dot{X}^\theta_h)^2 \varphi'] +\mathbb{E}[(\dot{X}^\theta_h)^2 \varphi'' X^\theta_h h^{-\beta} ] + \mathbb{E}[X^\theta_h\varphi' \ddot{X}^\theta_h]}{\mathbb{E}{\varphi}} +$$
$$+ h^{-2\beta}\mathbb{E}[\dot{X}^\theta_h \varphi']\frac{m_{\theta,h}(x)\mathbb{E}[\dot{X}^\theta_h \varphi'] - \mathbb{E}[X^\theta_h \dot{X}^\theta_h \varphi']}{(\mathbb{E}[\varphi])^2} - h^{-\beta} \frac{\dot{m}_{\theta,h}(x)\mathbb{E}[\dot{X}^\theta_h \varphi'] + h^{-\beta} m_{\theta,h}(x) \mathbb{E}[(\dot{X}^\theta_h)^2 \varphi''] +  m_{\theta,h}(x)\mathbb{E}[\ddot{X}^\theta_h \varphi'] }{\mathbb{E}[\varphi]}.$$
As for the study of \eqref{eq: derivé mtheta}, we want to rely on Proposition $6$ to treat each term of the form $\mathbb{E}[Z\varphi^{(k)}]$, with $k \ge 1$, where $Z$ is bounded in $L^p$ and use $|\mathbb{E}[Z\varphi^{(k)}]| \le \mathbb{E}[|Z||\varphi^{(k)}|] = R(\theta, h^{1 - \alpha \beta - \epsilon}, x)$. \\
We take successively the following variables as choice for $Z$: $(\frac{\dot{X}^\theta_h}{h})^2$, $\frac{\dot{X}^\theta_h}{h}$, $(\frac{\dot{X}^\theta_h}{h})^2$, $(\frac{\dot{X}^\theta_h}{h})^2 X^\theta_h$, $\frac{\ddot{X}^\theta_h}{h} X^\theta_h$, $\frac{\dot{X}^\theta_h}{h}$, $\frac{\dot{X}^\theta_h}{h} X^\theta_h$, $\frac{\dot{X}^\theta_h}{h}$, $(\frac{\dot{X}^\theta_h}{h})^2$, $\frac{\ddot{X}^\theta_h}{h}$. \\
All those variable $Z$ are bounded in $L^p$ for $p \ge 2$ by \eqref{eq: estimation dot x} - \eqref{eq: estimation ddot x}, the third point of Lemma 1 and Holder inequality. We deduce
\begin{equation}
\ddot{m}_{\theta,h}(x) = \frac{\mathbb{E}[\ddot{X}^\theta_h\varphi] + R(\theta, h^{3 - \alpha \beta -\epsilon - \beta},x)}{\mathbb{E}[\varphi]} - \frac{\mathbb{E}[\dot{X}^\theta_h\varphi]R(\theta, h^{2 - \alpha \beta -\epsilon - \beta},x)}{(\mathbb{E}[\varphi])^2} +
\label{eq: dot m apres prop4} 
\end{equation}
$$+ \frac{R(\theta, h^{3 - \alpha \beta -\epsilon - \beta},x) + R(\theta, h^{3 - \alpha \beta -\epsilon - 2\beta},x) + R(\theta, h^{2 - \alpha \beta -\epsilon - \beta},x)}{\mathbb{E}[\varphi]} + \frac{R(\theta, h^{4 - 2\alpha \beta -\epsilon - 2\beta},x)}{(\mathbb{E}[\varphi])^2} + $$
$$ - \frac{R(\theta, h^{2 - \alpha \beta -\epsilon - \beta},x)\dot{m}_{\theta,h}(x) + R(\theta, h^{3 - \alpha \beta -\epsilon - 2\beta},x) + R(\theta, h^{2 - \alpha \beta -\epsilon - \beta},x)}{\mathbb{E}[\varphi]}.$$
We are no longer considering $m_{\theta,h}(x)$ because, by the expression \eqref{eq: m come R(1) senza dl}, we can include it in the function $R$. \\
Using \eqref{eq: dotx with term rest} and \eqref{eq: expected ddotx varphi}, 
$\mathbb{E}[\dot{X}^\theta_h\varphi]= h\dot{b}(x, \theta)\mathbb{E}[\varphi] + R(\theta, h^{\frac{3}{2}}, x )$. \\
Hence $\mathbb{E}[\dot{X}^\theta_h\varphi]R(\theta, h^{2 - \alpha \beta -\epsilon - \beta},x) = R(\theta, h^{3 - \alpha \beta -\epsilon - \beta},x),$ by the definition of rest function $R$. \\
We have already proved \eqref{eq: dl dot mtheta}, so
$$R(\theta, h^{2 - \alpha \beta -\epsilon - \beta},x)\dot{m}_{\theta,h}(x) = R(\theta, h^{3 - \alpha \beta -\epsilon - \beta},x).$$
Let us now consider $\mathbb{E}[\ddot{X}^\theta_h\varphi]$. Replacing the dynamic of $\ddot{X}^\theta_h$ by \eqref{eq: dynamic ddotx}, it is
$$ \mathbb{E}[\varphi\int_0^h \ddot{b}(X^\theta_s, \theta) ds] + \mathbb{E}[\varphi\int_0^h (b''(X^\theta_s, \theta)(\dot{X}^\theta_s)^2 + 2\dot{b}'(X^\theta_s, \theta)\dot{X}^\theta_s + b'(X^\theta_s, \theta)\ddot{X}^\theta_s) ds] + $$
$$+ \mathbb{E}[\varphi\int_0^h (a''(X^\theta_s)(\dot{X}^\theta_s)^2 + a'(X^\theta_s)\ddot{X}^\theta_s)dW_s] + \mathbb{E}[\varphi\int_0^h \int_\mathbb{R} (\gamma''(X^\theta_{s^-})(\dot{X}^\theta_s)^2 + \gamma'(X^\theta_{s^-})\ddot{X}^\theta_s) z \tilde{\mu}(dz, ds)]=$$
\begin{equation}
= \mathbb{E}[\varphi\int_0^h \ddot{b}(X^\theta_s, \theta) ds] +  R(\theta, h^{\frac{3}{2}}, x).
\label{expected ddotb varphi} 
\end{equation}
Indeed, using Holder inequality,
$$|\mathbb{E}[\varphi\int_0^h (b''(X^\theta_s, \theta)(\dot{X}^\theta_s)^2 + 2\dot{b}'(X^\theta_s, \theta)\dot{X}^\theta_s + b'(X^\theta_s, \theta)\ddot{X}^\theta_s) ds]| \le$$
$$\le (\mathbb{E}[\varphi^q])^\frac{1}{q} (\mathbb{E}[(\int_0^h (b''(X^\theta_s, \theta)(\dot{X}^\theta_s)^2 + 2\dot{b}'(X^\theta_s, \theta)\dot{X}^\theta_s + b'(X^\theta_s, \theta)\ddot{X}^\theta_s) ds)^p])^\frac{1}{p} \le $$
$$\le (c(1 + |x|^c ) h^{3p} + c(1 + |x|^c ) h^{2p} + ch^{p-1}\int_0^h \mathbb{E}[|\ddot{X}_s|^p] ds)^\frac{1}{p},$$
where in the last inequality we have used that $\varphi$ is bounded and we acted as in \eqref{eq: estimation b expected norme p ddotx }. By \eqref{eq: estimation ddot x}, it is upper bounded by $(c h^{3p} + c h^{2p})^\frac{1}{p}(1 + |x|^c )$. It turns
\begin{equation}
\mathbb{E}[\varphi\int_0^h (b''(X^\theta_s, \theta)(\dot{X}^\theta_s)^2 + 2\dot{b}'(X^\theta_s, \theta)\dot{X}^\theta_s + b'(X^\theta_s, \theta)\ddot{X}^\theta_s) ds] =  R(\theta, h^2, x).
\label{estim b" ddotx} 
\end{equation}
In the same way, from Holder inequality, \eqref{eq: estimation a expected norme p ddotx } and the fact that $\varphi$ is bounded we get
$$ |\mathbb{E}[\varphi\int_0^h (a''(X^\theta_s)(\dot{X}^\theta_s)^2 + a'(X^\theta_s)\ddot{X}^\theta_s)dW_s] |\le (c(1 + |x|^c )h^{2p + \frac{p}{2}}+ ch^{\frac{p}{2} - 1}\int_0^h \mathbb{E}[|\ddot{X}_s|^p] ds)^\frac{1}{p}.$$ 
Using \eqref{eq: estimation ddot x}, it is upper bounded by $\le (ch^{2p + \frac{p}{2}} + ch^{p + \frac{p}{2}})^\frac{1}{p}(1 + |x|^c )$ and so we obtain
\begin{equation}
|\mathbb{E}[\varphi\int_0^h (a''(X^\theta_s)(\dot{X}^\theta_s)^2 + a'(X^\theta_s)\ddot{X}^\theta_s)dW_s] | = R(\theta, h^{\frac{3}{2}}, x).
\label{eq: estim a" ddotx} 
\end{equation}
Using again Holder inequality, \eqref{eq: estimation gamma expected norme p ddotx }, the fact that $\varphi$ is bounded and \eqref{eq: estimation ddot x}, we have
$$|\mathbb{E}[\varphi\int_0^h \int_\mathbb{R} (\gamma''(X^\theta_{s^-})(\dot{X}^\theta_s)^2 + \gamma'(X^\theta_{s^-})\ddot{X}^\theta_s) z \tilde{\mu}(dz, ds)] | \le c(h^{2p + 1} + h^{p + 1}+ h^{2p + \frac{p}{2}} +  h^{p + \frac{p}{2}} )^\frac{1}{p}(1 + |x|^c ).$$
Hence, since $p \ge 2$, $$|\mathbb{E}[\varphi\int_0^h \int_\mathbb{R} (\gamma''(X^\theta_{s^-})(\dot{X}^\theta_s)^2 + \gamma'(X^\theta_{s^-})\ddot{X}^\theta_s) z \tilde{\mu}(dz, ds)]| = R(\theta, h^{1 +\frac{1}{p}}, x).$$
Since $p$ turns out from Holder inequality and on which we have only the constraint $p \ge 2$, we can choose $p= 2$, getting
\begin{equation}
|\mathbb{E}[\varphi\int_0^h \int_\mathbb{R} (\gamma''(X^\theta_{s^-})(\dot{X}^\theta_s)^2 + \gamma'(X^\theta_{s^-})\ddot{X}^\theta_s) z \tilde{\mu}(dz, ds)]| = R(\theta, h^{\frac{3}{2}}, x).
\label{eq: estim gamma" ddotx } 
\end{equation}
From \eqref{estim b" ddotx}, \eqref{eq: estim a" ddotx} and \eqref{eq: estim gamma" ddotx } we have \eqref{expected ddotb varphi} as we wanted. \\
The first term of \eqref{expected ddotb varphi} can be seen as
$\mathbb{E}[\varphi \int_0^h(\ddot{b}(X^\theta_s, \theta) - \ddot{b}(x, \theta)) ds] + \mathbb{E}[\varphi \int_0^h\ddot{b}(x, \theta) ds].$ \\
Using Holder inequality and the fact that $\varphi$ is bounded we get
$$\mathbb{E}[\int_0^h (\ddot{b}(X^\theta_s, \theta)- \ddot{b}(x, \theta)) ds \, \varphi_{h^\beta}(X_h^\theta - x)] \le $$
$$ \le c (\mathbb{E}[(\int_0^h (\ddot{b}(X^\theta_s, \theta)- \ddot{b}(x, \theta)) ds )^p])^\frac{1}{p} \le c (\mathbb{E}[(\int_0^h \left \| \frac{\partial \ddot{b}}{\partial x }\right \|_\infty |X_s^\theta - x| ds )^p])^\frac{1}{p}.$$
From Jensen inequality we get it is upper bounded by
$ c(h^{p-1}\int_0^h \mathbb{E}[|X_s^\theta - x|^p]ds )^\frac{1}{p} \le c (h^{p+1}(1 + |x|^p))^\frac{1}{p}, $
where we have used the second point of Lemma \ref{lemma: Moment inequalities}. It yields
$$\mathbb{E}[\varphi \int_0^h(\ddot{b}(X^\theta_s, \theta) - \ddot{b}(x, \theta)) ds] = R(\theta, h^{1 + \frac{1}{p}},x).$$
Therefore, considering $p = 2$, \eqref{expected ddotb varphi} becomes 
$\mathbb{E}[\varphi \, \ddot{X}^\theta_h]= \mathbb{E}[\varphi\, \ddot{b}(x, \theta)h] + R(\theta, h^{\frac{3}{2}}, x).$ \\
Replacing in \eqref{eq: dot m apres prop4} and using the development \eqref{eq: dl phi} or \eqref{eq: dl phi alpha >1} of the denominator we obtain, for $|x| \le h^{-k_0}$,
$$\ddot{m}_{\theta,h}(x)= h \ddot{b}(x, \theta) +  R(\theta, h^{\frac{3}{2}}, x) + R(\theta, h^{3 - \alpha\beta -\epsilon - \beta}, x) +  R(\theta, h^{3 - \alpha\beta -\epsilon - 2\beta}, x)+$$
$$+ R(\theta, h^{2 - \alpha\beta -\epsilon - \beta}, x) + R(\theta, h^{4 - 2\alpha\beta -\epsilon - 2\beta}, x) = h \ddot{b}(x, \theta) +  R(\theta, h^{\frac{3}{2} \land (2 - \alpha\beta -\epsilon - \beta) }, x).$$
\end{proof}
We want now to justify \eqref{eq: estimation 3dot m }. \\
In the expression of $\dddot{m}_{\theta,h}(y)$, the numerator is the sum of product of terms with the following form: \\
$\mathbb{E}[\varphi^{(k)} X_h^{h_0} \dot{X}_h^{h_1} \ddot{X}_h^{h_2} \dddot{X}_h^{h_3}]h^{-\beta k},$
where $k \ge 1$ and $h_1 + h_2 + h_3 \ge k$. \\
The only term with a different form is $\mathbb{E}[\varphi \dddot{X}]$, that is $R(\theta, h, y)$ by the boundedness of $\varphi$ and the equation \eqref{eq: estimation 3dot x}. \\
We observe that, using Proposition \ref{prop: truc moche z} defining $Z = \frac{X_h^{h_0} \dot{X}_h^{h_1} \ddot{X}_h^{h_2} \dddot{X}_h^{h_3}}{h^{h_1 + h_2 + h_3}} $, we get
$$|\mathbb{E}[\varphi^{(k)} X_h^{h_0} \dot{X}_h^{h_1} \ddot{X}_h^{h_2} \dddot{X}_h^{h_3}]|h^{-\beta k} \le h^{-\beta k + h_1 + h_2 + h_3 + 1 - \alpha \beta -\epsilon} \le h^{(1-\beta) k + 1 - \alpha \beta -\epsilon}. $$
We observe that the exponent on $h$ is more then $1$ if and only if $\beta < \frac{k}{k + \alpha} - \frac{\epsilon}{k + \alpha}$, with $ k \ge 1$. Since $\frac{1}{1 + \alpha} - \frac{\epsilon}{1 + \alpha}$ is the smallest, the Assumption $\beta < \frac{1}{1 + \alpha} - \frac{\epsilon}{1 + \alpha}$ that we added in Proposition \ref{prop: dl dotm ddotm} assures that $|\dddot{m}_{\theta,h}(y)| = R(\theta, h,y)$, as we wanted.

\subsection{Proof of limit theorems}\label{S:Proof_limit}
In this subsection we prove the theorems stated in Section \ref{S:Limit}.

\subsubsection{Proof of Proposition \ref{prop: LT1}}
\begin{proof}
$(i)$ follows from Lemma 4.4 in \hyperref[csl:16]{(Gloter, Loukianova, \& Mai, 2018)}, ergodic theorem and the $L^1$ convergence to zero of $\frac{1}{n} \sum_{i = 0}^{n-1}(1 + |X_{t_i}|)^c1_{\left \{|X_{t_i}| > \Delta_{n,i}^{-k} \right \} }$, which is a consequence of the third point of Lemma \ref{lemma 2.1 GLM}. Remark that in \hyperref[csl:16]{(Gloter, Loukianova, \& Mai, 2018)} the Lemma 4.4 is stated the for $\alpha \in (0,1)$ only. However an inspection of the proof shows that it is valid for $\alpha \in (0,2)$. \\
Concerning $(ii)$, we can see 
$\frac{1}{t_n} \sum_{i=0}^{n-1}\Delta_{n,i}f(X_{t_i}, \theta)\varphi_{\Delta_{n,i}^\beta}(X_{t_{i+1}} - X_{t_i})1_{\left \{|X_{t_i}| \le \Delta_{n,i}^{-k} \right \} }$ as
\begin{equation}
\frac{1}{t_n} \sum_{i=0}^{n-1}\Delta_{n,i}f(X_{t_i}, \theta) 1_{\left \{|X_{t_i}| \le \Delta_{n,i}^{-k} \right \} } +\frac{1}{t_n} \sum_{i=0}^{n-1}\Delta_{n,i}f(X_{t_i}, \theta)(\varphi_{\Delta_{n,i}^\beta}(X_{t_{i+1}} - X_{t_i}) - 1) 1_{\left \{|X_{t_i}| \le \Delta_{n,i}^{-k} \right \} }.
\label{eq: prop 1 varphi spezzato} 
\end{equation}
We have already showed in $(i)$ that on the first term of \eqref{eq: prop 1 varphi spezzato} we have the convergence wanted and so, in order to get the thesis, it is enough to prove the following:
\begin{equation}
\sup_{\theta \in \Theta} |\frac{1}{t_n} \sum_{i=0}^{n-1}\Delta_{n,i}f(X_{t_i}, \theta)(\varphi_{\Delta_{n,i}^\beta}(X_{t_{i+1}} - X_{t_i}) - 1) 1_{\left \{|X_{t_i}| \le \Delta_{n,i}^{-k} \right \} }| \xrightarrow{\mathbb{P}} 0
\label{eq: thesis prop 1} 
\end{equation}
We observe that 
$$|\frac{1}{t_n} \sum_{i=0}^{n-1}\Delta_{n,i}f(X_{t_i}, \theta)(\varphi_{\Delta_{n,i}^\beta}(X_{t_{i+1}} - X_{t_i}) - 1) 1_{\left \{|X_{t_i}| \le \Delta_{n,i}^{-k} \right \} }| \le |\frac{1}{n \Delta_n} \sum_{i=0}^{n-1}\Delta_{n,i}f(X_{t_i}, \theta)(\varphi_{\Delta_{n,i}^\beta}(X_{t_{i+1}} - X_{t_i}) - 1)|.$$
By the definition of $\varphi$, it is different from zero only if $|\Delta X_i| > \Delta_{n,i}^\beta$. Using Markov inequality and Lemma \ref{lemma: Moment inequalities},
\begin{equation}
\mathbb{P}(|X_{t_{i + 1}} - X_{t_i}| > \Delta_{n,i}^\beta ) \le \mathbb{E}[|X_{t_{i + 1}} - X_{t_i}|^2] \Delta_{n,i}^{-2\beta} \le c \Delta_{n,i}^{1 - 2\beta}.
\label{eq: stima prob con indicatrice} 
\end{equation}
It means that the left hand side of \eqref{eq: thesis prop 1} converges to zero in $L^1$ and so in probability, indeed
$$\mathbb{E}[\sup_{\theta \in \Theta}|\frac{1}{t_n} \sum_{i=0}^{n-1}\Delta_{n,i}f(X_{t_i}, \theta)(\varphi_{\Delta_{n,i}^\beta}(X_{t_{i+1}} - X_{t_i}) - 1) 1_{\left \{|X_{t_i}| \le \Delta_{n,i}^{-k} \right \} }|] \le$$
$$\le \mathbb{E}[ \sup_{\theta \in \Theta}|\frac{1}{n \Delta_n} \sum_{i=0}^{n-1}\Delta_{n,i}f(X_{t_i}, \theta) 1_{\left \{ |X_{t_{i + 1}} - X_{t_i}| > \Delta_{n,i}^\beta \right \} }|] \le \frac{1}{n \Delta_n} \sum_{i=0}^{n-1}\Delta_{n,i} \mathbb{E}[\sup_{\theta \in \Theta}|f(X_{t_i}, \theta)|^2]^\frac{1}{2} \mathbb{E}[|1_{\left \{ |X_{t_{i + 1}} - X_{t_i}| > \Delta_{n,i}^\beta \right \} }|^2]^\frac{1}{2} \le    $$
$$\le \frac{c}{n \Delta_n} \sum_{i=0}^{n-1}\Delta_{n,i} \mathbb{P}(|X_{t_{i + 1}} - X_{t_i}| > \Delta_{n,i}^\beta )^\frac{1}{2} \le c\Delta_{n,i}^{\frac{1}{2}-\beta},$$
where we have first used Cauchy-Schwarz inequality and then the polynomial growth of $|\sup_{\theta \in \Theta} f|$and the third point of Lemma \ref{lemma 2.1 GLM} and \eqref{eq: stima prob con indicatrice}. Since the exponent on $\Delta_{n,i}$ is positive we get the thesis.
\end{proof}

\subsubsection{Proof of Proposition \ref{prop: LT2} and Lemma \ref{lemma: conditional expected value zeta} }
\textit{Proof of Proposition \ref{prop: LT2}.} \\
In order to show that $\frac{1}{t_n} \sum_{i =0}^{n -1}f_{i,n}(X_{t_i}, \theta) \zeta_i \varphi_{\Delta_{n,i}^\beta}(X_{t_{i+1}} - X_{t_i})1_{\left \{|X_{t_i}| \le \Delta_{n,i}^{-k} \right \} }$ converges to zero in probability, we want to use the Lemma 9 of (missing citation) and so we have to show the following:
\begin{equation}
\frac{1}{t_n} \sum_{i =0}^{n -1} \mathbb{E}[f_{i,n}(X_{t_i}, \theta) \zeta_i \varphi_{\Delta_{n,i}^\beta}(X_{t_{i+1}} - X_{t_i})1_{\left \{|X_{t_i}| \le \Delta_{n,i}^{-k} \right \} }|\mathcal{F}_{t_i}] \rightarrow 0,
\label{eq: condition 1 proposition 2} 
\end{equation}
\begin{equation}
\frac{1}{(t_n)^2} \sum_{i =0}^{n -1} \mathbb{E}[f^2_{i,n}(X_{t_i}, \theta) \zeta^2_i \varphi^2_{\Delta_{n,i}^\beta}(X_{t_{i+1}} - X_{t_i})1_{\left \{|X_{t_i}| \le \Delta_{n,i}^{-k} \right \} }|\mathcal{F}_{t_i}] \rightarrow 0.
\label{eq: condition 2 proposition 2} 
\end{equation}
If Lemma \ref{lemma: conditional expected value zeta} holds we have that, using \eqref{eq: convergence expected value prop 2}, the left hand side of \eqref{eq: condition 1 proposition 2} results upper bounded by
$ \Delta_{n}^{\delta \land \frac{1}{2}} \frac{1}{n} \sum_{i =0}^{n -1}|f_{i,n}(X_{t_i}, \theta)| R(\theta, 1, X_{t_i}), $
where we have used the property \eqref{propriety power R} on $R$ and the fact that $|\Delta_{n,i}| \le \Delta_n$. Since the exponent on $\Delta_n$ is positive and $\frac{1}{n} \sum_{i =0}^{n -1}f_{i,n}(X_{t_i}, \theta) R(\theta, 1, X_{t_i})$ is bounded in $L^1$ using the polynomial growth of both $f_{i,n}$ and $R$ and the third point of Lemma \ref{lemma 2.1 GLM}, we get the convergence in probability \eqref{eq: condition 1 proposition 2}. \\
Concerning \eqref{eq: condition 2 proposition 2}, if Lemma \ref{lemma: conditional expected value zeta} holds we can use \eqref{eq: conv squared prop 2} getting that \eqref{eq: condition 2 proposition 2} is
$\frac{1}{n^2 \Delta_n} \sum_{i =0}^{n -1}f^2_{i,n}(X_{t_i}, \theta) R(\theta, 1, X_{t_i}),  $
where we have used also the property \eqref{propriety power R} on $R$ and the fact that $|\Delta_{n,i}| \le \Delta_n$. Since $n \Delta_n \rightarrow \infty$ and $ \frac{1}{n} \sum_{i =0}^{n -1}f^2_{i,n}(X_{t_i}, \theta) R(\theta, 1, X_{t_i})$ is bounded in $L^1$ by the polynomial growth of both $f_{i,n}$ and $R$ and the third point of Lemma \ref{lemma 2.1 GLM}, we get the convergence \eqref{eq: condition 2 proposition 2} as we wanted. \\
Hence, if Lemma \ref{lemma: conditional expected value zeta} holds, then Proposition \ref{prop: LT2} is proved. \finpv
\textit{Proof of Lemma \ref{lemma: conditional expected value zeta}.} \\
By the definition \eqref{eq: def zeta} of $\zeta_i$ and the dynamic of the process $X$, we get
\begin{equation}
\zeta_i = X_{t_{i+1}} - X_{t_i} - \int_{t_i}^{t_{i+ 1}} b(\theta_0, X_s) ds +  \Delta_{n,i}\, \int_{\mathbb{R} \backslash \left \{0 \right \}} z \, \gamma(X_{t_i})\,[1 - \varphi_{\Delta_{n,i}^\beta}(\gamma(X_{t_i})z)]\,F(z)dz.
\label{eq: reformulation zeta} 
\end{equation}
We write the left hand side of \eqref{eq: convergence expected value prop 2} by using the last equation and adding and subtracting $m_{\theta_0, \Delta_{n,i}}(X_{t_i})$:
\begin{equation}
\mathbb{E}[(X_{t_{i+1}} -m_{\theta_0, \Delta_{n,i}}(X_{t_i}))\varphi_{\Delta_{n,i}^\beta}(X_{t_{i+1}} - X_{t_i})1_{\left \{|X_{t_i}| \le \Delta_{n,i}^{-k} \right \} }|\mathcal{F}_{t_i}] + \mathbb{E}[(m_{\theta_0, \Delta_{n,i}}(X_{t_i}) - X_{t_i} - \int_{t_i}^{t_{i+ 1}} b(\theta_0, X_s) ds +
\label{eq: lemma 3 apres reformulation} 
\end{equation}
$$ +  \Delta_{n,i}\, \int_{\mathbb{R} \backslash \left \{0 \right \}} z \, \gamma(X_{t_i})\,[1 - \varphi_{\Delta_{n,i}^\beta}(\gamma(X_{t_i})z)]\,F(z)dz)\varphi_{\Delta_{n,i}^\beta}(X_{t_{i+1}} - X_{t_i})1_{\left \{|X_{t_i}| \le \Delta_{n,i}^{-k} \right \} }|\mathcal{F}_{t_i}].$$
By the $\mathcal{F}_{t_i}$-measurability of $X_{t_i}$, the first term of \eqref{eq: lemma 3 apres reformulation} is equal to
$$\mathbb{E}[(X_{t_{i+1}} -m_{\theta_0, \Delta_{n,i}}(X_{t_i}))\varphi_{\Delta_{n,i}^\beta}(X_{t_{i+1}} - X_{t_i})|\mathcal{F}_{t_i}]1_{\left \{|X_{t_i}| \le \Delta_{n,i}^{-k} \right \} },$$
that is zero by the definition of $m_{\theta_0, \Delta_{n,i}}$. \\
On the second term of \eqref{eq: lemma 3 apres reformulation} we use the development \eqref{eq: dl m} or \eqref{eq: dl m alpha>1}, respectively for $\alpha <1$ and $\alpha > 1$. Hence, we obtain
\begin{equation}
\mathbb{E}[(\int_{t_i}^{t_{i+ 1}}(b(\theta_0, X_{t_i})- b(\theta_0, X_s)) ds + R(\theta_0, \Delta_{n,i}^{1 + \delta}, X_{t_i}))\varphi_{\Delta_{n,i}^\beta}(X_{t_{i+1}} - X_{t_i})1_{\left \{|X_{t_i}| \le \Delta_{n,i}^{-k} \right \} }|\mathcal{F}_{t_i}],
\label{eq: lemma 3 pto 1} 
\end{equation}
where $\delta > 0$ is defined below equation \eqref{eq: link betwen m and zeta}. 
Using the boundedness of both $\varphi$ and the indicator function and \eqref{eq: expected value int b as R} on the first term of \eqref{eq: lemma 3 pto 1}, we get that \eqref{eq: lemma 3 pto 1} is upper bounded by 
$$R(\theta_0, \Delta_{n,i}^\frac{3}{2}, X_{t_i}) + R(\theta_0, \Delta_{n,i}^{1 + \delta}, X_{t_i}) = R(\theta_0, \Delta_{n,i}^{(1+ \delta) \land \frac{3}{2}}, X_{t_i}), $$
as we wanted. \\
Concerning the second point of Lemma \ref{lemma: conditional expected value zeta}, we use \eqref{eq: def zeta} in order to say that
\begin{equation}
\zeta_i^2 \le c(\int_{t_i}^{t_{i+1}} a(X_s)dW_s)^2 + c(\int_{t_i}^{t_{i+1}}\int_{\mathbb{R} \backslash \left \{0 \right \}} z \, \gamma(X_{s^-})\tilde{\mu}(ds,dz))^2 +c  \Delta^2_{n,i} (\int_{\mathbb{R} \backslash \left \{0 \right \}} z \, \gamma(X_{t_i})\,[1 - \varphi_{\Delta_{n,i}^\beta}(\gamma(X_{t_i})z)]\,F(z)dz)^2.
\label{eq: zeta_i^2} 
\end{equation}
Using this estimation in the left hand side of \eqref{eq: conv squared prop 2} we obtain three terms, the first is
$$\mathbb{E}[c(\int_{t_i}^{t_{i+1}} a(X_s)dW_s)^2\varphi^2_{\Delta_{n,i}^\beta}(X_{t_{i+1}} - X_{t_i})1_{\left \{|X_{t_i}| \le \Delta_{n,i}^{-k} \right \} }|\mathcal{F}_{t_i}] \le \mathbb{E}[c(\int_{t_i}^{t_{i+1}} a(X_s)dW_s)^2|\mathcal{F}_{t_i}],$$
by the boundedness of both $\varphi$ and the indicator function. Using the conditional form of Ito's isometry it is
\begin{equation}
c \mathbb{E}[\int_{t_i}^{t_{i+1}} a^2(X_s)ds |\mathcal{F}_{t_i}]= R(\theta_0, \Delta_{n,i}, X_{t_i}),
\label{eq: estimation a prop 2} 
\end{equation}
by the polynomial growth of $a$, the third point of Lemma \ref{lemma: Moment inequalities} and the definition of the function $R$. \\
We can upper bound the second term of \eqref{eq: zeta_i^2} using first of all the boundedness of both $\varphi$ and the indicator function, and then Kunita's inequality in the conditional form (Appendix of \hyperref[csl:24]{(Jacod \& Protter, 2011)}). We get the following estimation:
$$\mathbb{E}[c(\int_{t_i}^{t_{i+1}}\int_{\mathbb{R} \backslash \left \{0 \right \}} z \, \gamma(X_{s^-})\tilde{\mu}(ds,dz))^2 \varphi^2_{\Delta_{n,i}^\beta}(X_{t_{i+1}} - X_{t_i})1_{\left \{|X_{t_i}| \le \Delta_{n,i}^{-k} \right \} }|\mathcal{F}_{t_i}] \le$$
$$\le \mathbb{E}[c(\int_{t_i}^{t_{i+1}}\int_{\mathbb{R} \backslash \left \{0 \right \}} z \, \gamma(X_{s^-})\tilde{\mu}(ds,dz))^2 |\mathcal{F}_{t_i}] \le c\mathbb{E}[\int_{t_i}^{t_{i+1}}\int_{\mathbb{R} \backslash \left \{0 \right \}} |z|^2 \, \gamma^2(X_{s^-})\bar{\mu}(ds,dz) |\mathcal{F}_{t_i}]  \le  $$
\begin{equation}
\le c\mathbb{E}[\int_{t_i}^{t_{i+1}}\gamma^2(X_{s^-})ds |\mathcal{F}_{t_i}] = R(\theta_0, \Delta_{n,i}, X_{t_i}),
\label{eq: estimation gamma prop 2} 
\end{equation}
where in the last inequality and equality we have used, respectively, the definition of the compensator measure $\bar{\mu}$ and the polynomial growth of $\gamma$ and the third point of Lemma \ref{lemma: Moment inequalities}. \\
Concerning the third term of \eqref{eq: zeta_i^2}, we have already showed in Remark 3 an estimation, depending on $\alpha$, that is at most $\Delta_{n,i}^\frac{1}{2}$. Its square is therefore at least a $R(\theta, \Delta_{n,i}, X_{t_i})$ function, it follows that \eqref{eq: conv squared prop 2} holds. \\
We now want to prove \eqref{eq: x - m prop 2}.
Using \eqref{eq: link betwen m and zeta},
\begin{equation}
(X_{t_{i+1}} - m_{\theta_0}(X_{t_i}))^2 \le c\zeta_i^2 + c(\int_{t_i}^{t_{i+1}} b(X_s, \theta_0) ds - \Delta_{n,i}b(X_{t_i}, \theta_0))^2 + R(\theta_0, \Delta_{n,i}^{2 + 2 \delta}, X_{t_i}).
\label{eq: substitution X - mtheta } 
\end{equation}
We can replace it in \eqref{eq: x - m prop 2}, getting three terms that are of magnitude at most $\Delta_{n, i}$. \\
Indeed, on the first we can use \eqref{eq: conv squared prop 2}. \\
On the second term  we can use the boundedness of both $\varphi$ and the indicator function and Jensen inequality, getting 
$$c\mathbb{E}[(\int_{t_i}^{t_{i+1}} b(X_s, \theta_0) ds - \Delta_{n,i}b(X_{t_i}, \theta_0))^2\varphi^2_{\Delta_{n,i}^\beta}(X_{t_{i+1}} - X_{t_i})1_{\left \{|X_{t_i}| \le \Delta_{n,i}^{-k} \right \} }|\mathcal{F}_{t_i}] \le $$
$$c  \Delta_{n,i} \mathbb{E}[ \int_{t_i}^{t_{i+1}} (b(X_s, \theta_0) ds - b(X_{t_i}, \theta_0))^2|\mathcal{F}_{t_i}] \le $$
\begin{equation}
\le c  \Delta_{n,i} \mathbb{E}[ \int_{t_i}^{t_{i+1}} b^2(X_s, \theta_0) ds|\mathcal{F}_{t_i}] + c  \Delta^2_{n,i} \mathbb{E}[ b^2(X_{t_i}, \theta_0)|\mathcal{F}_{t_i}]= R(\theta_0, \Delta_{n,i}^2, X_{t_i}),  
\label{eq: estimation b al quadrato} 
\end{equation}
where in the last equality we have used the polynomial growth of $b$ on both of the two terms and moreover the third point of Lemma \ref{lemma: Moment inequalities} on the first term. \\
In conclusion, we obtain 
$$\mathbb{E}[(X_{t_{i+1}} - m_{\theta_0}(X_{t_i}))^2\varphi^2_{\Delta_{n,i}^\beta}(X_{t_{i+1}} - X_{t_i})1_{\left \{|X_{t_i}| \le \Delta_{n,i}^{-k} \right \} }|\mathcal{F}_{t_i}] = $$
$$ =  R(\theta_0, \Delta_{n,i}, X_{t_i}) + R(\theta_0, \Delta_{n,i}^2, X_{t_i}) + R(\theta_0, \Delta_{n,i}^{2+ 2 \delta}, X_{t_i}) = R(\theta_0, \Delta_{n,i}, X_{t_i}).$$
Hence, we have the thesis.  \finpv

\subsubsection{Proof of Proposition \ref{prop: LT3}.}
In order to prove Proposition \ref{prop: LT3}, the following lemma will be useful: \\
\begin{lemma}
Let us denote by $\tilde{X}^J$ the jump part of $X$ given by
\begin{equation}
\tilde{X}_t^J := \int_0^t \int_{\mathbb{R} \backslash \left \{0 \right \}} z \, \gamma(X_{s^-})\tilde{\mu}(ds,dz), \qquad t \ge 0
\label{eq: definition X^J} 
\end{equation}
and 
$\Delta_i \tilde{X}^J : = \tilde{X}_{t_{i+1}}^J - \tilde{X}_{t_i}^J. $ \\
Then, for each $q \ge 2$,  $\exists \epsilon > 0$ such that
\begin{equation}
\mathbb{E}[|\Delta_i \tilde{X}^J\varphi_{\Delta_{n,i}^\beta}(X_{t_{i+1}} - X_{t_i})|^q |\mathcal{F}_{t_i} ] = R(\theta_0, \Delta_{n,i}^{1 + \beta(q - \alpha)}, X_{t_i}) = R(\theta_0, \Delta_{n,i}^{1 + \epsilon}, X_{t_i}).
\label{eq: salti trascurabili} 
\end{equation}
\label{lemma: lemma per mostrare prop LT3} 
\end{lemma}
\textit{Proof of Lemma \ref{lemma: lemma per mostrare prop LT3}.} \\
For all $n \in \mathbb{N}$ and $i \in \mathbb{N}$ we define the set on which all the jumps of $L$ on the interval $(t_i, t_{i+1}]$ are small:
\begin{equation}
N_n^i : = \left \{ |\Delta L_s| \le \frac{4 \Delta_{n,i}^\beta}{\gamma_{min}}; \quad \forall s \in (t_i, t_{i+1}]  \right \},
\label{eq: definition Nî_n} 
\end{equation}
where $\Delta L_s := L_s - L_{s^-}$. 
We hence split the left hand side of \eqref{eq: salti trascurabili} as
\begin{equation}
\mathbb{E}[|\Delta_i \tilde{X}^J\varphi_{\Delta_{n,i}^\beta}(X_{t_{i+1}} - X_{t_i})|^q 1_{N^i_n} |\mathcal{F}_{t_i} ] + \mathbb{E}[|\Delta_i \tilde{X}^J\varphi_{\Delta_{n,i}^\beta}(X_{t_{i+1}} - X_{t_i})|^q 1_{(N^i_n)^c} |\mathcal{F}_{t_i} ]. 
\label{eq: salti con partizione} 
\end{equation}
We now observe that, by the definition of $N^i_n$, 
$$|\mathbb{E}[|\Delta_i \tilde{X}^J\varphi_{\Delta_{n,i}^\beta}(X_{t_{i+1}} - X_{t_i})|^q 1_{N^i_n} |\mathcal{F}_{t_i} ] | \le $$
\begin{equation}
\le c \mathbb{E}[|\int_{t_i}^{t_{i+1}} \int_{|z| \le \frac{4 \Delta_{n,i}^\beta}{\gamma_{min}}} z \, \gamma(X_{s^-}) \tilde{\mu}(ds,dz)|^q+ |\int_{t_i}^{t_{i+1}} \int_{|z| \ge \frac{4 \Delta_{n,i}^\beta}{\gamma_{min}}} |z| \, |\gamma(X_{s^-})| \bar{\mu}(ds,dz) |^q |\mathcal{F}_{t_i} ].
\label{eq: con Nin} 
\end{equation}
We observe that the order of the second term depend on $\alpha$. Acting as in Remark 3, we get that its order is $\Delta_{n,i}^q$ if $\alpha \in (0,1)$ while it is $\Delta_{n,i}^{q + q\beta(1 - \alpha)}$ if $\alpha \in (1,2)$. Since $q$ is more than $q + q\beta(1 - \alpha)$ if and only if $\alpha > 1$, we can say that the second term of \eqref{eq: con Nin} is upper bounded by $c\Delta_{n,i}^{q \land (q + q\beta(1 - \alpha))}$. The first term of \eqref{eq: con Nin} is instead upper bounded by 
 $$c \mathbb{E}[|\int_{t_i}^{t_{i+1}} \int_{|z| \le \frac{4 \Delta_{n,i}^\beta}{\gamma_{min}}} |z|^2 \, |\gamma(X_{s^-})|^2 \bar{\mu}(ds,dz)|^\frac{q}{2} |\mathcal{F}_{t_i} ] + c\mathbb{E}[\int_{t_i}^{t_{i+1}} \int_{|z| \le \frac{4 \Delta_{n,i}^\beta}{\gamma_{min}}} |z|^q \, |\gamma(X_{s^-})|^q \bar{\mu}(ds,dz)|\mathcal{F}_{t_i} ] \le $$
$$\le c \left \| \gamma \right \|^q_\infty (\mathbb{E}[|\int_{t_i}^{t_{i+1}} \int_{|z| \le \frac{4 \Delta_{n,i}^\beta}{\gamma_{min}}} |z|^{1 - \alpha} dzds|^\frac{q}{2} |\mathcal{F}_{t_i} ] + \mathbb{E}[\int_{t_i}^{t_{i+1}} \int_{|z| \le \frac{4 \Delta_{n,i}^\beta}{\gamma_{min}}} |z|^{q - 1 - \alpha}  dzds|\mathcal{F}_{t_i} ]  ) \le$$
\begin{equation}
\le c(\int_{t_i}^{t_{i+1}}\Delta_{n,i}^{(2 - \alpha)\beta}ds)^\frac{q}{2} + c \int_{t_i}^{t_{i+1}}\Delta_{n,i}^{(q - \alpha)\beta}ds + \Delta_{n,i}^q \le c(\Delta_{n,i}^{(1 + (2 - \alpha)\beta)\frac{q}{2}} + \Delta_{n,i}^{(q - \alpha)\beta + 1}) = R(\theta_0, \Delta_{n,i}^{(q - \alpha)\beta + 1}, X_{t_i}), 
\label{eq: salti Ni_n} 
\end{equation}
where we have used Kunita inequality, the definition of $\bar{\mu}$ and the second point of Assumption 4. 
Using the consideration below equation \eqref{eq: con Nin} and \eqref{eq: salti Ni_n} we get
\begin{equation}
|\mathbb{E}[|\Delta_i \tilde{X}^J\varphi_{\Delta_{n,i}^\beta}(X_{t_{i+1}} - X_{t_i})|^q 1_{N^i_n} |\mathcal{F}_{t_i} ] | \le R(\theta_0, \Delta_{n,i}^{(q - \alpha)\beta + 1}, X_{t_i}) + R(\theta_0, \Delta_{n,i}^{q \land (q + q\beta(1 - \alpha))}, X_{t_i}) = R(\theta_0, \Delta_{n,i}^{(q - \alpha)\beta + 1}, X_{t_i}).
\label{eq: stima finale con Nin} 
\end{equation}
For $\alpha \in (0,2)$, $\alpha \neq 1$ and $\beta \in(0, \frac{1}{2})$ the exponent on $\Delta_{n,i}$ can be seen as $1 + \epsilon$, with $\epsilon > 0$. \\
Concerning the second term of \eqref{eq: salti con partizione}, we have
\begin{equation}
\mathbb{E}[|\Delta_i \tilde{X}^J\varphi_{\Delta_{n,i}^\beta}(X_{t_{i+1}} - X_{t_i})|^q 1_{(N^i_n)^c} |\mathcal{F}_{t_i} ] \le c\mathbb{E}[(|\Delta_i X|^q + |\Delta X_i^c|^q) |\varphi^q_{\Delta_{n,i}^\beta}(X_{t_{i+1}} - X_{t_i}) | 1_{(N^i_n)^c} |\mathcal{F}_{t_i} ],
\label{eq: stima Nin^c} 
\end{equation}
where $|\Delta_i X| : = |X_{t_{i+1}} - X_{t_i}|$ and $\Delta X_i^c$ is the increment of the continuous part of X in the interval $(t_i, t_{i+1}]$. We observe that, by the definition of $\varphi_{\Delta_{n,i}^\beta}(X_{t_{i+1}} - X_{t_i})$, the first term in the right hand side is different from zero only if $|\Delta_i X|^q \le \Delta_{n,i}^{\beta q}$. Therefore
\begin{equation}
\mathbb{E}[|\Delta_i X|^q \varphi^q_{\Delta_{n,i}^\beta}(X_{t_{i+1}} - X_{t_i}) 1_{(N^i_n)^c} |\mathcal{F}_{t_i} ] \le \Delta_{n,i}^{\beta q}\mathbb{P}_i((N^i_n)^c) \le c\Delta_{n,i}^{\beta q + 1 - \alpha \beta}.
\label{eq: primo termine Nin^c} 
\end{equation}
Indeed
\begin{equation}
\mathbb{P}_i((N^i_n)^c) = \mathbb{P}_i(\exists s \in (t_i, t_{i+1}]\, : \, |\Delta L_s| > \frac{4 \Delta_{n,i}^\beta}{\gamma_{min}} ) \le c \int_{t_i}^{t_{i+1}} \int_{\frac{4 \Delta_{n,i}^\beta}{\gamma_{min}} }^{\infty} F(z) dz ds \le c \Delta_{n,i}^{1 - \alpha \beta},
\label{eq: estim Nin^c} 
\end{equation}
where we have used the third point of Assumption 4. Since $q \ge 2$, $\beta q + 1 - \alpha \beta$ is always more than $1$. \\
In the same way
\begin{equation}
\mathbb{E}[|\Delta X_i^c|^q \varphi^q_{\Delta_{n,i}^\beta}(X_{t_{i+1}} - X_{t_i}) 1_{(N^i_n)^c} |\mathcal{F}_{t_i} ] \le c\Delta_{n,i}^{\frac{1}{2} q} \Delta_{n,i}^{1 - \alpha \beta}(1 + |X_{t_i}|^c ),
\label{eq: secondo termine Nin^c} 
\end{equation}
that is again more than $1$. Using \eqref{eq: salti con partizione}, \eqref{eq: stima finale con Nin}, \eqref{eq: primo termine Nin^c} and \eqref{eq: secondo termine Nin^c} we get the thesis.
\finpv
We can now prove Proposition \ref{prop: LT3}. \\ \\
\textit{Proof of Proposition \ref{prop: LT3}.} \\
We denote
\begin{equation}
s_i^n: = \frac{1}{\sqrt[]{t_n}}(X_{t_{i+1}} - m_{\theta_0}(X_{t_i}))f(X_{t_i}, \theta)\varphi_{\Delta_{n,i}^\beta}(X_{t_{i+1}} - X_{t_i})1_{\left \{|X_{t_i}| \le \Delta_{n,i}^{-k} \right \} }.
\label{eq: definition s_i^n in the dotU_n convergence} 
\end{equation}
In order to show the asymptotic normality we have to prove that $s^n$ is a martingale difference array such that
\begin{equation}
\sum_{i= 0}^{n-1} \mathbb{E}[|s_i^n|^{2 + r}| \mathcal{F}_{t_i}] \xrightarrow{\mathbb{P}} 0,
\label{eq: condition 1' conv dotU_n} 
\end{equation}
for a constant $\delta > 0$, and
\begin{equation}
\sum_{i= 0}^{n-1} \mathbb{E}[|s_i^n|^2| \mathcal{F}_{t_i}] \xrightarrow{\mathbb{P}} \int_\mathbb{R} f^2(x, \theta) a^2(x) \pi(dx),
\label{eq: condition 2' conv dotU_n} 
\end{equation}
c.f. Theorem A2 in the Appendix of \hyperref[csl:10]{(Shimizu \& Yoshida, 2006)}. \\
We observe that $s_i^n$ is a martingale difference array since, $\forall i \ge 0$,
$$\mathbb{E}[s_i^n| \mathcal{F}_{t_i}] = \frac{f(X_{t_i}, \theta)1_{\left \{|X_{t_i}| \le \Delta_{n,i}^{-k} \right \} }}{\sqrt[]{t_n}} \mathbb{E}[(X_{t_{i+1}} - m_{\theta_0}(X_{t_i}))\varphi_{\Delta_{n,i}^\beta}(X_{t_{i+1}} - X_{t_i})| \mathcal{F}_{t_i}] = 0,$$
by the measurability of $f$ and the indicator function and the definition of $m_{\theta_0}(X_{t_i})$. \\
We now want to prove \eqref{eq: condition 2' conv dotU_n}. Using \eqref{eq: link betwen m and zeta} and the definition of $\zeta_i$ we have that
\begin{equation}
(X_{t_{i+1}} - m_{\theta_0}(X_{t_i}))^2 = (\int_{t_i}^{t_{i+1}} a(X_s) dW_s)^2 + 2B_{i,n}\int_{t_i}^{t_{i+1}} a(X_s) dW_s + B^2_{i,n},
\label{eq: x - m funzione di a } 
\end{equation}
where 
$$B_{i,n} : = \int_{t_i}^{t_{i+1}} (b(X_s, \theta_0) - b(X_{t_i}, \theta_0))  ds + R(\theta_0, \Delta_{n,i}^{1 + \delta}, X_{t_i}) + $$
$$ + \int_{t_i}^{t_{i+1}} \int_{\mathbb{R} \backslash \left \{0 \right \}} z \, \gamma(X_{s^-})\tilde{\mu}(ds,dz) +
\Delta_{n,i}  \int_{\mathbb{R} \backslash \left \{0 \right \}} z \, \gamma(X_{s^-}) [1 - \varphi_{\Delta^\beta_{n,i}}(\gamma(X_{t_i})z)]F(z)dz. $$
Replacing \eqref{eq: x - m funzione di a } in the definition \eqref{eq: definition s_i^n in the dotU_n convergence} of $s_i^n$ we get three terms. We start proving that
\begin{equation}
\frac{1}{t_n} \sum_{i =0}^{n-1} \mathbb{E}[B_{i,n}^2 f^2(X_{t_i}, \theta)\varphi^2_{\Delta_{n,i}^\beta}(X_{t_{i+1}} - X_{t_i})1_{\left \{|X_{t_i}| \le \Delta_{n,i}^{-k} \right \} } | \mathcal{F}_{t_i} ] \xrightarrow{\mathbb{P}} 0.
\label{eq: convergence Bin} 
\end{equation}
Indeed, 
$$\mathbb{E}[B_{i,n}^2\varphi^2_{\Delta_{n,i}^\beta}(X_{t_{i+1}} - X_{t_i})| \mathcal{F}_{t_i} ] \le c
\mathbb{E}[[(\int_{t_i}^{t_{i+1}} (b(X_s, \theta_0) - b(X_{t_i}, \theta_0))  ds) ^2  + $$
\begin{equation}
 + R(\theta_0, \Delta_{n,i}^{2 + 2 \delta}, X_{t_i}) + (\int_{t_i}^{t_{i+1}} \int_{\mathbb{R} \backslash \left \{0 \right \}} z \, \gamma(X_{s^-})\tilde{\mu}(ds,dz))^2 +
 \label{eq: Bin per esteso} 
\end{equation}
$$ + (\Delta_{n,i}  \int_{\mathbb{R} \backslash \left \{0 \right \}} z \, \gamma(X_{s^-}) [1 - \varphi_{\Delta^\beta_{n,i}}(\gamma(X_{t_i})z)]F(z)dz)^2  ]\varphi^2_{\Delta_{n,i}^\beta}(X_{t_{i+1}} - X_{t_i})| \mathcal{F}_{t_i} ] \le$$
\begin{equation}
\le R(\theta_0, \Delta_{n,i}^2, X_{t_i}) +  R(\theta_0, \Delta_{n,i}^{2 + 2 \delta}, X_{t_i}) + R(\theta_0, \Delta_{n,i}^{(1 + \beta(q - \alpha)) \land (1 + \epsilon)}, X_{t_i}), 
\label{eq: sol Bin} 
\end{equation}
where we have used \eqref{eq: estimation b al quadrato} on the first term of \eqref{eq: Bin per esteso}, \eqref{eq: salti trascurabili} of the previous lemma on the third and Remark 3 on the fourth. Indeed, in Remark 3, we found that the last term in less than $ R(\theta_0, \Delta_{n,i}^{2}, X_{t_i}) $ if $\alpha  \le 1$ and less than $ R(\theta_0, \Delta_{n,i}^{2 + 2\beta(1 - \alpha)}, X_{t_i}) $ if $\alpha > 1$; in both cases the exponent on $\Delta_{n,i}$ is always more than $1$, hence we can write it as $1 + \epsilon$. \\
We can upper bound with \eqref{eq: sol Bin} the left hand side of \eqref{eq: convergence Bin} getting
$\frac{1}{t_n} \sum_{i =0}^{n-1} f^2(X_{t_i}, \theta) R(\theta_0, \Delta_{n,i}^{1 + \epsilon}, X_{t_i}) 1_{\left \{|X_{t_i}| \le \Delta_{n,i}^{-k} \right \} },$
that converges to $0$ in norm $L^1$ by the polynomial growth of both $f$ and $R$ and the third point of Lemma \ref{lemma 2.1 GLM} and using that $| \Delta_{n,i}| \le \Delta_n $. We obtain therefore the convergence in probability \eqref{eq: convergence Bin} wanted. \\
Let us now consider the contribution of the first term of \eqref{eq: x - m funzione di a } for the proof of \eqref{eq: condition 2' conv dotU_n}.  
We can see it as
\begin{equation}
\frac{1}{t_n} \sum_{i =0}^{n-1} f^2(X_{t_i}, \theta) \mathbb{E}[(\int_{t_i}^{t_{i+1}} a(X_s) dW_s)^2 | \mathcal{F}_{t_i} ]1_{\left \{|X_{t_i}| \le \Delta_{n,i}^{-k} \right \} } + 
\label{eq: a inizio} 
\end{equation}
$$ + \frac{1}{t_n} \sum_{i =0}^{n-1} f^2(X_{t_i}, \theta) \mathbb{E}[(\int_{t_i}^{t_{i+1}} a(X_s) dW_s)^2 (\varphi^2_{\Delta_{n,i}^\beta}(X_{t_{i+1}} - X_{t_i}) -1) | \mathcal{F}_{t_i} ]1_{\left \{|X_{t_i}| \le \Delta_{n,i}^{-k} \right \} }.$$
On the first term of \eqref{eq: a inizio} we use Ito's isometry, getting
$$\frac{1}{t_n} \sum_{i =0}^{n-1} f^2(X_{t_i}, \theta) \mathbb{E}[\int_{t_i}^{t_{i+1}} a(X_s)^2 ds | \mathcal{F}_{t_i} ]1_{\left \{|X_{t_i}| \le \Delta_{n,i}^{-k} \right \} } = $$
\begin{equation}
= \frac{1}{t_n} \sum_{i =0}^{n-1} f^2(X_{t_i}, \theta) (\Delta_{n,i} a^2(X_{t_i}) + R(\theta_0, \Delta_{n,i}^\frac{3}{2}, X_{t_i}))1_{\left \{|X_{t_i}| \le \Delta_{n,i}^{-k} \right \} },
\label{eq: sostituisco a(Xs)} 
\end{equation}
where we have used \eqref{eq: expected value int b as R} with $a^2$ in place of $b$. Using the first point of Proposition 1  we get that
\begin{equation}
\frac{1}{t_n} \sum_{i =0}^{n-1} f^2(X_{t_i}, \theta) \Delta_{n,i} a^2(X_{t_i}) 1_{\left \{|X_{t_i}| \le \Delta_{n,i}^{-k} \right \} } \xrightarrow{\mathbb{P}} \int_{\mathbb{R}} f^2(x, \theta)a^2(x) \pi(dx),
\label{eq: convergence a} 
\end{equation}
while $\frac{1}{t_n} \sum_{i =0}^{n-1} f^2(X_{t_i}, \theta) R(\theta_0, \Delta_{n,i}^\frac{3}{2}, X_{t_i})1_{\left \{|X_{t_i}| \le \Delta_{n,i}^{-k} \right \} } $ goes to zero in norm $L^1$ and therefore in probability. \\
Let us now consider the second term of \eqref{eq: a inizio}.
Using Cauchy- Schwarz inequality we get it is upper bounded by
$$\frac{1}{n \Delta_n} \sum_{i =0}^{n-1} f^2(X_{t_i}, \theta) \mathbb{E}[|\int_{t_i}^{t_{i+1}} a(X_s) dW_s|^{4}| \mathcal{F}_{t_i} ]^\frac{1}{2} \mathbb{E}[|\varphi^2_{\Delta_{n,i}^\beta}(X_{t_{i+1}} - X_{t_i}) -1|^2 | \mathcal{F}_{t_i} ]^\frac{1}{2} \le$$
$$ \le \frac{1}{n \Delta_n} \sum_{i =0}^{n-1} f^2(X_{t_i}, \theta) \mathbb{E}[|\int_{t_i}^{t_{i+1}} a^2(X_s) ds|^{2}| \mathcal{F}_{t_i} ]^\frac{1}{2} \mathbb{E}[1_{\left \{| X_{t_{i+1}} - X_{t_i}| > \Delta_{n,i}^{\beta} \right \} } | \mathcal{F}_{t_i} ]^\frac{1}{2},$$
where we have used Burkholder Davis Gundy inequality and the fact that, by the definition of $\varphi$, it is different from $0$ only if $| X_{t_{i+1}} - X_{t_i}| > \Delta_{n,i}^{\beta}$. Using Jensen inequality and \eqref{eq: stima prob con indicatrice} in the conditional form we can upper bound it with
$$ \frac{1}{n \Delta_n} \sum_{i =0}^{n-1} f^2(X_{t_i}, \theta) \mathbb{E}[\Delta_{n,i}^2|\frac{1}{\Delta_{n,i}}\int_{t_i}^{t_{i+1}} a^2(X_s) ds|^{2}| \mathcal{F}_{t_i} ]^\frac{1}{2} \mathbb{P}(| X_{t_{i+1}} - X_{t_i}| > \Delta_{n,i}^{\beta} | \mathcal{F}_{t_i} )^\frac{1}{2} \le$$
$$ \frac{c}{n \Delta_n} \sum_{i =0}^{n-1} f^2(X_{t_i}, \theta) \mathbb{E}[\Delta_{n,i} \int_{t_i}^{t_{i+1}} a^4(X_s) ds| \mathcal{F}_{t_i} ]^\frac{1}{2} R(\theta_0, \Delta_{n,i}^{\frac{1}{2} - \beta}, X_{t_i}) \le $$
\begin{equation}
 \le \frac{c}{n \Delta_n} \sum_{i =0}^{n-1} f^2(X_{t_i}, \theta) \Delta_{n,i}^\frac{1}{2} [\Delta_{n,i} a^4(X_{t_i}) + R(\theta_0, \Delta_{n,i}^\frac{3}{2}, X_{t_i})]^\frac{1}{2} R(\theta_0, \Delta_{n,i}^{\frac{1}{2} - \beta}, X_{t_i}),
\label{eq: conv zero con prob} 
\end{equation}
where we have also used \eqref{eq: expected value int b as R} with $a^4$ in place of $b$. We observe that \eqref{eq: conv zero con prob} goes to $0$ in $L^1$ and therefore in probability, indeed its $L^1$ norm is upper bounded by
$$ \le \Delta_n^{\frac{1}{2} - \beta} \frac{c}{n } \sum_{i =0}^{n-1} \mathbb{E} [f^2(X_{t_i}, \theta) R(\theta_0, 1, X_{t_i}) ( a^2(X_{t_i}) + R(\theta_0, \Delta_{n,i}^\frac{3}{4}, X_{t_i}))],$$
that goes to $0$ by the polynomial growth of $f$, $R$ and $a$ and the third point of Lemma \ref{lemma: Moment inequalities} and since $\beta < \frac{1}{2}$. \\
Let us now consider the second term of \eqref{eq: x - m funzione di a } for the proof of \eqref{eq: condition 2' conv dotU_n}. Using Cauchy-Schwarz inequality, \eqref{eq: sol Bin} and Ito's isometry we get
$$\frac{2}{t_n} \sum_{i =0}^{n-1} f^2(X_{t_i}, \theta) \mathbb{E}[B_{i,n} \int_{t_i}^{t_{i+1}} a(X_s) dW_s \, \varphi^2_{\Delta_{n,i}^\beta}(X_{t_{i+1}} - X_{t_i})1_{\left \{|X_{t_i}| \le \Delta_{n,i}^{-k} \right \} }| \mathcal{F}_{t_i} ] \le$$
$$\le \frac{c}{n \Delta_n} \sum_{i =0}^{n-1} f^2(X_{t_i}, \theta)R(\theta_0, \Delta_{n,i}^{1 + \epsilon}, X_{t_i})^\frac{1}{2}\mathbb{E}[\int_{t_i}^{t_{i+1}} a(X_s)^2 ds | \mathcal{F}_{t_i} ]^\frac{1}{2} \le$$
\begin{equation}
\le \Delta_n^{\frac{\epsilon}{2}} \frac{c}{n } \sum_{i =0}^{n-1} f^2(X_{t_i}, \theta)R(\theta_0, 1, X_{t_i})(a^2(X_{t_i}) + R(\theta_0, \Delta_{n,i}^\frac{1}{4}, X_{t_i})),
\label{eq: doppio prodotto} 
\end{equation}
where in the last inequality we have used the property \eqref{propriety power R} of $R$ and \eqref{eq: sostituisco a(Xs)} with the trivial estimation $|\Delta_{n,i}| \le \Delta_n$. By the polynomial growth of both $a$, $f$ and $R$ and the fact that the exponent on $\Delta_n$ is positive we have that \eqref{eq: doppio prodotto} converges to $0$ en norm $L^1$. Hence it converges to $0$ in probability, \eqref{eq: condition 2' conv dotU_n} follows. \\
Our goal is now to prove \eqref{eq: condition 1' conv dotU_n}. Using \eqref{eq: x - m funzione di a } we have that
$$\sum_{i= 0}^{n-1} \mathbb{E}[|s_i^n|^{2 + r}| \mathcal{F}_{t_i}] \le $$
\begin{equation}
\le c \frac{1}{(n \Delta_n)^{1 + \frac{r}{2}}}\sum_{i= 0}^{n-1}f^{2 + r}(X_{t_i}, \theta) (\mathbb{E}[B_{i,n}^{2 + r} \varphi^{2 +r}_{\Delta_{n,i}^\beta}(X_{t_{i+1}} - X_{t_i})|\mathcal{F}_{t_i}] + \mathbb{E}[( \int_{t_i}^{t_{i+1}} a(X_s) dW_s)^{2 + r}|\mathcal{F}_{t_i}]).
\label{eq: preuve con 2 + delta} 
\end{equation}
We act as we have already done in the proof of \eqref{eq: condition 2' conv dotU_n} on the first term of \eqref{eq: preuve con 2 + delta}:  using \eqref{eq: salti trascurabili} we get it is upper bounded by
$$ \frac{c}{(n \Delta_n)^{1 + \frac{r}{2}}}\sum_{i= 0}^{n-1}f^{2 + r}(X_{t_i}, \theta) R(\theta_0, \Delta_{n,i}^{1 + \epsilon}, X_{t_i})\le \Delta_n^\epsilon \frac{c}{(n \Delta_n)^{\frac{r}{2}}} \frac{1}{n}\sum_{i= 0}^{n-1}f^{2 + r}(X_{t_i}, \theta) R(\theta_0, 1, X_{t_i}), $$
that converges to $0$ in norm $L^1$ (and therefore in probability) since $\epsilon > 0 $ and $n \Delta_n \rightarrow \infty $ for $n \rightarrow \infty$. Concerning the second term of \eqref{eq: preuve con 2 + delta}, using Burkholder-Davis-Gundy 
inequality and \eqref{eq: sostituisco a(Xs)} we have
\begin{equation}
\mathbb{E}[( \int_{t_i}^{t_{i+1}} a(X_s) dW_s)^{2 + r}|\mathcal{F}_{t_i}]) \le R(\theta_0, \Delta_{n,i}^{1 + \frac{r}{2}}, X_{t_i}).
\label{eq: a con 2 + delta} 
\end{equation}
Using \eqref{eq: a con 2 + delta} we get that the second term of \eqref{eq: preuve con 2 + delta} is upper bounded by
$\frac{c}{n^\frac{r}{2}} \frac{1}{n } \sum_{i= 0}^{n-1}f^{2 + r}(X_{t_i}, \theta)R(\theta_0, 1, X_{t_i}),$
that converges to $0$ in norm $L^1$ and hence in probability since $n^\frac{r}{2} \rightarrow \infty$. We deduce \eqref{eq: condition 1' conv dotU_n} and therefore the wanted asymptotic normality. \finpv
\subsection{Proof of Propositions \ref{prop: truc moche h} and \ref{prop: truc moche z}.}\label{S:Proof_prop moche}
Since Proposition \ref{prop: truc moche h} is a consequence of Proposition \ref{prop: truc moche z}, let us start with the proof of Proposition \ref{prop: truc moche z}. To lighten the notation we forget the dependence on $\theta$ of $X^\theta$ and $Z_\theta$. \\
\textit{Proof of Proposition \ref{prop: truc moche z}.} \\
Using $\tilde{X}_t^J$ defined in \eqref{eq: definition X^J}, we introduce the event
\begin{equation}
E_h: = \left \{  \tilde{X}_h^J: =  \int_0^t \int_{\mathbb{R} \backslash \left \{0 \right \}} z \, \gamma(X_{s^-})\tilde{\mu}(ds,dz) \in [\frac{1}{2} h^\beta, 4 h^\beta ] \right \}.
\label{eq: definition E} 
\end{equation}
We have that
\begin{equation}
\mathbb{E}[|Z \varphi_{h^\beta}^{(k)}(X_h - x)|] = \mathbb{E}[|Z \varphi_{h^\beta}^{(k)}(X_h - x)| 1_{E_h}] + \mathbb{E}[|Z \varphi_{h^\beta}^{(k)}(X_h - x)|1_{E_h^c}].
\label{eq: prop 5 con decoupe su E} 
\end{equation}
We observe that, by its definition, $\varphi_{h^\beta}^{(k)}(X_h - x)$ is different from $0$ only if $|X_h - x| \in [h^\beta, 2 h^\beta] $.
But  $\Delta_h X := |X_h - x| = |X^c_h - x + \tilde{X}^J_h|$ hence on $E_h^c$, where $ \tilde{X}_h^J \notin [\frac{1}{2} h^\beta, 4 h^\beta ]$, from $|X_h - x| \in [h^\beta, 2 h^\beta] $ we deduce that it must be $|X^c_h - x| \ge \frac{1}{2} h^\beta$. Using this observation and Holder inequality we have that the second term on the right hand side of \eqref{eq: prop 5 con decoupe su E} is upper bounded by
$$(\mathbb{E}[|Z|^{p}])^\frac{1}{p} (\mathbb{E}[|\varphi_{h^\beta}^{(k)}(X_h - x)|^{q}1_{E_h^c}])^\frac{1}{q} \le c(\mathbb{P}(|X^c_h - x| \ge \frac{1}{2} h^\beta))^\frac{1}{q} \le ch^{\frac{r}{q}(\frac{1}{2} - \beta)}$$
$\forall r > 1$, where we have also used that $Z$ is bounded in $L^p$ and Remark 2 in \hyperref[csl:16]{(Gloter, Loukianova, \& Mai, 2018)}. \\
In order to estimate the first term on the right hand side of \eqref{eq: prop 5 con decoupe su E} we need the following lemma that we will prove at the end of the section:
\begin{lemma}
Let us consider $E_h$, the set defined in \eqref{eq: definition E}. We have
\begin{equation}
\mathbb{P}(E_h) \le R(\theta, h^{1 - \beta \alpha}, x).
\label{eq: P(E)} 
\end{equation}
\label{lemma: P(E)} 
\end{lemma}
If $Z \in \mathcal{Z}_{h,c,p}$, then using Holder inequality, the estimation \eqref{eq: P(E)} and the boundedness of $Z$ in $L^p$ we get
$$\mathbb{E}[|Z \varphi_{h^\beta}^{(k)}(X_h - x)| 1_{E_h}] \le (\mathbb{E}[|Z|^{p}]^\frac{1}{p}) (\mathbb{E}[|\varphi_{h^\beta}^{(k)}(X_h - x)|^{q} 1_{E_h}])^\frac{1}{q} \le$$
$$\le cR(\theta, h^{1 - \beta \alpha}, x)^\frac{1}{q} = cR(\theta, h^{\frac{1 - \beta \alpha}{q}}, x), $$
with $\frac{1}{p} + \frac{1}{q} = 1$. Hence, we get the Proposition \ref{prop: truc moche z}. \finpv

Proposition \ref{prop: truc moche h} is a consequence of Proposition \ref{prop: truc moche z}, observing that $(h(X_u, \theta))_{\theta \in \Theta} \in \mathcal{Z}_{t_{i+1} - t_i, c, p}$, for $u \in [t_i, t_{i + 1}]$, and the Markov property. \\
\\
In conclusion, we prove Lemma \ref{lemma: P(E)}. \\ \\
\textit{Proof of Lemma \ref{lemma: P(E)}.} \\
We use again the set $N_n^i$ defined in \eqref{eq: definition Nî_n}. We have
\begin{equation}
\mathbb{P}(E_h)= \mathbb{P}(E_h \cap N_n^i ) + \mathbb{P}(E_h \cap (N_n^i)^c ).
\label{eq: eh splittato} 
\end{equation}
On the second term of \eqref{eq: eh splittato} we use \eqref{eq: estim Nin^c}, getting
\begin{equation}
\mathbb{P}(E_h \cap (N_n^i)^c ) \le \mathbb{P}((N_n^i)^c) \le ch^{1 - \alpha \beta}.
\label{eq: Eh con Nin^c } 
\end{equation}
Concerning the set $E_h \cap N_n^i$, we use Markov inequality and we obtain, $\forall r > 1$,
\begin{equation}
\mathbb{P}(E_h \cap N_n^i) \le c \mathbb{E}[|\tilde{X}_h^J|^r 1_{N_n^i}]h^{- \beta r} \le ch^{- \beta r}h^{1 + \beta(r - \alpha)}= ch^{1 - \beta \alpha},
\label{eq: Eh con Nin} 
\end{equation}
where in the last inequality we used \eqref{eq: stima finale con Nin}. \\
Using \eqref{eq: eh splittato}, \eqref{eq: Eh con Nin^c } and \eqref{eq: Eh con Nin} we get the Lemma \ref{lemma: P(E)}.
\finpv

\end{document}